\documentclass[12pt]{article}
\usepackage{amsmath}
\usepackage{amsthm}
\usepackage{amsfonts}
\usepackage{enumitem} 
\usepackage{natbib}
\setcitestyle{authoryear}
\usepackage{url} 
\usepackage{dsfont}
\usepackage{comment} 

\usepackage[ruled]{algorithm2e} 
\usepackage{breqn} 
\usepackage{scalerel, stackengine} 
\usepackage{blkarray} 

\usepackage{graphicx}
\usepackage{epsfig} 
\usepackage{pict2e}
\usepackage{tikz}
\usetikzlibrary{patterns}
\usetikzlibrary{graphs}
\usetikzlibrary{shapes,arrows}
\usepackage[justification=centering]{subfig}
\usepackage{caption}

\newtheorem{prop}{Proposition}
\newtheorem{ex}{Example}
\newtheorem{theo}{Theorem}

\newtheorem{lem}{Lemma}

\theoremstyle{definition}
\newtheorem{rem}{Remark}

\newcommand{\N}{\mathbb{N}}
\newcommand{\R}{\mathbb{R}}
\newcommand{\RR}{{\rm I}\kern-0.18em{\rm R}}
\newcommand{\h}{{\rm I}\kern-0.18em{\rm H}}
\newcommand{\K}{{\rm I}\kern-0.18em{\rm K}}

\newcommand{\indic}{\mathds{1}} 

\newcommand{\boldx}{\bold x} 
\newcommand{\boldy}{\bold y} 
\newcommand{\boldA}{\bold A} 
\newcommand{\boldX}{\bold X} 
\newcommand{\boldZ}{\bold Z} 
\newcommand{\boldM}{\bold M} 
\newcommand{\boldN}{\bold N} 
\newcommand{\boldT}{\bold T} 
\newcommand{\boldP}{\bold P} 
\newcommand{\boldu}{\bold u} 
\newcommand{\boldv}{\bold v} 
\newcommand{\boldp}{\bold p} 
\newcommand{\boldzeta}{\boldsymbol \zeta} 
\newcommand{\bolde}{\bold e} 
\newcommand{\boldo}{\bold 0} 
\newcommand{\boldone}{\bold 1} 
\newcommand{\boldTheta}{\bold \Theta} 
\DeclareMathOperator{\Diag}{Diag} 

\renewcommand{\P}{\mathbb{P}} 
\newcommand{\E}{\mathbb{E}} 
\newcommand{\B}{\mathcal{B}}

\numberwithin{equation}{section}

\newtheorem{assumpHRV}{Assumption}

\newcommand{\blind}{1}

\begin{document}


\def\spacingset#1{\renewcommand{\baselinestretch}%
{#1}\small\normalsize} \spacingset{1}

\if1\blind
{
  \title{\bf Multivariate sparse clustering for extremes}
  \author{Nicolas Meyer\thanks{nicolas.meyer@umontpellier.fr (corresponding author)}\hspace{.2cm}\\
    IMAG, Univ. Montpellier, CNRS, Montpellier, France\\
    LEMON, Inria, Montpellier, France\\
    and \\
    Olivier Wintenberger\thanks{olivier.wintenberger@sorbonne-universite.fr}\\
    Sorbonne Universit\'e, LPSM, F-75005, Paris, France\\
	Wolfgang Pauli Institut, c/o Fakultät für Mathematik,\\
	Universität Wien, 1090 Vienna, Austria }
  \maketitle
} \fi

\if0\blind
{
  \bigskip
  \bigskip
  \bigskip
  \begin{center}
    {\LARGE\bf Multivariate sparse clustering for extremes}
\end{center}
  \medskip
} \fi

\bigskip

\begin{abstract}
Identifying directions where extreme events occur is a major challenge in multivariate extreme value analysis. In this paper, we use the concept of sparse regular variation introduced by \cite{meyer_wintenberger} to infer the tail dependence of a random vector $\boldX$. This approach relies on the Euclidean projection onto the simplex which better exhibits the sparsity structure of the tail of $\boldX$ than the standard methods. Our procedure based on a rigorous methodology aims at capturing clusters of extremal coordinates of $\boldX$. It also includes the identification of the threshold above which the values taken by $\boldX$ are considered as extreme. We provide an efficient and scalable algorithm called MUSCLE and apply it on numerical examples to highlight the relevance of our findings. Finally we illustrate our approach with financial return data.
\end{abstract}

\noindent%
{\it Keywords:}  Euclidean projection onto the simplex, model selection, multivariate extremes, regular variation

\vfill

\spacingset{1.5} 

\section{Introduction}\label{sec:intro}

The aim of this article is to study the tail dependence of a random vector $\boldX \in \R^d_+$ with continuous marginals. In this context it is customary to assume that $\boldX$ is regularly varying (see e.g. \cite{resnick_87}, \cite{resnick_07}, \cite{hult_lindskog_06}), i.e. that there exist $a_n \to \infty$ and a non-zero Radon measure $\mu$ on the Borel $\sigma$-field of $\R^d_+ \setminus \{\boldo\}$ such that
\begin{equation} \label{eq:conv_reg_var_tail_measure}
n \P( a_n^{-1} \boldX \in \cdot) \overset{v}{\to} \mu(\cdot)\, , \quad n \to \infty\,,
\end{equation}
where $\overset{v}{\to}$ denotes the vague convergence in the space of nonnegative Radon measures on $[0,\infty]^d \setminus \{\boldo\}$.
The limit measure $\mu$ is called the \emph{tail measure} of the regularly varying vector $\boldX$. It satisfies the homogeneity property $\mu(tB) = t^{-\alpha} \mu(B)$, for any set $B$ in $\R^d_+ \setminus \{\boldo\}$ and any $t> 0$. The parameter $\alpha$ is called the \emph{tail index} of $\boldX$. It highlights the intensity of the extremes. The smaller this index is, the heaviest the tail of $\boldX$ is likely to be.

It is often more convenient to decompose the former convergence into a radial and an angular part (see for instance \cite{beirlant_et_al}, Section 8.2.3): the regular variation property is equivalent to the convergence
\begin{equation}\label{eq:intro_regular_variation_spectral_vector}
\P \big( ( |\boldX| / t,  \boldX / |\boldX|) \in \cdot \mid |\boldX| > t \big) \stackrel{w}{\to} \P( (Y, \boldTheta) \in \cdot )\, , \quad t \to \infty\, ,
\end{equation}
where $\stackrel{w}{\to}$ denotes weak convergence, and where $\boldTheta$ is a random vector on the positive unit sphere $\{\boldx \in [0,\infty)^d : |\boldx| \,= 1\}$ independent of the random variable $Y$ which satisfies $\P(Y > y) = y^{-\alpha}$, $y > 1$. The random vector $\boldTheta$ is called the spectral vector and its distribution $\P( \boldTheta \in \cdot)$ the spectral measure. Its support indicates the directions supported by large events. The subspaces of the positive unit sphere on which the spectral vector puts mass correspond to the directions where large events are likely to appear. Note that the choice of the norm $| \cdot |$ in Equation \eqref{eq:intro_regular_variation_spectral_vector} is arbitrary. In this article we choose the $\ell^1$-norm and thus focus on the simplex $\mathbb S^{d-1}_+ := \{\boldx \in [0,\infty)^d : x_1 + \cdots + x_d = 1\}$.

In order to study the support of the spectral measure we partition the simplex in terms of the nullity of some coordinates (\cite{chautru_2015}, \cite{goix_sabourin_clemencon_17}, \cite{simpson_et_al}). For $\beta \subset \{1, \ldots, d\}$ the subspace $C_\beta$ is defined as
\begin{equation} \label{def:cones_C}
C_\beta = \big \{ \boldx \in \mathbb S^{d-1}_+ : x_i > 0 \text{ for } i \in \beta, \, x_i = 0 \text{ for } i \notin \beta \big\} \, .
\end{equation}
This partition highlights the extremal structure of $\boldX$. For a given $\beta \subset \{1, \ldots, d\}$ the inequality $\P(\boldTheta \in C_\beta) > 0$ implies that the marginals $X_j$, $j \in \beta$, are likely to take simultaneously large values while the ones for $j \in \beta^c$ are of smaller order. Hence the identification of clusters of directions $\beta$ which concentrate the mass of the spectral measure brings out groups of coordinates which can be large together.

Highlighting such groups is at the core of several recent papers on multivariate extremes, all of them relying on some hyperparameters (\cite{chiapino_sabourin_2016}, \cite{goix_sabourin_clemencon_17}, \cite{chiapino_sabourin_segers_2019}, \cite{simpson_et_al}). This approach faces a crucial issue, namely the difference of support between $\boldTheta$ and $\boldX/|\boldX|$. Indeed, the spectral measure is likely to place mass on low-dimensional subspaces $C_\beta$, $\beta \neq \{1, \ldots, d\}$. We say that this measure is \emph{sparse} when the number of coordinates in the associated clusters $\beta$ is small. Conversely, the distribution of the self-normalized vector $\boldX/|\boldX|$ only concentrates on the subset $C_{ \{1, \ldots, d\} }$ since $\boldX$ has continuous marginals.

All the existing approaches proposed in the literature rely on nonstandard regular variation for which $\alpha = 1$ and all marginals are tail equivalent, possibly after a standardization. However, sparsity arises all the more for standard regular variation \eqref{eq:intro_regular_variation_spectral_vector}. In this case, it is possible that the marginals of $\boldX$ are not tail equivalent so that the support of the spectral measure is included in $\mathbb S^{r-1}_+$ for $r \ll d$. This is the approach we use in this article. For a comparison of standard and nonstandard regular variation we refer to \cite{resnick_07}, Section 6.5.6.

In this article we provide a method which highlights the sparsity of the tail structure by exhibiting sparse clusters of extremal directions. By sparse clusters we mean groups of coordinates $\beta$ which contain a reduced number of directions compared to $d$. We refer to this method as \emph{sparse clustering}. The statistical procedure we propose to achieve this clustering relies on the framework of \cite{meyer_wintenberger} which allows to circumvent the estimation's issue that arises with the spectral measure. The angular component $\boldX/|\boldX|$ in \eqref{eq:intro_regular_variation_spectral_vector} is replaced by $\pi(\boldX/t)$, where $\pi$ denotes the Euclidean projection onto $\mathbb S^{d-1}_+$ (\cite{duchi_et_al_2008}, \cite{kyrillidis_et_al_2013}, \cite{condat_2016}). This substitution leads to the concept of \emph{sparse regular variation}. A random vector $\boldX$ is said to be sparsely regularly varying if
\begin{equation} \label{eq:sparse_reg_var}
\P \big( ( |\boldX|/t, \pi(\boldX/t) ) \in \cdot \mid |\boldX| > t \big) \stackrel{w}{\to} \P((Y, \boldZ) \in \cdot)\,, \quad t \rightarrow \infty\,,
\end{equation}
where $\boldZ$ is a random vector on the simplex $\mathbb S^{d-1}_+$ and $\P(Y > y) = y^{-\alpha}$, $y > 1$. \cite{meyer_wintenberger} proved that under mild assumptions both concepts of regular variation \eqref{eq:intro_regular_variation_spectral_vector} and \eqref{eq:sparse_reg_var} are equivalent (see Theorem 1 in their article). In particular, the relation $\boldZ = \pi(Y \boldTheta)$ holds.

Similarly to the existing approaches with $\boldTheta$, we are willing to capture the tail dependence of $\boldX$ via the identification of the clusters $\beta$ which satisfy $\P(\boldZ \in C_\beta) > 0$. We call such $\beta$'s the \emph{extremal clusters}. They can be identified via the study of $\pi(\boldX/t)$ since the convergence $\P(\pi(\boldX/t) \in C_\beta \mid |\boldX| > t) \to \P(\boldZ \in C_\beta)$ holds for any $\beta \subset \{1, \ldots, d\}$ (see \cite{meyer_wintenberger}, Proposition 2). This encourages to consider for any $\beta$ the quantity
\begin{equation}\label{eq:estimator_T_nk}
T_{n,k}(\beta) = \sum_{j=1}^k \indic \{\pi(\boldX_{(j)} / |\boldX_{(k+1)}|) \in C_\beta\}\,,
\end{equation}
where $\boldX_1, \ldots, \boldX_n$ is a sample of iid sparsely regularly varying random vectors, $k=k_n$ is an intermediate sequence called \emph{level} which satisfies $k \to \infty$ and $k/n \to 0$, and $\boldX_{(j)}$ denotes the observation with $j$-th largest norm: $|\boldX_{(1)}| \geq \cdots \geq |\boldX_{(n)}|$.

It turns out that the number of positive $T_{n,k}(\beta)  $ often overestimates the total number of extremal clusters. We call the clusters which satisfy $T_{n,k}(\beta) > 0$ and $\P( \boldZ \in C_\beta ) = 0$ the \emph{biased clusters}. The approach we propose to reduce this bias relies on model selection. It consists in fitting a multinomial model to the data and to compare the Kullback-Leibler divergence between the data and this theoretical model. We obtain a minimization criterion based on a penalized likelihood similarly to Akaike's criterion (\cite{akaike_1973}). This approach provides a way to select the appropriate number of extremal clusters for a given level $k$. This is the first step of our procedure, which we call the \emph{bias selection}.

The second step then consists in extending the procedure in order to automatically select an appropriate level $k$. We call this step the \emph{level selection}. Several authors have pointed out that choosing a reasonable level, or equivalently a reasonable threshold above which the data are considered as extreme, is a challenging task in practice. This issue is tackled in a few articles (\cite{starica_1999}, \cite{abdous_ghoudi_2005}, \cite{kiriliouk_et_al_2019}, \cite{wan_davis_2019}, see also the review on marginals threshold selection by \cite{caeiro_gomes_2015}). It turns out that in the sparse regular variation framework the choice of such a level and the identification of the extremal clusters are closely related. Therefore our approach consists in extending the bias selection by including $k$ as a parameter to tune. Since Akaike's procedure only holds for a constant sample size we have to adapt the standard approach to an extreme setting where the number of extremes varies. Therefore we include the non-extreme values in the model and separate the data into an \emph{extreme} group and a \emph{non-extreme} one. The procedure then provides a level $k$ for which this separation is reasonable. To the best of our knowledge, our work is the first one which simultaneously tackles this issue with the study of tail dependence.

\paragraph{Outline of the paper}
The paper is organized as follows. Section \ref{sec:theoretical_background} introduces the theoretical background on sparse regular variation and level selection that is needed throughout the paper. In Section \ref{sec:asymptotic_results} we introduce the statistical framework of our method and establish asymptotic results for the estimators of the probabilities $\P(\boldZ \in C_\beta)$. Section \ref{sec:methodology} details the methodology of our approach. We develop the two steps of the model selection, the bias selection and the level selection.  In Section \ref{sec:numerical_results} we illustrate our findings on numerical results and compare our approach with the existing procedures proposed by \cite{goix_sabourin_clemencon_17} and \cite{simpson_et_al}.
Finally we illustrate our approach on financial data  in Section \ref{sec:applications}. The proofs are given in the Supplementary Material.

\section{Preliminaries} \label{sec:theoretical_background}

\subsection{Notation}
Symbols in bold such as $\boldx \in \R^d$ are column vectors with components denoted by $x_j$, $j \in \{1, \ldots, d\}$. Operations and relationships involving such vectors are meant componentwise. If $\boldx = (x_1, \ldots,x_d)^\top \in \R^d$, then $\Diag(\boldx)$ or $\Diag(x_1, \ldots,x_d)$ denotes the diagonal matrix whose diagonal is $\boldx$. We denote by $Id_s$ the identity matrix of $\R^s$. We define $\R^d_+ := \{\boldx \in \R^d: \, x_1 \geq 0, \ldots, x_d \geq 0\}$, $\boldo := (0,\ldots,0)^\top \in \R^d$, and $\boldone := (1,\ldots,1)^\top \in \R^d$. For $j = 1,\ldots, d$, $\bolde_j$ denotes the $j$-th vector of the canonical basis of $\R^d$. In all the paper we denote the $\ell^1$-norm by $|\cdot|$. For $d \geq 1$ we denote by ${\cal P}_d$ the power set of $\{1, \ldots, d\}$ and by ${\cal P}_d^*$ the set ${\cal P}_d \setminus \{\emptyset\}$. If $\beta \in \mathcal{P}_d$ we denote by $|\beta|$ the number of coordinates in $\beta$.

\subsection{Sparse regular variation} \label{subsec:sparse_reg_var}

We consider a sparsely regularly varying random vector $\boldX \in \R^d_+$ as defined in \eqref{eq:sparse_reg_var} and focus on its angular component $\pi(\boldX/t)$:
\begin{equation}\label{eq:conv_vector_Z_recall}
\P \left( \pi(\boldX / t) \in \cdot \mid |\boldX| > t \right) \overset{w}{\to} \P( \boldZ \in \cdot )\, , \quad t \to \infty\,.
\end{equation}
The orthogonal projection on the simplex enjoys many sparsity properties which justifies its use to study high-dimensional data. The vector $\pi(\boldX/t)$ may put mass in every subspace $C_\beta$ even if $\boldX$ is almost surely positive. This is a key difference with the self-normalized vector $\boldX/|\boldX|$ which shares the same sparsity properties as $\boldX$, and therefore always concentrates on the interior $C_{\{1,\ldots,d\}}$ of the simplex.

\begin{rem}\label{rem:projection_ball}
	Our statistical methodology exhibits the choice of a level $k$ which corresponds to the number of vectors among a sample $\boldX_1, \ldots, \boldX_n$ which are considered as extreme. This is achieved by studying also the $n-k$ non-extreme vectors. In terms of the convergence \eqref{eq:conv_vector_Z_recall} the latter vectors correspond to vectors whose norm is below the threshold $t=|\boldX_{(k+1)}|$. In order to propose a consistent methodology based on these non-extreme vectors we need to slightly modify the projection and to consider $\pi$ as the Euclidean projection onto the unit positive $\ell^1$-ball $\mathcal{B}_+^d = \{\boldx \in \R^d_+ : x_1 + \ldots + x_d \leq 1\}$. It does not change the theory of sparse regular variation since projecting onto the sphere or the ball is equivalent for vectors with norm larger than $1$. The only difference is that a vector $\boldv$ such that $|\boldv| < 1$ now satisfies $\pi(\boldv) = \boldv$.
\end{rem}

Our aim is to infer the distribution of the angular vector $\boldZ$ in order to identify the extremal directions of $\boldX$. This is achieved by focusing on the probabilities $p^*(\beta) := \P(\boldZ \in C_\beta)$ for $\beta \in {\cal P}_d^*$. We define the set of extremal clusters
\begin{equation} \label{def:positive_cones_S_Z}
{\cal S}^*(\boldZ) := \{\beta : p^*(\beta) > 0\}\,,
\end{equation}
and denote by $s^*$ its cardinality. \cite{meyer_wintenberger} proved that for any $\beta$ we have the convergence
\begin{equation} \label{eq:cv_projection_Z_on_C_beta}
\P(\pi(\boldX/t) \in C_\beta \mid |\boldX| > t ) \to p^*(\beta)\, , \quad t \to \infty\, .
\end{equation}
This convergence allows one to study the behavior of $\boldZ$ on the subsets $C_\beta$ via the one of $\pi(\boldX/t)$. The aim of this paper is to build a statistical procedure to identify the extremal clusters $\beta \in {\cal S}^*(\boldZ)$.

\begin{ex}[Discrete spectral measure]\label{ex:asympt_indep}
	For $\beta \in {\cal P}_d^*$, we denote by $\bolde(\beta)$ the sum $\sum_{j \in \beta} \bolde_j$ so that the vector $\bolde(\beta)/|\beta|$ belongs to the simplex $\mathbb S^{d-1}_+$ (recall that $|\beta|$ corresponds to the length of the cluster $\beta$). We consider the following family of discrete distributions on the simplex:
	\begin{equation} \label{eq:class_discrete_distributions}
	\sum_{ \beta \in {\cal P}_d^* } c(\beta) \, \delta_{\bolde(\beta)/|\beta|}\, ,
	\end{equation}
	where $(c(\beta))_\beta$ is a probability vector on $\R^{2^d-1}$ (see \cite{segers_12}, Example 3.3). \cite{meyer_wintenberger} proved that in this case we have $\boldZ = \boldTheta$ a.s. and that the family of distribution in \eqref{eq:class_discrete_distributions} is the only possible discrete distributions for $\boldZ$. For this type of distributions we have ${\cal S}^*(\boldZ) = \{\beta : c(\beta) > 0\}$.
	
	If we choose $c(\beta) = 0$ for all $\beta$'s except the ones of length $1$ then the spectral measure becomes $\sum_{j = 1}^d \,c_j\, \delta_{\bolde_j}$, $(c_j)_{1 \leq j \leq d}\in \mathbb S^{d-1}_+$. This corresponds to \emph{asymptotic independence} (see e.g. \cite{ledford_tawn_1996}, \cite{heffernan_tawn_2004}, \cite{dehaan_ferreira}, Section 6.2).
\end{ex}

If $\boldTheta$ places mass on a subset $C_\beta$ then so does $\boldZ$, but the converse is not true. Thus the set of clusters we identify with our method includes the usual ones on which $\boldTheta$ puts mass. However, the notion of maximal cluster (an extremal cluster which is not included in another extremal one) defined by \cite{meyer_wintenberger} coincide for $\boldTheta$ and $\boldZ$ and links both types of clusters.

\begin{ex}
	Consider a spectral measure in dimension $2$ with $\Theta_1 \sim \mathcal{U}\,(0,1)$. Then the distribution of $\boldZ$ is given by $Z_1 = \frac{1}{4} \delta_0 + \frac{1}{2} \mathcal{U}\,(0,1) + \frac{1}{4} \delta_1$, see \cite{meyer_wintenberger}, Example 1. In this case the clusters $\{1\}$ and $\{2\}$ are extremal clusters for $\boldZ$ but not for $\boldTheta$. The only maximal cluster for $\boldZ$ and $\boldTheta$ is $C_{\{1,2\}}$.
\end{ex}

\subsection{Impact of the level on the sparsity structure}

We briefly explain in this section how the choice of a threshold $t > 0$ influences the sparsity of the projected vector $\pi(\boldx/t)$ for $\boldx \in \R^d_+$. For $t> 0$, let us denote by $\pi_t$ the Euclidean projection onto the positive sphere $\{\boldx \in \R^d_+ : x_1 + \cdots + x_d = t\}$. The relation $\pi_t(\boldx) = t \pi(\boldx/t)$ implies that the sparsity structures of $\pi_t(\boldx)$ and $\pi(\boldx/t)$ are the same. The number of null coordinates of the projected vector $\pi_t(\boldx)$ strongly depends on the choice of $t$. Indeed, if $t$ is close to $|\boldx|$, then $\pi_t(\boldx)$ has only non-null coordinates (as soon as $\boldx$ itself has non-null coordinates). On the contrary, the vector $\pi_t(\boldx)$ is sparse if $t \ll |\boldx|$.

Moving on to a statistical framework, we consider a sample $\boldX_1, \ldots, \boldX_n$ of iid random vectors in $\R^d_+$. It is common in extreme value theory to define a level $k=k_n$ satisfying $k \to \infty$ and $k/n \to 0$ (see e.g. \cite{dehaan_ferreira}, \cite{beirlant_et_al}, \cite{resnick_07}). It leads to the choice of a threshold $t=u_n \to \infty$ such that
\begin{equation}\label{eq:link_un_k}
	\dfrac{n}{k} \P( |\boldX| > u_n ) \to 1\,, \quad n \to \infty\,.
\end{equation}
The level $k$ must be seen as the number of extreme vectors used for the statistical analysis. It is therefore natural to consider the $k$-largest vectors in terms of their norm, i.e. the vectors $\boldX_{(1)}, \cdots, \boldX_{(k)}$ where $|\boldX_{(j)}|$ denotes the $j$-th largest norm $|\boldX_{(1)}| \geq \cdots \geq |\boldX_{(n)}|$. Note that since we assumed the marginals of $\boldX$ to be continuous, these inequalities are strict almost surely. This encourages to work with the random threshold $|\boldX_{(k+1)}|$. By Vervaat's Lemma (see Lemma 1.0.2 in \cite{dehaan_ferreira}), the assumption \eqref{eq:link_un_k} implies that $|\boldX_{(k+1)}| / u_n$ converges to $1$ in probability as $n \to \infty$.

\begin{figure}[!th]
	\begin{center}		
		\begin{minipage}[c]{0.35\textwidth}
			\begin{tikzpicture}[scale=1.6]
				\tikzset{cross/.style={cross out, draw=black, minimum size=2*(#1-\pgflinewidth), inner sep=0pt, outer sep=0pt},
					cross/.default={2pt}}
				
				\draw (0,0) node[below left] {$0$};
				\draw (2.5,0) node[below] {$x_1$};
				\draw (0,2.5) node[left] {$x_2$};
				\draw (0,0.5) node[left] {$1$};
				\draw (0.5,0) node[below] {$1$};
				
				\draw[->] (0,0) -- (2.5,0) ;
				\draw[->] (0,0) -- (0,2.5);
				
				\draw[black] (0.9,1) node {$\times$};
				\draw[black] (0.7,1.5) node {$\times$};
				\draw[black] (0.4,1.85) node {$\times$}; 
				\draw[black] (1.1,1.5) node {$\times$};
				\draw[black] (1.2,0.95) node {$\times$};
				\draw[black] (1.6,1.2) node {$\times$};
				\draw[black] (2,0.4) node {$\times$};
				\draw[black](1.95,0.65) node {$\times$};
				\draw[black](1.3,0.4) node {$\times$}; 
				\draw[black] (1,0.5) node {$\times$};
				\draw[black] (0.65,0.8) node {$\times$};
				\draw[black](0.42,0.84) node {$\times$};
				\draw[black](0.48,0.72) node {$\bullet$}; 
				
				\draw (0.5,0) -- (0,0.5);
				\draw (1.2,0) -- (0,1.2);
				\draw (1.2,0) node[below] {$u_n$};
				\draw[very thick, densely dotted] (1.2,0) -- (2.2,1) ;
				\draw[very thick, densely dotted] (0,1.2) -- (1,2.2);
				
				\fill [pattern = north west lines, pattern color = purple]
				(1.2,0)
				-- (2.2,0)
				--(2.2,1);
				
				\fill [pattern = north west lines, pattern color = purple]
				(0,1.2)
				-- (0,2.2)
				--(1,2.2);
				
				\fill [pattern = horizontal lines, pattern color = blue]
				(0,1.2)
				-- (1.2,0)
				--(2.2,1)
				--(1,2.2);
			\end{tikzpicture}
		\end{minipage}
		\hspace{3cm}
		\begin{minipage}[c]{0.35\textwidth}
			\begin{tikzpicture}[scale=1.6]
				\tikzset{cross/.style={cross out, draw=black, minimum size=2*(#1-\pgflinewidth), inner sep=0pt, outer sep=0pt},
					cross/.default={2pt}}
				
				\draw (0,0) node[below left] {$0$};
				\draw (2.5,0) node[below] {$x_1$};
				\draw (0,2.5) node[left] {$x_2$};
				\draw (0,0.5) node[left] {$1$};
				\draw (0.5,0) node[below] {$1$};
				
				\draw[->] (0,0) -- (2.5,0) ;
				\draw[->] (0,0) -- (0,2.5);
				
				\draw[black] (0.9,1) node {$\times$};
				\draw[black] (0.7,1.5) node {$\times$};
				\draw[black] (0.4,1.85) node {$\times$}; 
				\draw[black] (1.1,1.5) node {$\times$};
				\draw[black] (1.2,0.95) node {$\times$};
				\draw[black] (1.6,1.2) node {$\times$};
				\draw[black] (2,0.4) node {$\times$};
				\draw[black](1.95,0.65) node {$\times$};
				\draw[black](1.3,0.4) node {$\bullet$}; 
				
				\draw (0.5,0) -- (0,0.5);
				\draw (1.7,0) -- (0,1.7);
				\draw (1.7,0) node[below] {$u_n$};
				\draw[very thick, densely dotted] (1.7,0) -- (2.3,0.6) ;
				\draw[very thick, densely dotted] (0,1.7) -- (0.6,2.3);
				
				\fill [pattern = north west lines, pattern color = purple]
				(1.7,0)
				-- (2.3,0)
				--(2.3,0.6);
				
				\fill [pattern = north west lines, pattern color = purple]
				(0,1.7)
				-- (0,2.3)
				--(0.6,2.3);
				
				\fill [pattern = horizontal lines, pattern color = blue]
				(0,1.7)
				-- (1.7,0)
				--(2.3,0.6)
				--(0.6,2.3);
			\end{tikzpicture}
		\end{minipage}
	\end{center}
	
	\subfloat[For $k=12$ the points in the blue area are projected on the interior of the positive sphere $\{\boldy \in \R^d_+ : y_1 + \cdots + y_d = u_n\}$ while the ones in the red area are projected on the edges of this sphere.]
	{\hspace{0.45\linewidth}} \hspace{1.5cm}
	\subfloat[For $k=8$ all points in the blue area are projected on the interior of the positive sphere $\{\boldy \in \R^d_+ : y_1 + \cdots + y_d = u_n\}$.]
	{\hspace{0.45\linewidth}} 
	\caption{Influence of the level $k$ on the sparsity structure of the data. The threshold $u_n$ corresponds to the norm of the vector $\boldX_{(k+1)}$ which is represented by a bullet.\label{fig:influence_level}}
\end{figure}

A small $k$ corresponds to a large threshold $u_n$ and vice versa. In this case only a few extreme vectors are kept for the statistical analysis and they are close to the threshold. Thus, these vectors are projected on subsets $C_\beta$ with large $|\beta|$'s which means that the projected vectors are not very sparse. On the other hand, choosing a large $k$ means choosing a low threshold $u_n$ so that we move away from the extreme region. In this case the largest vectors are projected on subsets $C_\beta$ with small $|\beta|$'s, i.e. the projected vectors are sparse. We refer to Figure \ref{fig:influence_level} for an illustration of these two cases. Following these remarks we have to make a balanced choice between providing a sparse structure for the data and staying in the extreme region. 

\section{Asymptotic analysis of the extremal clusters}\label{sec:asymptotic_results}

We consider a sequence of iid sparsely regularly varying random vectors $\boldX_1, \ldots, \boldX_n$ with generic distribution $\boldX$ and angular limit vector $\boldZ$. We also consider a level $k$ satisfying $k \to \infty$ and $k/n \to 0$ and a threshold $u_n$ such that \eqref{eq:link_un_k} is satisfied. In order to identify the set ${\cal S}^*(\boldZ)$ defined in \eqref{def:positive_cones_S_Z} we provide suitable estimators for the probabilities $p^*(\beta)$, $\beta \in {\cal P}_d^*$. We define the estimators
\[
T_n(x, \beta) = \sum_{j=1}^n \indic_{\{\boldX/u_n \in A(x,\beta)\}}\,,
\quad \beta \in \mathcal{P}_d^*\,, \quad x > 0\,,
\]
where $A(x, \beta) = \{\boldy \in \R^d_+ :  x |\boldy| > 1,\, \pi(x\boldy) \in C_\beta\}$ so that the estimator $T_{n,k}(\beta)$ defined in \eqref{eq:estimator_T_nk} satisfies $T_{n,k}(\beta) = T_n(u_n / |\boldX_{(k+1)}|, \beta)$.
An empirical version of ${\cal S}^*(\boldZ)$ is then given by
\begin{equation}\label{eq:def_positive_cones}
\widehat{\cal S}_n := \big\{ \beta \in \mathcal{P}_d^* : T_{n, k}(\beta) > 0 \big\}\,.
\end{equation}
We denote by $\hat s_n$ the cardinality of this set. Finally, we define
\[
p_n(\beta) = \P(\pi(\boldX/u_n) \in C_\beta \mid |\boldX| > u_n)\,, \quad \beta \in \mathcal{P}_d^*\,.
\]
which converges to $p^*(\beta) \in \mathcal{P}_d^*$ for any $\beta$, see Equation \eqref{eq:cv_projection_Z_on_C_beta}.

\subsection{The bias between $\widehat{\cal S}_n$ and ${\cal S}^*(\boldZ)$}\label{subsec:estimation_of_S_Z}

In this section we compare the set ${\cal S}^*(\boldZ)$ with its empirical counterpart $\widehat{\cal S}_n$. We first establish the consistency of our estimator.

\begin{prop} \label{prop:consistency}
	For any $\beta \in \mathcal{P}_d^*$,
	\begin{equation} \label{eq:conv_LLN_vector_bias}
	\frac{T_{n,k}(\beta)}{k}
	= \dfrac{1}{k} \sum_{j=1}^k \indic_{\{ \pi(\boldX_{(j)} / |\boldX_{(k+1)}|) \in C_\beta\}}
	\to p^*(\beta)\, , \quad n \to \infty\,,
	\end{equation}
	in probability.
\end{prop}
The proof relies on Proposition 2.2 of \cite{dehaan_resnick_93}. It suffices to prove that $\boldZ$ does not put any mass on the boundary of $\{|\boldx| > 1\} \cap \pi^{-1}(C_\beta)$, which has already been established in Proposition 2 of \cite{meyer_wintenberger}.

Proposition \ref{prop:consistency} implies that if $p^*(\beta) = 0$, i.e. if $\boldZ$ does not place mass on the subset $C_\beta$, then $T_{n,k}(\beta) / k$ becomes smaller and smaller as $n$ increases. Actually as soon as the dimension $d$ is large a lot of $T_{n,k}(\beta)$'s are even equal to $0$ since the number of extreme vectors, that is $k$, is far below the number of clusters, that is $2^d-1$.

In order to study the bias between $\widehat{\cal S}_n$ and ${\cal S}^*(\boldZ)$ we focus on the speed of convergence of $\P(T_{n,k}(\beta)=0)$, and thus on the one of $\P(\boldX/u_n \in A(x,\beta)) = \P(\{\boldy \in \R^d_+ :  x |\boldy| > 1,\, \pi(x\boldy) \in C_\beta\})$. \cite{meyer_wintenberger} established the equivalence
\[
\pi(\boldx)\in C_\beta
\quad \text{if and only if} \quad
\forall i \in \beta^c,\, \forall j \in \beta,\, x_i \leq \frac{|\boldx_\beta|}{|\beta|} < x_j\,.
\]
In other words, all $x_j$, $j \in \beta$, should be of the same order, while the $x_i$, $i \in \beta^c$, should be of smaller order. We set $\boldx_{\beta, \, i} = \sum_{j \in \beta} (x_j - x_i)$ for any $i$ and define
\[
\mathbb C_\beta
= \{\boldx \in \R^d_+ : \min_{i \in \beta^c} \boldx_{\beta, \, i} \geq 0 \}
= \Big\{\boldx \in \R^d_+ : \sum_{j \in \beta} (x_j - \max_{i \in \beta^c} x_i)\geq 0 \Big\}\,,
\]
which forms a cone of $\R^d_+$. Studying the convergence of $\P(\boldX/u_n \in A(x,\beta))$ then boils down to studying the asymptotic behavior of $\boldX$ on the cone $\mathbb C_\beta$.
Based on the theory of hidden regular variation (HRV) by \cite{lindskog_resnick_roy_2014}, we make the following assumption on $\boldX$.

\begin{assumpHRV}\label{ass:exponent_measure}
	For every $\beta \in \mathcal{P}_d^*$ the vector $\boldX$ is regularly varying on $\R^d_+ \setminus \mathbb C_\beta$ with tail index $\alpha(\beta)$ and exponent measure $\mu_\beta$ satisfying
	\[
	\mu_\beta(\{\boldx \in \R^d_+ : \max_{i \in \beta} \boldx_{\beta, \, i} < 1,\, \min_{i \in \beta^c} \boldx_{\beta, \, i} \geq 1\}) > 0\,.
	\]
\end{assumpHRV}
This assumption allows one to deal with the asymptotic behavior of $\P(T_{n,k}(\beta) = 0) $ even when $p^*(\beta)=0$, as stated in the following lemma.

\begin{lem}\label{lem:proba_T_n_beta_zero}
	Under Assumption \ref{ass:exponent_measure} we have for every $\beta \in {\cal P}_d^*$,
	\[
	\dfrac{ \log \P(T_{n,k}(\beta) = 0) }{ -k p_n(\beta) } \to 1\,, \quad n \to \infty\,.
	\]
\end{lem}

Lemma \ref{lem:proba_T_n_beta_zero} encourages to focus on the quantity $k p_n(\beta)$ and to consider the set
\begin{equation}\label{eq:def_S_infty_k}
\mathcal{S}_\infty = \{ \beta \in {\cal P}^*_d : k p_n(\beta) \to \infty \text{ as } n \to \infty \}\,.
\end{equation}
We denote by $s_\infty$ its cardinality. This set contains ${\cal S}^*(\boldZ)$ so that we have the inequality $s^* \leq s_\infty$. Subsequently, Lemma \ref{lem:proba_T_n_beta_zero} implies that
\begin{align*}
\P\big( {\cal S}_\infty \subset \widehat{\cal S}_n \big)
= 1 - \P \big( \exists \beta \in {\cal S}_\infty, \, \beta \notin \widehat{\cal S}_n \big)
\geq 1 - \sum_{\beta \in {\cal S}_\infty} \P( T_{n,k}(\beta) = 0 ) \to 1\, ,
\end{align*}
as $n \to \infty$. This leads to the following proposition.

\begin{prop}\label{prop:inclusion_S_Z_Sn}
	Under Assumption \ref{ass:exponent_measure} the inclusions
	\[
	{\cal S}^*(\boldZ) \subset {\cal S}_\infty \subset \widehat{\mathcal{S}}_n
	\]
	hold true with probability converging to 1.
\end{prop}
These inclusions highlight the fact that the observations $T_{n,k}(\beta)$ tend to overestimate the number of clusters $\beta$ in ${\cal S}^*(\boldZ)$. They imply that we only have a ``one-side bias" composed of clusters that appear empirically but which theoretically do not contain any mass. One of the main challenge of our study is the derivation of the asymptotic properties of $T_{n,k}(\beta)$ for biased clusters $\beta \in \widehat{\cal S}_n \setminus{\cal S}^*(\boldZ)$.

By Lemma \ref{lem:proba_T_n_beta_zero} the inclusion $\widehat{\cal S}_n \subset {\cal S}_\infty$ means that we only observe faces $C_\beta$ for which $\P(T_{n,k}(\beta) = 0)$ does not decrease super-exponentially fast with $n$. We define the sets of admissible sequences $(k_n)$ by
\[
K = \{(k_n) : k_n \to \infty,\, k_n/n \to 0,\, {\cal S}_\infty =  \widehat{\cal S}_n \text{ a.s. for all $n$ large enough} \}\,.
\]
That $K$ is non empty is a strong assumption equivalent to the fact that $\widehat{\cal S}_n$ converges a.s. to  ${\cal S}_\infty$ for some sequence of levels $(k_n)$. Combining the definition of $K$ with Assumption \ref{ass:exponent_measure} we can rely on the statistics $T_{n,k}(\beta)$, $(k_n)\in K$, which are non-null sufficiently often even when $p_n(\beta) \to p^\ast(\beta)=0$, in order to quantify the bias.


\subsection{Asymptotic normality} \label{subsec:multivariate_convergence}

We now establish a convergence result for the joint distribution of $T_{n,k}(\beta)$ for $\beta \in \mathcal{S}_\infty$. This is achieved via the study of the joint distribution of $T_n(x, \beta)$ for $x \in [\frac{1}{1+\tau}, 1+\tau]$, $\tau > 0$. Having in mind the model selection proposed in Section \ref{sec:methodology} we consider for any $0 \leq s < r \leq s_\infty$ and any disjoint clusters $\beta_1, \ldots, \beta_r \in \mathcal{S}_\infty$ the vectors
\[
\boldT_n^{s,r}(x)
= \Big( T_n(x, \beta_1), \ldots, T_n(x, \beta_s), \sum_{j=s+1}^r T_n(x,\beta_j) \Big)^\top \in \R^{s+1}\,,
\]
and
\[
\boldP^{s,r}_n(x)
=
\Big(
x^{\alpha(\beta_1)} p_n(\beta_1), \ldots, x^{\alpha(\beta_s)} p_n(\beta_s), \sum_{j=s+1}^r x^{\alpha(\beta_j)} p_n(\beta_j)
\Big)^\top \in \R^{s+1}\,.
\]
For $\tau > 0$ we denote by $\ell^\infty([\frac{1}{1+\tau}, 1+\tau])$ the set of functions defined and bounded on $[\frac{1}{1+\tau}, 1+\tau]$.

\begin{theo}\label{theo:CLT}
	Let Assumption \ref{ass:exponent_measure}  hold. Assume that there exists $(k_n) \in K$ and choose $u_n$ such that $k \sim n \P(|\boldX| > u_n)$ as $n \to \infty$.
	 \begin{enumerate}
		\item The following convergence holds in $\ell^\infty([\frac{1}{1+\tau}, 1+\tau])$ as $n \to \infty$:
		\begin{align}
			\bigg\{
			\sqrt k \Diag(\boldP^{s,r}_n(x))^{-1/2} \bigg( \dfrac{ \boldT_{n}^{s,r}(x) }{k} - \E\Big[ \dfrac{ \boldT_{n}^{s,r}(x) }{k}  \Big] \bigg) ; \,
			(1+\tau)^{-1} &\leq x \leq 1+\tau
			\bigg\}_{s<r} \nonumber\\
			&\overset{d}{\to}
			(\boldN^{s,r})_{s<r}
			\,, \label{eq:CLT_Tnx}
		\end{align}
		where the constant limit process is identified to $\boldN^{s,r}$, a standard centered multivariate Gaussian vector in $\R^{s+1}$.
		\item If we assume moreover that for any $\beta \in {\cal S}_\infty$,
		\begin{equation}\label{eq:bias}
			\sup_{x \in [\frac{1}{1+\tau},1+\tau]}\sqrt{\dfrac{k}{p_n(\beta)}} \, \Big|\dfrac{n}{k} \P(\boldX/u_n \in A(x, \beta)) - x^{\alpha(\beta)} p_n(\beta) \Big| \to 0\,, \quad n \to \infty\,,
		\end{equation}
		then we have
		\begin{align}
			\bigg\{
			\sqrt k \Diag(\boldP^{s,r}_n(x))^{-1/2} \bigg( \dfrac{ \boldT_{n}^{s,r}(x) }{k} -\boldP^{s,r}_n(x) \bigg)\, ; \,
			(1+\tau)^{-1} \leq x \leq 1+\tau
			\bigg\}_{s<r}
			\overset{d}{\to}
			(\boldN^{s,r})_{s<r}
			\,, \label{eq:CLT_bias}
		\end{align}
		in $\ell^\infty([\frac{1}{1+\tau}, 1+\tau])$ as $n \to \infty$.
	\end{enumerate}
\end{theo}

Based on Theorem \ref{theo:CLT}, we establish the asymptotic behavior of the estimators $T_{n,k}(\beta)$. We define
\[
\boldT^{s,r}_{n,k}
= \boldT^{s,r}_n(u_n/|\boldX|_{(k+1)})
= \Big( T_{n,k}(\beta_1), \ldots, T_{n,k}(\beta_s), \sum_{j=s+1}^r T_{n,k}(\beta_j) \Big)^\top \in \R^{s+1}\,.
\]

\begin{prop}\label{prop:CLT_X(k)}
	Under the assumptions of Theorem \ref{theo:CLT}, under \eqref{eq:bias}, and under the bias assumption
	\begin{equation}\label{eq:bias_assumption}
		\sqrt{k} (p_n(\beta) - p^*(\beta)) \to 0\,, \quad n \to \infty\,, \quad \beta \in \mathcal{S}_\infty\,,
	\end{equation}
	we have the convergence
	\begin{equation}\label{eq:cv_T_un_Xk}
	\sqrt k \Diag(\boldP^{s,r}_n(1))^{-1/2}
	\bigg(
	 \dfrac{\boldT_{n,k}^{s,r}}{k}
	 - \boldP_n^{s,r}(1)
	 \bigg)
	 \stackrel{d}{\to}
	 (Id_{s+1} - \sqrt{\boldP^{s,r}} \cdot \sqrt{\boldP^{s,r}}^\top) \boldN\,,
	\quad n \to \infty\,,
	\end{equation}
	where $\boldN \in \R^{s+1}$ is a standard centered multivariate Gaussian vector, and where $\boldP^{s,r}$ is the limit vector of $\boldP^{s,r}_n(1)$:
	\[
	\boldP^{s,r}
	= \Big(p^*(\beta_1), \ldots, p^*(\beta_s), \sum_{j=s+1}^r p^*(\beta_j)\Big)^\top
	=
	\lim_{n \to \infty}
	\Big(p_n(\beta_1), \ldots, p_n(\beta_s), \sum_{j=s+1}^r p_n(\beta_j)\Big)^\top\,.
	\]
\end{prop}

\begin{rem}\label{rem:bias_assumption}
	The bias assumption \eqref{eq:bias_assumption} holds for $\beta \in \mathcal{S}_\infty \setminus \mathcal{S}^*(\boldZ)$ if $k = o(n^{\kappa})$ as $n \to \infty$ where $\kappa>\frac{2(\alpha(\beta) - \alpha)}{2 \alpha(\beta) - \alpha}$ for every $\beta \in \mathcal{S}_\infty \setminus \mathcal{S}^*(\boldZ)$. We refer to the Supplementary Material for a proof.
\end{rem}

\begin{rem} \label{rem:limit_chi_squared}
For $r = s_\infty$ the matrix $Id_{s+1} - \sqrt {\boldP^{s,r}} \cdot \sqrt {\boldP^{s,r}}^\top$ is symmetric and satisfies
\begin{align*}
(Id_{s+1} - \sqrt {\boldP^{s,r}} \cdot \sqrt {\boldP^{s,r}}^\top)^2
&= Id_{s+1} - 2 \sqrt {\boldP^{s,r}} \cdot \sqrt {\boldP^{s,r}}^\top + (\sqrt {\boldP^{s,r}} \cdot \sqrt {\boldP^{s,r}}^\top)^2\\
&= Id_{s+1} - \sqrt {\boldP^{s,r}} \cdot \sqrt {\boldP^{s,r}}^\top\,,
\end{align*}
since $\sqrt {\boldP^{s,r}}^\top \cdot \sqrt {\boldP^{s,r}} = \sum_{j=1}^r p^*(\beta_j) = 1$. Therefore it corresponds to an orthogonal projection with rank $s$. Cochran's theorem then ensures that the $\ell^2$-norm of the vector $(Id_{s+1} - \sqrt {\boldP^{s,r}} \sqrt {\boldP^{s,r}}^\top) \boldN$ follows a chi-squared distribution with $s$ degrees of freedom.

Going back to Proposition \ref{prop:CLT_X(k)} we obtain the following convergence:
\begin{equation}\label{eq:conv_chi_squared}
	k \sum_{j=1}^s \frac{\big( T_{n,k}(\beta_j) / k - p_n(\beta_j) \big)^2}{p_n(\beta_j)}
	+ k \frac{\big[ \sum_{j=s+1}^r (T_{n,k}(\beta_j) /k - p_n(\beta_j)) \big]^2}{\sum_{j=s+1}^r p_n(\beta_j)}
	\overset{d}{\to} \psi(s)\,,
\end{equation}
where $\psi(s)$ follows a chi-squared distribution with $s$ degrees of freedom. This convergence is useful to identify the parameter $s$ in the bias selection, see Lemma 4 in the Supplementary Material.
\end{rem}

\section{Methodology}
\label{sec:methodology}

We develop in this section our methodology to estimate the set ${\cal S}^*(\boldZ)$. We use the same notation as in Section \ref{sec:asymptotic_results}.

\subsection{Bias selection} \label{subsec:aic_model_k}

We consider the vector $\boldT_{n,k} \in \R^{2^d-1}$ with components $T_{n,k}(\beta)$ whose distribution $\boldP_k$ is multinomial with probability weights $(p_n(\beta))_{\beta \in \mathcal P_d^\ast}$, and adding up to $k$.  We propose a bias selection which consists in comparing the distribution $\boldP_k$ with the theoretical multinomial model $\boldM_k$ with $2^d-1$ outcomes adding up to $k$ and a probability vector $(p_1, \ldots, p_s, p, \ldots, p, 0, \ldots, 0)^\top \in [0,1]^{2^d-1}$, with $p_1 \geq \cdots \geq p_s > p$ and $r-s$ components $p$ satisfying $p_1 + \cdots + p_s + (r-s)p = 1$. The parameters $p_j$ model the probability that $\boldZ$ belongs to the associated subsets $C_\beta$ while the parameter $p$ models the probability that a biased cluster appears. We denote by $\boldp$ the vector $(p_1, \ldots, p_s)^\top \in \B^s_+(0,1)$. The likelihood $L_{\boldM_k}$ of the model $\boldM_k$ is given by
\begin{equation}\label{eq:log_likelihood_y}
L_{\boldM_k} (\boldp ; \boldy)
= \frac{k!}{\prod_{i=1}^{2^d-1} y_i! } \prod_{i=1}^{s} p_i^{y_i} \prod_{i=s+1}^{r} \Big(\frac{1-\sum_{j=1}^s p_j}{r-s}\Big)^{y_i} \indic_{\{ y_{r+1} = \cdots = y_{2^d-1} = 0\}} \, ,
\end{equation}
for any vector $\boldp \in \B^s_+(0,1) = \{\boldu \in \R^s_+ : u_1 + \cdots + u_s \leq 1\}$ and any $\boldy \in \N_0^{2^d-1}$ adding up to $k$, where $\N_0$ denotes the sets of non-negative integers.

The identification of the extremal clusters $\beta$ in ${\cal S}^*(\boldZ)$ is achieved by choosing the model $\boldM_k$ which best fits the sample $\boldT_{n,k}$. 
Following the AIC approach of \cite{akaike_1973}, we select the multinomial model which minimizes the expectation of the Kullback-Leibler (KL) divergence (see \cite{kullback_leibler_1951}) between the true distribution $\boldP_k$ and the model $\boldM_k$ evaluated at $\hat\boldp$, where $\hat \boldp$ denotes the maximum-likelihood estimator of $\boldp$. Hence we consider the quantity
\begin{equation}\label{eq:kullback_model_k_estimator}
	\E[KL (\boldP_k \| \boldM_k)|_{\boldp = \hat \boldp}]
	= \E \big[ \log L_{\boldP_k}(\boldT_{n,k}) \big]
	- \E[\E \big[ \log L_{\boldM_k} (\boldp ; \boldT_{n,k}) \big]|_{\boldp = \hat \boldp}]\,,
\end{equation}
where $L_{\boldP_k}$ denotes the likelihood of the distribution $\boldP_k$. Theorem \ref{theo:cv_KL_chi_square} below provides an asymptotic expansion of this quantity.

Before stating this result we compute the maximum likelihood estimator of the model $\boldM_k$. The first components of the model $\boldM_k$ being associated to the extremal clusters, we reorder the coordinates of the vector $\boldT_{n,k}$ so that its components are ordered in the decreasing order. Hence we define $T_{n,k, 1} = \max_\beta T_{n,k}(\beta)$ and
\[
T_{n,k, j} = \max\, \{T_{n,k}(\beta), \, \beta \in {\cal P}_d^*\} \setminus \{T_{n,k, 1}, \ldots, T_{n,k, j-1}\}\,, \quad j = 2, \ldots, 2^d-1\,.
\]
The expression in \eqref{eq:log_likelihood_y} is also maximal when $r$ corresponds to the number $\hat s_n$ of clusters that appear empirically. This leads to the following expression of the log-likelihood $\log L_{\boldM_k} (\boldp ; \boldT_{n,k})$:
\begin{equation}\label{eq:log_likelihood_k}
	\log(k!) - \sum_{i=1}^{2^d-1} \log(T_{n,k, i}!)
	+ \sum_{i=1}^{s} T_{n,k, i} \log(  p_i)
	+ \Big(\sum_{i=s + 1}^{r} T_{n,k, i} \Big) \log\Big(\frac{1-\sum_{j=1}^s p_j}{r-s}\Big)\,.
\end{equation}
The optimization of this quantity then provides the maximum likelihood estimator $\hat \boldp \in \R^s$ with components $\hat p_j := T_{n,k, j} / k$ for $1 \leq j \leq s$.


\begin{theo}\label{theo:cv_KL_chi_square}
	Under the assumptions of Proposition \ref{prop:CLT_X(k)} the following convergence holds:
	\begin{align*}
		\E[KL(\boldP_k \| \boldM_k )|_{\boldp = \hat \boldp}]
		- \E \big[ \log L_{\boldP_k} (\boldT_{n,k}) \big]
		+ \E[\log L_{\boldM_k} (\hat \boldp ; \boldT_{n,k})]
		\to s\,, \quad n \to \infty\,.
	\end{align*}
\end{theo}

Based on Theorem \ref{theo:cv_KL_chi_square} we choose the model $\boldM_k$ which minimizes the quantity
\begin{equation}\label{eq:estimator_kullback_k}
- \log L_{\boldM_k} (\hat \boldp ; \boldT_{n,k}) + s\,.
\end{equation}
Therefore for a given sequence $(k_n) \in K$ the bias selection procedure consists in choosing the parameter $\hat s(k)$ which minimizes this penalized log-likelihood.

\subsection{Level selection}

The second step of the model selection consists in considering $k$ as a parameter which has to be estimated and tuned. It is therefore necessary to consider all observations $\boldX_1, \ldots, \boldX_n$ and not only the extreme ones. We consider a vector $\boldT'_n \in \R^{2^d}$ such that
\[
\mathcal{L}( (T'_{n, 1}, \ldots, T'_{n, 2^d-1}) \mid T'_{n, 2^d} = n-k )
= \boldT_{n,k}\,.
\]
The last component $T'_{n, 2^d}$ corresponds to the number of non-extreme values of the sample.
We assume that this vector follows a multinomial distribution $\boldP'_n$ with parameter $n$ and probability vector $\boldp'_n = (q_n p_{n, 1}, \ldots, q_n p_{n, 2^d-1}, 1-q_n)^\top \in \R^{2^d}$.
%

Similarly to Section \ref{subsec:aic_model_k} we consider a multinomial model $\boldM'_n$ with probability vector given by $(q' p'_1, \ldots, q' p'_{s'}, q' p', \ldots, q' p', 0, \ldots, 0, 1- q')^\top \in \R^{2^d}$ with $p'_1 \geq \ldots \geq p'_{s'} > p'$ and $r'-s'$ components $q' p'$ satisfying the relation $p'_1 + \ldots + p'_s + (r'-s')p' = 1$. Here $q'$ models the proportion of extreme vectors. We denote by $\boldp'$ the vector $(p'_1, \ldots, p'_{s'}, q')^\top \in \B_+^{s'} \times (0,1)$.

We consider the Kullback-Leibler divergence between $\boldP'_n$ and $\boldM'_n$ given by
\begin{equation}\label{eq:kullback_leibler_model_n}
KL \big(\boldP'_n \| \boldM'_n \big)
= \E \bigg[ \log \bigg( \frac{L_{\boldP'_n}(\boldT'_n)}{L_{\boldM'_n} (\boldp' ;\boldT'_n)} \bigg) \bigg]
= \E \big[ \log L_{\boldP'_n}(\boldT'_n) \big]  - \E \big[ \log L_{\boldM'_n} (\boldp' ; \boldT'_n) \big] \, ,
\end{equation}
where $L_{\boldP'_n}$ (resp. $L_{\boldM'_n}$) denotes the likelihood of the distribution $\boldP'_n$ (resp. $\boldM'_n$). Following the same ideas as in Section \ref{subsec:aic_model_k} and similarly to an AIC procedure we estimate the Kullback-Leibler divergence in Equation \eqref{eq:kullback_leibler_model_n} by the estimator $KL \big(\boldP'_n \big\| \boldM'_n \big)|_{\widehat {\boldp'}}$ where $\widehat {\boldp'}$ denotes the maximum likelihood estimator of $\boldp'$.

We make the following assumptions. 
\begin{enumerate}[label={\bf (B\arabic*)}]
	\item\label{ass:bias_T} For $k \in K$ and $\beta_j \in \mathcal{S}_\infty$ we have
	\[
	\dfrac{\E[T_{n,n-T'_{n, 2^d}, j} \mid T'_{n, 2^d}]}{n - T'_{n, 2^d}}
	= \dfrac{\E[T_{n,k, j}]}{k} + O(1)\,, \quad n \to \infty\,.
	\]
	\item\label{ass:bias_k} For $n$ sufficiently large, $k \in K$, there exist $c$, $C>0$ such that $c n q_n \leq k \leq C n q_n$.
\end{enumerate}

Assumptions \ref{ass:bias_T} and \ref{ass:bias_k} allow one to control the bias between $T_{n,n-T'_{n, 2^d}, j} \mid T'_{n, 2^d}$ and $T_{n,k, j}$, and between $k$ and $nq_n$ respectively. 

The following theorem provides an asymptotic expansion of the expectation of this estimator.

\begin{theo}\label{theo:cv_KL_threshold}
	Under \ref{ass:bias_T}, \ref{ass:bias_k} and the assumptions of Proposition \ref{prop:CLT_X(k)} we have
	\[
	\E \big[ KL (\boldP'_n \| \boldM'_n )|_{\widehat {\boldp'}} \big]
	= n q_n \Big( \dfrac{\E \big[ - \log L_{\boldM_k} (\hat \boldp ; \boldT_n) \big] + s}{k}  + \log(k/n)\Big) + O(nq_n)\,, \quad n \to \infty\,.
	\]
\end{theo}

Theorem \ref{theo:cv_KL_threshold} encourages to choose a level $k$ which minimizes the penalized log-likelihood
\[
\dfrac{- \log L_{\boldM_k} (\hat \boldp ; \boldT_n) + s}{k} +\log\Big(\dfrac{k}{n}\Big)\,.
\]
It turns out that the additive penalization term $ \log(k/n)$ leads to numerical instability as $k/n$ is small. To cope with this issue we upper bound it by $k/n-1$. The level $k$ which minimizes the criterion  is then smaller than the one that appears with $\log(k/n)$. Thus it satisfies more likely  the bias assumptions  \ref{ass:bias_T}, \ref{ass:bias_k}. So in practice we choose a level $k$ which minimizes the following penalized log-likelihood
\begin{equation}\label{eq:minimize_KL_practice}
	\dfrac{- \log L_{\boldM_k} (\hat \boldp ; \boldT_n) + s}{k} + \dfrac{k}{n} \,.
\end{equation}

Note that the two steps of our procedure are clearly identified in the penalized log-likelihood. The term $- \log L_{\boldM_k} (\hat \boldp ; \boldT_n) + s$ corresponds to the bias selection and the multiplicative factor and the additional term to the level one.

\subsection{Algorithm: MUltivariate Sparse CLustering for Extremes (MUSCLE)}\label{subsec:our_approach}

In practice we choose a large range $\cal K$ of $k$ (often between $0.5\%$ and $10\%$ of $n$) and we compute the value of \eqref{eq:minimize_KL_practice} for these $k$ and for $s = 1, \ldots, \hat s_n$, where $\hat s_n$ depends on the chosen level $k$. We choose $\hat k$ which minimizes the penalized log-likelihood \eqref{eq:minimize_KL_practice} and then choose $\hat s(\hat k)$ which minimizes \eqref{eq:minimize_KL_practice} for $k= \hat k$. Then we define $\widehat{\mathcal {S}^*}$ as the set gathering the $\hat s(\hat k)$ clusters corresponding to the largest $T_{n, \hat k}(\beta)$'s. Finally we consider the probability vector $\hat \boldzeta$ defined by
\[
\hat \zeta (\beta) := \frac{T_{n, \hat k}(\beta)}{\sum_{\gamma \in \widehat {\cal S}^*} T_{n, \hat k}(\gamma)}\,,
\]
for $\beta \in \widehat{\mathcal {S}^*}$ and $0$ elsewhere, as an estimator of $\boldp^*$. Our procedure entails the following parameter-free algorithm called \emph{MUSCLE} for MUltivariate Sparse CLustering for Extremes.

\begin{rem}\label{rem:plot_s0_k}
	While our procedure leads to the choice of a unique $\hat k$, we expect that this approach is not too sensitive to this choice. Therefore, it is relevant to plot the function $k \mapsto \hat s(k)$ which provides the chosen value of $s$ for every $k \in \cal K$. We expect that this function is approximately constant around the chosen value $\hat k$.
\end{rem}

\begin{algorithm}[!h]
	\SetAlgoLined
	\KwData{A sample $\boldX_1, \ldots, \boldX_n \in \R^d_+$ and a range of values $\cal K$ for the level}
	\KwResult{A list $\widehat {\cal S}^*$ of clusters $\beta$ and the associated probability vector $\hat \boldzeta$.}
	\For{$k \in \mathcal{K}$}{
		Compute $u_n = |\boldX|_{(k+1)}$ the $(k+1)$-th largest norm\;
		Assign each $\pi(\boldX_j/u_n)$ the subsets $C_\beta$ it belongs to\;
		Compute $T_{n,k}(\beta)$ for each $\beta \in \mathcal{P}_d^*$\;
		Compute the minimizer $\hat s(k)$ which minimizes the criterion given in Equation \eqref{eq:estimator_kullback_k}\;
	}
	Choose $\hat k$ which minimizes \eqref{eq:minimize_KL_practice} plugging in the minimal value in \eqref{eq:estimator_kullback_k}\;
	\textbf{Output:} $\widehat {\cal S}^* = \{ \text{the clusters $\beta$ associated to the } T_{n, \hat k, 1}, \ldots, T_{n, \hat k,  \hat s(\hat k)}\}$ and $\hat \boldzeta$ as above.
	\caption{MUltivariate Sparse CLustering for Extremes (MUSCLE)}
	\label{algo:tail_inference_sparse_rv}
\end{algorithm}

\section{Numerical results} \label{sec:numerical_results}

\subsection{Overview}

The aim of the numerical results is to compare the extremal clusters given by MUSCLE with the theoretical ones in ${\cal S}^*(\boldZ)$. To this end we compare the estimated probability vector $\hat \boldzeta$ with the theoretical one $\boldp^*$ via the Hellinger distance
\begin{equation}\label{eq:hell_distance}
h(\boldp^*, \hat \boldzeta) = \dfrac{1}{\sqrt 2} \Big[\sum_{ \beta \in {\cal P}_d^* } \big(p^*(\beta)^{1/2} - \hat \zeta(\beta)^{1/2} \big)^2 \Big]^{1/2}\, .
\end{equation}
The closer $h(\boldp^*, \hat \boldzeta)$ is to $0$, the better $\hat \boldzeta$ estimates $\boldp^*$. In order to compare our method with some existing ones, we also compute the Hellinger distance between the true probabilities $\P(\boldTheta \in C_\beta)$ and the estimated ones given by the algorithm called DAMEX of \cite{goix_sabourin_clemencon_17} and the two methods of \cite{simpson_et_al}. We represent the mean Hellinger distance over $N=100$ simulations. The parameters in the method of \cite{goix_sabourin_clemencon_17} are chosen to be $\epsilon = 0.1$, $k = \sqrt{n}$, and $p=0.1$, see the notation in their paper. Regarding the methods of \cite{simpson_et_al} we use the parameters given by the authors in Section 4.2 of their paper, i.e. we set $\pi=0.01$, and $p=0.5$ and $u_\beta$ to be the $0.75$ quantile of observed $Q$ values in region $C_\beta$ for the first method, and $\delta = 0.5$ and $u_\beta$ to be the $0.85$ quantile of observed $Q$ values in region $C_\beta$ for the first method. We refer to \cite{simpson_et_al} for more insights on these parameters.

\begin{rem}
	Contrary to the aforementioned methods, we recall that MUSCLE does not require any hyperparameter. This is a main advantage from a statistical and computational point of view. On the contrary, for the other methods these values could be tuned via cross-validation. For the numerical results we do not choose this approach and keep the values fixed by the authors of the cited papers.
\end{rem}

In the following section we develop the example of a max-mixture distribution. The code related to this article can be found at \url{https://drive.google.com/drive/folders/11TvhbVMPXcSkxmdnnAySvZt64lpMKZqL?usp=sharing}. Another example related to asymptotic independence is given in the Supplementary Material.

\subsection{Max-mixture distribution}

For any $\beta \in {\cal P}_d^*$, let $\boldA_\beta \in \R_+^{|\beta|}$ be a random vector with standard Fr\'echet marginal distributions and with dependence structure given below, and let $\{\boldA_\beta : \beta \in {\cal P}_d^*\}$ be independent random vectors. Then the vector $\boldX = (X_1, \ldots, X_d)^\top$ whose components are defined via $X_j = \max_{\beta \in {\cal P}_d^* : i \in \beta} \lambda_{i, \, \beta} X_{j\, , \beta}$, with $ \lambda_{i, \, \beta} \in [0,1]$ and $ \sum_{\beta \in {\cal P}_d^* : i \in \beta} \lambda_{i, \, \beta} = 1$, has also standard Fr\'echet marginal distributions and is regularly varying.

For our simulations we consider the five-dimensional example introduced by \cite{simpson_et_al} which we recall for completeness. We consider two bivariate Gaussian copulas with correlation parameter $\rho$ and Fr\'echet marginals $\boldA_{\{1,2\}}$ and $\boldA_{\{4,5\}}$, and three extreme-value logistic copulas with dependence parameter $\alpha$ and Fr\'echet marginals $\boldA_{\{1,2,3\}}$, $\boldA_{\{3,4,5\}}$, and $\boldA_{\{1,2,3,4,5\}}$. For $\rho <1$, the Gaussian copula is asymptotically independent (see Example \ref{ex:asympt_indep} for more insights on this notion) and thus the spectral measure defined in \eqref{eq:intro_regular_variation_spectral_vector} concentrates on the subsets $C_{\{1\}}$, $C_{\{2\}}$, $C_{\{4\}}$, and $C_{\{5\}}$. For $\alpha \in (0,1)$ the logistic distribution is asymptotically dependent so that the spectral measure also places mass on the subsets $C_{\{1,2,3\}}$, $C_{\{3,4,5\}}$, and $C_{\{1,2,3,4,5\}}$. Following \cite{simpson_et_al}, we set
\begin{align*}
&\lambda_{\{1,2\}} = (5,5)/7\,, \quad  \lambda_{\{4,5\}} = (5,5)/7\\
\lambda_{\{1,2,3\}} = (1,1,3)/7&\,, \quad  \lambda_{\{3,4,5\}} = (3,1,1)/7\,, \quad
\lambda_{\{1,2,3,4,5\}} = (1,1,1,1,1)/7\,,
\end{align*}
so that equal mass is assigned to each of the seven aforementioned subsets. In order to compute the mass the distribution of $\boldZ$ assigns to every subset $C_\beta$ we start from the distribution of $\boldTheta$ and use Monte-Carlo simulation. We then compare these probabilities with their estimated ones $\hat \boldzeta$ given by MUSCLE.

We run our algorithm for different values of $\rho \in \{0, 0.25, 0.5, 0.75\}$ and $\alpha \in \{0.1, 0.2, \ldots, 0.9\}$. Figures \ref{fig:boxplots_max_mixture_1} and \ref{fig:boxplots_max_mixture_2} shows the average mean Hellinger distance for our method, the one of \cite{goix_sabourin_clemencon_17}, and the two of \cite{simpson_et_al} over $100$ simulations. Our method provides a mean Hellinger distance which stabilizes between $0.2$ and $0.3$ for all values of $\rho$ and $\alpha$. For $\alpha \leq 0.7$ the distance slightly decreases with alpha, while it increases for $\alpha \geq 0.8$. The standard deviation is quite small for $\alpha \leq 0.7$ and then increases with $\alpha$. Regarding the approach of \cite{goix_sabourin_clemencon_17}, the mean Hellinger distance tends to increase with $\alpha$ and with $\rho$.  The smallest values is obtained for $\rho \in \{0, 0.25\}$ and for small $\alpha$. The estimation particularly deteriorates for $\rho =0.75$. Finally both methods proposed by \cite{simpson_et_al} provide a mean Hellinger distance which increases with $\alpha$ and $\rho$. The second one seems to provide almost always better results than the first one.

\begin{figure}
	\centering
	\includegraphics[scale=0.13]{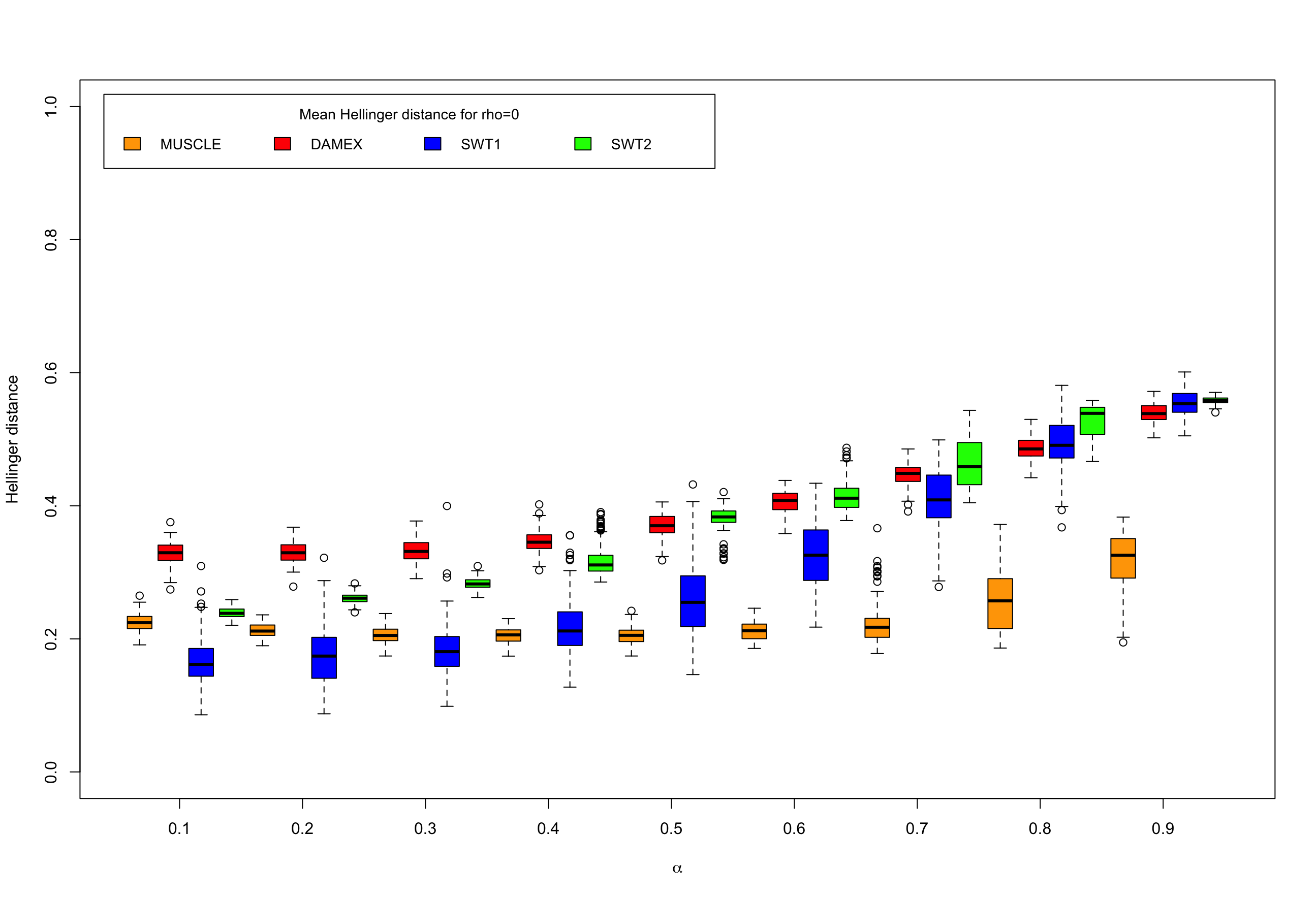}
	\includegraphics[scale=0.13]{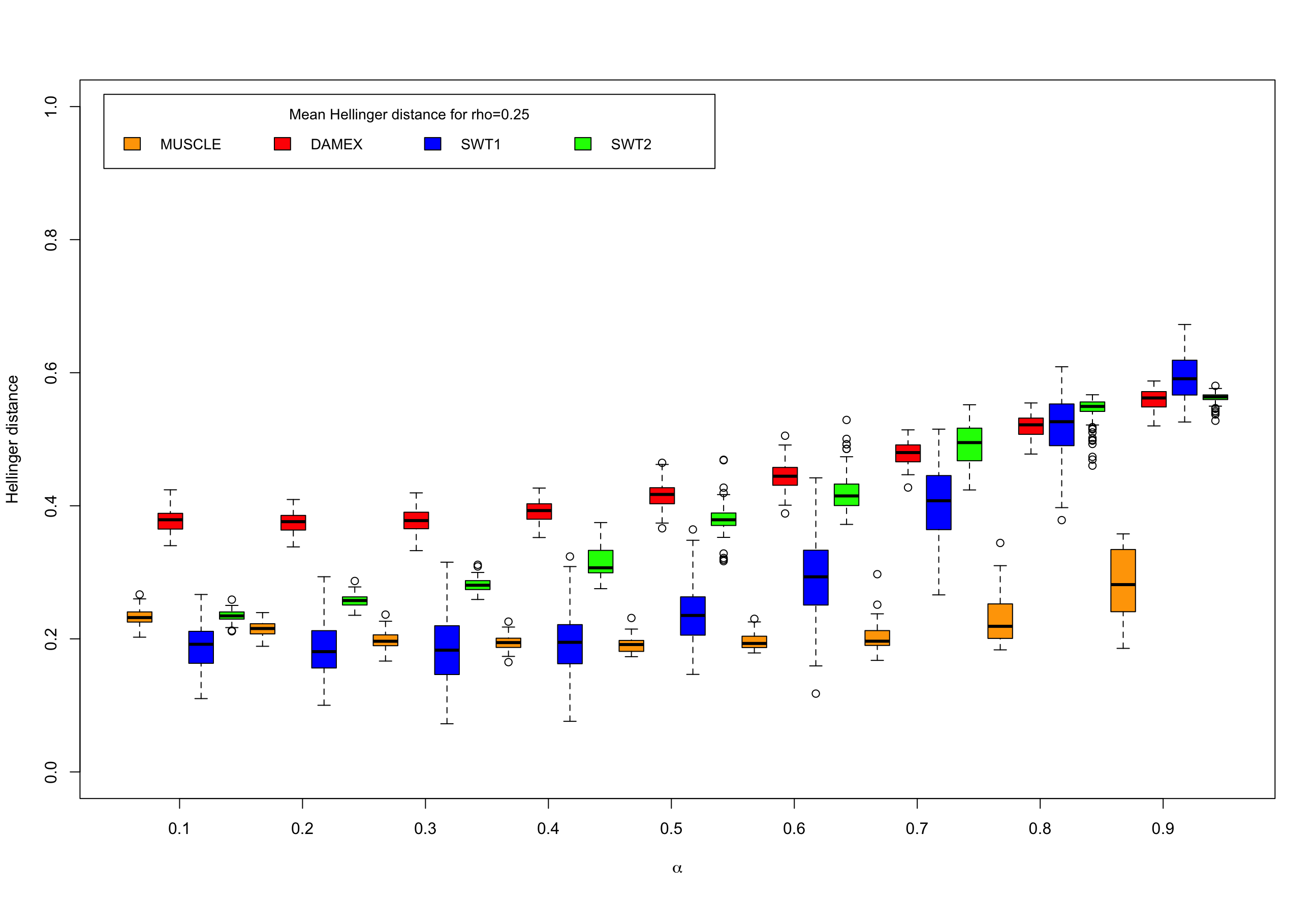}
	\captionof{figure}{Mean Hellinger distance over $100$ simulations for $\rho =0$ (top) and $\rho = 0.25$ (bottom). The abbreviation SWT1 (resp. SWT2) refers to the first (resp. second) method of \cite{simpson_et_al}. \label{fig:boxplots_max_mixture_1}}
\end{figure}

\begin{figure}
	\centering
	\includegraphics[scale=0.13]{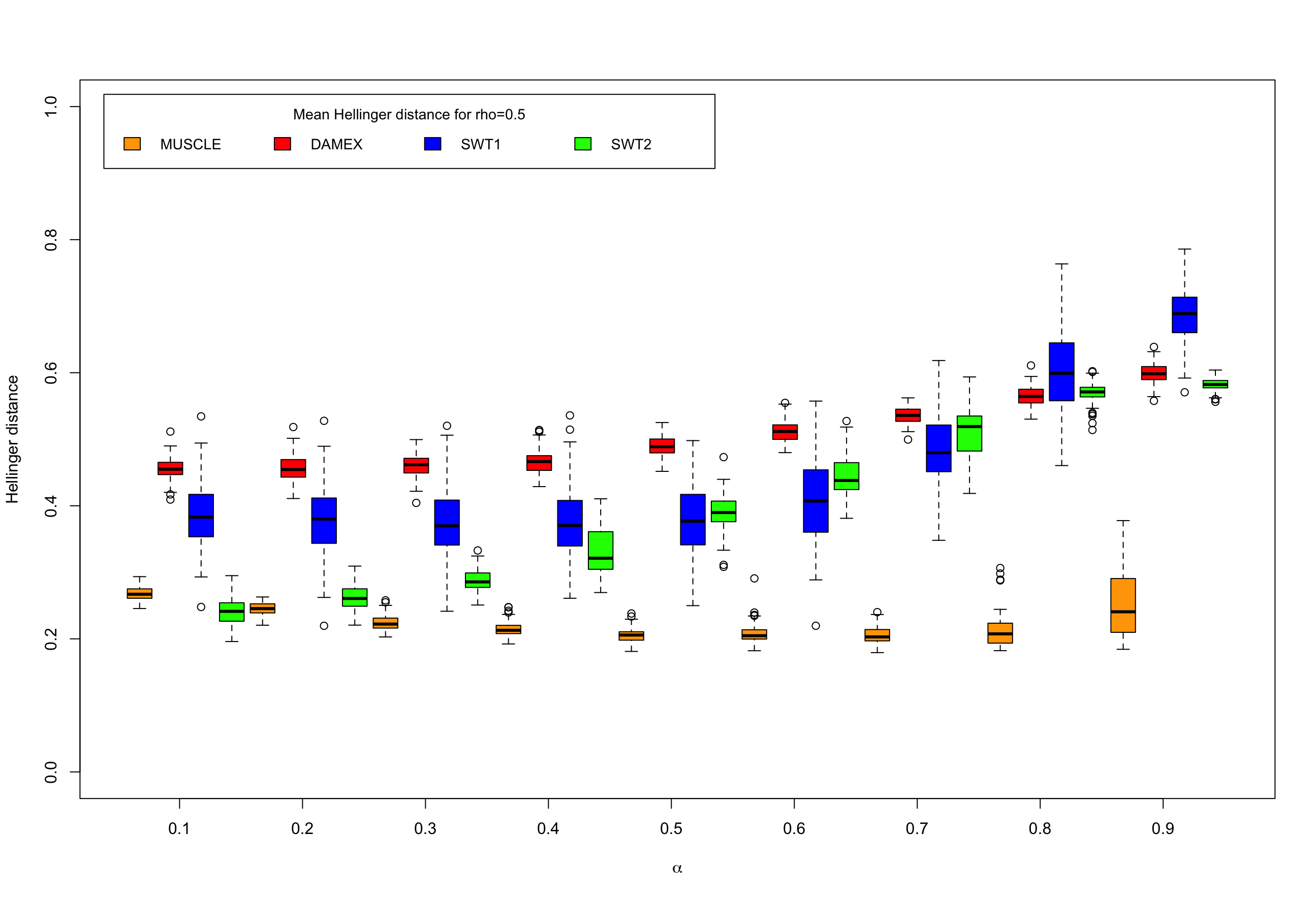}
	\includegraphics[scale=0.13]{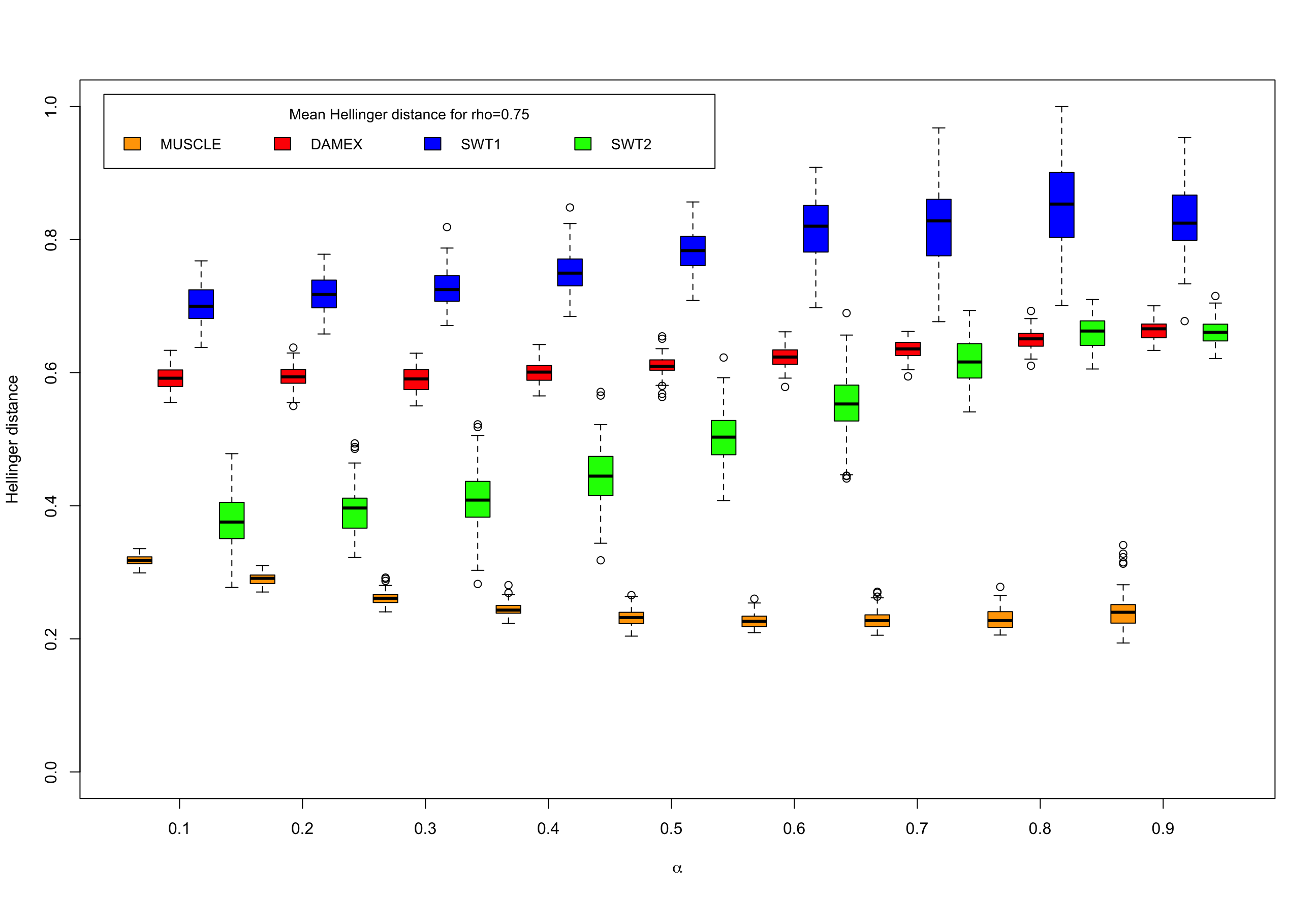}
	\captionof{figure}{Mean Hellinger distance over $100$ simulations for $\rho = 0.5$ (top) and $\rho = 0.75$ (bottom). The abbreviation SWT1 (resp. SWT2) refers to the first (resp. second) method of \cite{simpson_et_al}. \label{fig:boxplots_max_mixture_2}}
\end{figure}

While all methods provided by \cite{goix_sabourin_clemencon_17} and \cite{simpson_et_al} deteriorate when $\rho$ or $\alpha$ increase, our procedure provides results which stabilize around a mean Hellinger distance of $0.2$. This distance is the smallest one for $\rho = 0.5$ and $\rho = 0.75$ for all $\alpha$ compared to the one of the three other methods. For small $\rho$, MUSCLE better performs for large $\alpha$. For small $\alpha$ the second method of \cite{simpson_et_al} provides better results than our approach, while its standard deviation is larger. It turns out that except for small $\alpha$ with $\rho =0$ and $\rho =0.25$ our algorithm better detects the extremal clusters. 

\section{Application to real-world data} \label{sec:applications}

\subsection{Preprocessing for real-world data}\label{subsec:preprocess_data}

In Section \ref{sec:numerical_results} we considered an example with standard Fr\'echet marginal distributions so that the tail index (see Equation \eqref{eq:intro_regular_variation_spectral_vector}) of the considered vectors is equal to $1$. The influence of this index on the extremal clusters has been studied on some numerical results by \cite{meyer_wintenberger}. It turns out that a large tail index does not provide accurate results while a small one highlights one-dimensional clusters, see Remark 11 in their article. A tail index of $\alpha=1$ seems to provide the best results.

For real-world data the estimation of the tail index of a sample $\boldx_1, \ldots, \boldx_n$ is achieved with a Hill plot (\cite{hill_1975}). It consists in plotting
\[
\widehat{\alpha}(k) = \bigg( \frac{1}{k} \sum_{j = 1}^k \log( |\boldx|_{(j)}) - \log( |\boldx|_{(k)}) \bigg)^{-1}\,, \quad k=2, \ldots, n\,,
\]
where $|\boldx|_{(j)}$ denotes the order statistics of the norms $|\boldx_1|, \ldots, |\boldx_n|$, i.e. $|\boldx|_{(1)} \geq \ldots \geq |\boldx|_{(n)}$, and to choose $\hat \alpha$ as the value around which the plot stabilizes. Then, we consider the power transform $\boldx'_j = (\boldx_j)^{\hat \alpha}$. This transformation highlights the tail structure of the data without modifying the support of the spectral measure, see \cite{meyer_wintenberger}, Remark 8. It differs from the standardization discussed in the introduction for which the vectors are normalized via a rank transform.

In the following section we apply MUSCLE to financial data. An application on wind speed data can be found in the Supplementary Material.

\subsection{Extreme variability for financial data}

The data set we use corresponds to the value-average daily returns of $49$ industry portfolios compiled and posted as part of the Kenneth French Data Library. They are available at \url{https://mba.tuck.dartmouth.edu/pages/faculty/ken.french/data_library.html}. A related study on a similar dataset has been conducted by \cite{cooley_thibaud_2019}. We restrict our study to the period $1970-2019$ which provides $n = 12\,613$ observations denoted by $\boldx^{\text{obs}}_1, \ldots \boldx^{\text{obs}}_n \in \R^{49}$. Our goal is to study the variability of these returns so that we take the componentwise absolute value $\boldx_j = |\boldx^{\text{obs}}_j|$ of the data. Thus, we study the non-negative vectors $\boldx_1, \ldots, \boldx_n$ in $\R^d_+$ with $n = 12\,613$ and $d = 49$. Following Section \ref{subsec:preprocess_data}, we consider the vectors $\boldx'_j = (\boldx_j)^{\hat \alpha}$, where $\hat \alpha = 2.99$ is the Hill estimator of the sample $|\boldx_1|, \ldots, |\boldx_n|$.

Following Remark \ref{rem:plot_s0_k}, we plot the evolution of the estimator of the Kullback-Leibler divergence in \eqref{eq:minimize_KL_practice} as a function of $k$. We see on Figure \ref{fig:evolution_finance} that this estimator decreases until it reaches a minimal value for $\hat k= 441$, before increasing for $k \geq \hat k$. The level $\hat k$ corresponds to a proportion $\hat k /n = 3 \%$ and leads to a number of extremal clusters $\hat s(\hat k)= 14$. Contrary to the numerical results, we do not observe a range of $k$ for which the minimal value $\hat s(k)$ remains approximately constant.

\begin{figure}[!th]
	\centering
	\includegraphics[scale=0.17]{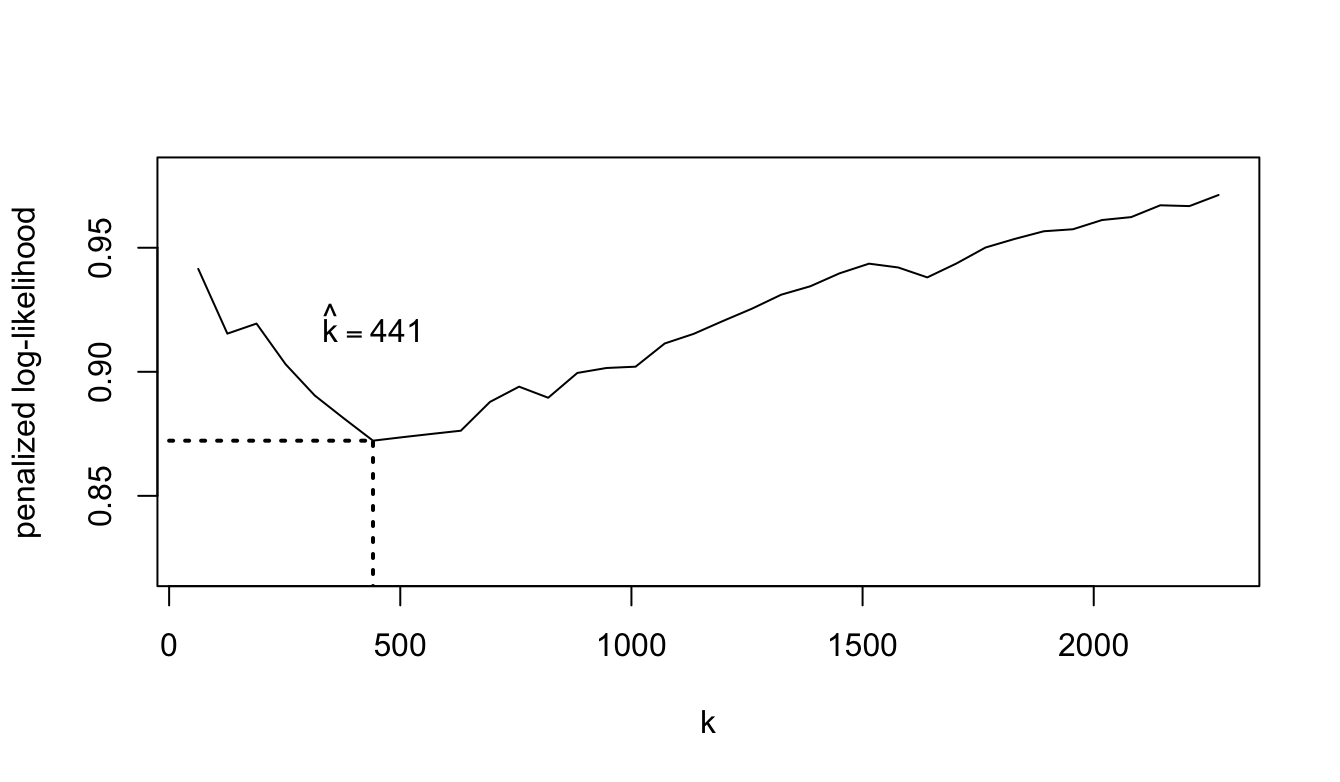}\hfill
	\includegraphics[scale=0.17]{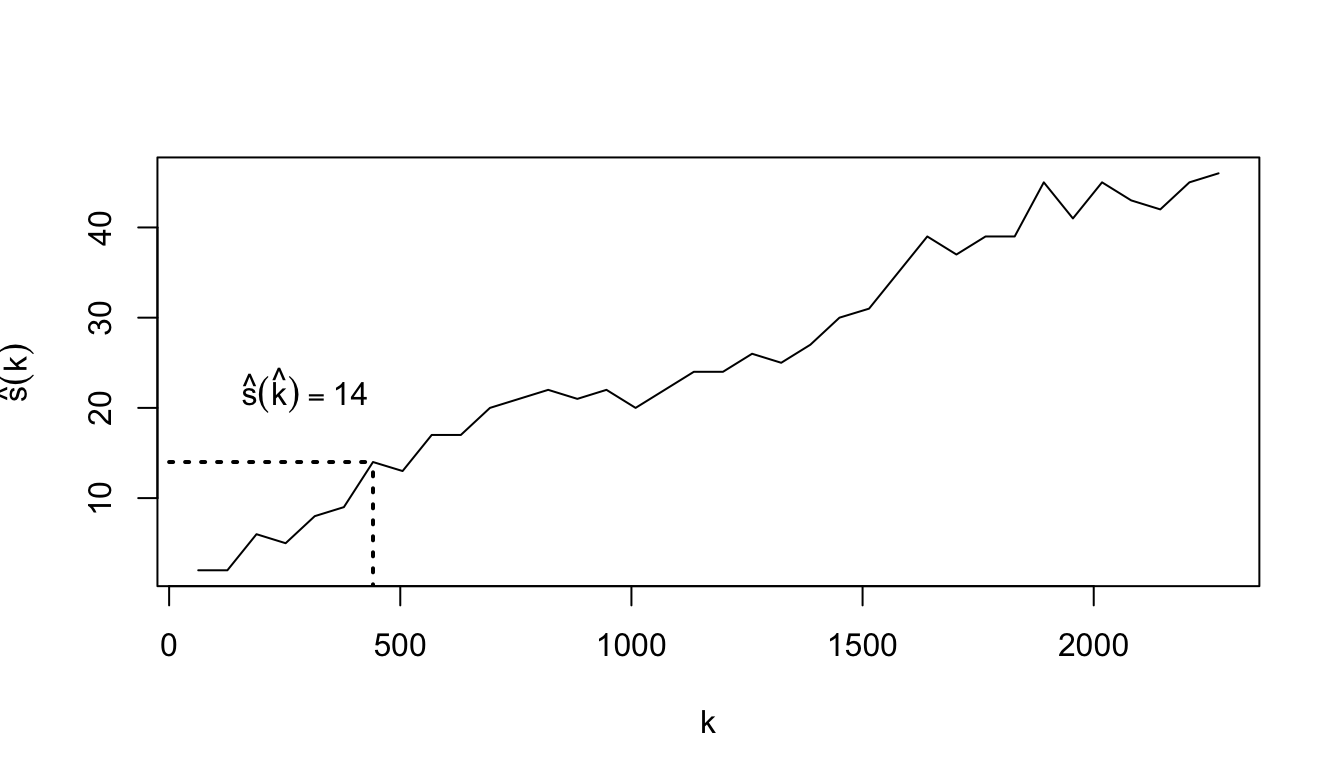}
	\caption{Evolution with respect to $k$ of the penalized log-likelihood given in \eqref{eq:minimize_KL_practice} (left) and of $\hat s_n(k)$ (right) for the financial data.\label{fig:evolution_finance}}
\end{figure}

MUSCLE provides $\hat s(\hat k) = 14$ extremal clusters which gather $12$ portfolios. These clusters and their inclusions are represented in Figure \ref{fig:features_financial_data}. The number of identified clusters is much smaller compared to the total number $2^{49} \approx 10^{15}$. Besides these clusters are at most three-dimensional so that our procedure drastically reduces the dimension of the study. Most of the extremal portfolios which appear in the clusters correspond to "office/executive" sectors, such as Health, Software, Hardware, Banks, Finance, Electronic Equipment (Chips), Real Estate. Some other clusters group portfolios related to heavy industries, such as Steel, Coal, and Gold. The only clusters gathering a heavy industry and service sectors are $\{\text{Coal, Banks}\}$ and $\{\text{Coal, Banks, Fin}\}$. The tail dependence of the variability of these different kinds of portfolios may result from the financing of the coal industry by several big banks, see \cite{financial_times}.

We conclude that the aforementioned $14$ clusters given by MUSCLE correspond to subsets $C_\beta$ which gather the mass of $\boldZ$. Among them, eight gather some mass of $\boldZ$ and are not included in larger subsets on which $\boldZ$ places mass. Following \cite{meyer_wintenberger}, Theorem 2, these \emph{maximal} subsets also concentrate the mass of the spectral measure. We refer to \cite{meyer_wintenberger}, Section 3.2, for a discussion on maximal and non-maximal subsets. Standard approaches which hold for low-dimensional extremes can then be applied to these subsets, see \cite{einmahl_dehaan_huang_93}, \cite{einmahl_97}, \cite{einmahl_segers_09}.

\begin{figure}[!h]
	\centering
	\begin{tikzpicture}
	\tikzstyle{group}=[ellipse, draw]
	
	\node[group] (CoalBanksFin) at (4.9,5) {Coal, Banks, Fin};
	
	\node[group] (HlthSoftw) at (1.5,3) {Hlth, Softw};
	\node[group] (CoalBanks) at (4.92,3) {Coal, Banks};
	\node[group] (SteelCoal) at (8.25,3) {Steel, Coal};
	\node[group] (GoldCoal) at (11.45,3) {Gold, Coal};
	\node[group] (HardwChips) at (15,3) {Hardw, Chips};
	
	\node[group] (Softw) at (1,1) {Softw};
	\node[group] (Hlth) at (3,1) {Hlth};
	\node[group] (Banks) at (5,1) {Banks};
	\node[group] (Coal) at (7,1) {Coal};
	\node[group] (Txtls) at (9,1) {Txtls};
	\node[group] (Gold) at (11,1) {Gold};
	\node[group] (RlEst) at (13,1) {RlEst};
	\node[group] (Smoke) at (15,1) {Smoke};
	
	\draw (Softw) to (HlthSoftw);
	\draw (Hlth) to (HlthSoftw);
	\draw (Coal) to (GoldCoal);
	\draw (Coal) to (CoalBanks);
	\draw (Coal) to (SteelCoal);
	\draw (Gold) to (GoldCoal);
	\draw (Banks) to (CoalBanks);
	\draw (CoalBanks) to (CoalBanksFin);
	\end{tikzpicture}
	\caption{Representation of the $14$ clusters and their inclusions. The abbreviations are the following ones: Softw = Computer Software, Txtls = Textiles, Hlth = Healthcare, RlEst = Real Estate, Hardw = Hardware, Chips = Electronic Equipment, Fin = Finance. \label{fig:features_financial_data}}
\end{figure}

After removing the $12$ extremal components we reapply MUSCLE to obtain the dependence structure of the non-extremal portfolios. The algorithm provides a unique cluster with all $37$ remaining portfolios. Hence these portfolios tend to have a dependent tail structure: their extreme variability is strongly correlated.

\section{Conclusion}
The statistical analysis introduced in this article provides a new approach to detect the extremal directions of a multivariate random vector $\boldX$. This method relies on the notion of sparse regular variation which better highlights the tail dependence of $\boldX$. Several convergence results are established in Section \ref{sec:asymptotic_results} and are used to build a rigorous statistical method based on model selection. This approach provides not only the clusters of directions on which the extremes of $\boldX$ gather but also a reasonable threshold above which the data are considered as extreme values. The latter issue has always been challenging and no theoretical-based procedure has been provided in a multivariate setting yet, even if it has been the subject of much attention in the literature. The choice of the directions is achieved with an AIC-type minimization whose penalization allows to reduce the number of selected subsets. Including the choice of the level $k$ of   the random threshold $|\boldX|_{(k)}$  then entails  multiplicative and additive penalization terms. This approach leads to the parameter-free algorithm MUSCLE whose purpose is to recover the extremal clusters of a sample of iid sparsely regularly varying random vectors $\boldX_1, \ldots, \boldX_n$.

The absence of any hyperparameter is a main difference with the existing methods (\cite{goix_sabourin_clemencon_17}, \cite{simpson_et_al}, \cite{chiapino_sabourin_2016}, \cite{chiapino_sabourin_segers_2019}).   Another main advantage of our procedure is that it is still efficient for large $d$. This follows from the expected linear-time algorithm introduced by \cite{duchi_et_al_2008} to compute the Euclidean projection.

The numerical experiments on max-mixture distributions provide promising results. Our algorithm provides better results than the ones of \cite{goix_sabourin_clemencon_17} and \cite{simpson_et_al} for $\rho$ close to $1$, or small $\rho$ and $\alpha$ close to $1$. Moreover the results do not vary a lot with $\rho$ and $\alpha$. Finally, the application of our algorithm on financial data highlights sparse clusters and thus reduces the dimension of the study. We obtain a sparse tail dependence structure for for the extreme variability of several industry portfolios. This reinforces the relevance of our approach for reducing the dimension in Extreme Value Theory.

\paragraph{Acknowledgments} We are grateful to two referees for careful reading of the paper and for useful suggestions.

\bibliographystyle{plainnat}

\begin{thebibliography}{0}
\providecommand{\natexlab}[1]{#1}
\providecommand{\url}[1]{\texttt{#1}}
\expandafter\ifx\csname urlstyle\endcsname\relax
  \providecommand{\doi}[1]{doi: #1}\else
  \providecommand{\doi}{doi: \begingroup \urlstyle{rm}\Url}\fi

\end{thebibliography}


\begin{thebibliography}{99}
	\bibitem[Abdous and Ghoudi(2005)]{abdous_ghoudi_2005}
	{\sc Abdous, B. and Ghoudi, K.} (2005). Non-parametric estimators of multivariate extreme dependence functions. {\em Nonparametric Statistics} {\bf 17,} 915--935.
	
	\bibitem[Akaike(1973)]{akaike_1973}
	{\sc Akaike, H.} (1973). Information theory and an extension of the maximum likelihood principle. {\em 2nd International Symposium on Information Theory,} 267--281.
	
	\bibitem[Beirlant et al.(2006)]{beirlant_et_al}
	{\sc Beirlant, J. and Goegebeur, Y. and Segers, J. and Teugels, J. L.} (2006). {\em Statistics of Extremes: Theory and Applications.}, John Wiley \& Sons Ltd., Chichester.
	
	\bibitem[Bingham et al.(1987)]{bingham_et_al}
	{\sc Bingham, N.H. and Goldie, C.M. and Teugels, J. L.} (1987). {\em Regular Variation.}, Cambridge University Press, Cambridge.
	
	\bibitem[Caiero and Gomes(2015)]{caeiro_gomes_2015}
	{\sc Caiero F. and Gomes, M.I.} (2015). Threshold selection in extreme value analysis. in {\em Extreme Value Modeling and Risk Analysis: Methods and Applications,} 71--89.
	
	\bibitem[Chautru(2015)]{chautru_2015}
	{\sc Chautru, E.} (2015). Dimension reduction in multivariate extreme value analysis. {\em Electronic Journal of Statistics} {\bf 9,} 383--418.
	
	\bibitem[Chiapino and Sabourin(2016)]{chiapino_sabourin_2016}
	{\sc Chiapino, M. and Sabourin, A.} (2016). Feature clustering for extreme events analysis, with application to extreme stream-flow data. {\em International Workshop on New Frontiers in Mining Complex Patterns,} 132--147.
	
	\bibitem[Chiapino et al.(2019)]{chiapino_sabourin_segers_2019}
	{\sc Chiapino, M., Sabourin, A. and Segers, J.} (2019). Identifying groups of variables with the potential of being large simultaneously. {\em Extremes} {\bf 22,} 193--222.
	
	\bibitem[Condat(2016)]{condat_2016}
	{\sc Condat L.} (2016). Fast projection onto the simplex and the $\ell_1$ ball. {\em Mathematical Programming} {\bf 158,} 575--585.
	
	\bibitem[Cooley and Thibaud(2019)]{cooley_thibaud_2019}
	{\sc Cooley, D. and Thibaud, E.} (2019). Decompositions of dependence for high-dimensional extremes. {\em Biometrika} {\bf 106,} 587--604
	
	\bibitem[Duchi et al.(2008)]{duchi_et_al_2008}
	{\sc Duchi, J. and Shalev-Shwartz, S. and Singer, Y. and Chandra, T.} (2008). Efficient projections onto the $\ell_1$-ball for learning in high dimensions. {\em Proceedings of the 25th international conference on Machine learning}, 272--279.
	
	\bibitem[Einmahl et al.(1993)]{einmahl_dehaan_huang_93}
	{\sc Einmahl, J. and de Haan, L. and Huang, X.} (1993). Estimating a multidimensional extreme-value distribution. {\em Journal of Multivariate Analysis} {\bf 47,} 35--47.
	
	\bibitem[Einmahl et al.(1997)]{einmahl_97}
	{\sc Einmahl, J. and de Haan, L. and Sinha, A.K.} (1997). Estimating the spectral measure of an extreme value distribution. {\em Stochastic Processes and their Applications} {\bf 70,} 143--171.
	
	\bibitem[Einmahl and Segers (2009)]{einmahl_segers_09}
	{\sc Einmahl, J. and Segers, J.} (2009). Maximum empirical likelihood estimation of the spectral measure of an extreme-value distribution. {\em The Annals of Statistics} {\bf 37,} 2953--2989.
	
	\bibitem[Goix et al.(2017)]{goix_sabourin_clemencon_17}
	{\sc Goix, N. and Sabourin, A. and Cl{\'e}men{\c{c}}on, S.} (2017). Sparse representation of multivariate extremes with applications to anomaly detection. {\em Journal of Multivariate Analysis} {\bf 161,} 12--31.
	
	\bibitem[de Haan and Ferreira(2006)]{dehaan_ferreira}
	{\sc de Haan, L. and Ferreira, A.} (2006). {\em Extreme Value Theory: An Introduction}, Springer, New-York.
	
	\bibitem[de Haan and Resnick(1993)]{dehaan_resnick_93}
	{\sc de Haan, L. and Resnick, S.I.} (1993). Estimating the limit distribution of multivariate extremes. {\em Stochastic Model} {\bf 9,} 275--309.
	
	\bibitem[Heffernan and Tawn(2004)]{heffernan_tawn_2004}
	{\sc Heffernan, J. E and Tawn, J. A} (1989). A conditional approach for multivariate extreme values (with discussion). {\em Journal of the Royal Statistical Society: Series B (Statistical Methodology)} {\bf 66,} 497--546.
	
	\bibitem[Hill(1975)]{hill_1975}
	{\sc Hill, B.M.} (1975). A simple general approach to inference about the tail of a distribution. {\em The Annals of Statistics} {\bf 3,} 1163--1174.
	
	\bibitem[Hult and Lindskog(2006)]{hult_lindskog_06}
	{\sc Hult, H. and Lindskog, F.} (1989). Regular variation for measures on metric spaces. {\em Publications de l'Institut Math\'ematique} {\bf 80,} 121--140.
	
	\bibitem[Kiriliouk et al.(2019)]{kiriliouk_et_al_2019}
	{\sc Kiriliouk, A. and Rootz{\'e}n, H. and Segers, J. and Wadsworth, J.} (2019). Peaks over thresholds modeling with multivariate generalized {P}areto distributions. {\em Technometrics} {\bf 61,} 123--135.

	\bibitem[Kullback and Leibler(1951)]{kullback_leibler_1951}
	{\sc Kullback, S. and Leibler, R.A.} (1951). On information and sufficiency. {\em The Annals of Mathematical Statistics} {\bf 22,} 79--86.
	
	\bibitem[Kyrillidis et al.(2013)]{kyrillidis_et_al_2013}
	{\sc Kyrillidis, A., Becker, S., Cevher, V. and Koch, C.} (2013). Sparse projections onto the simplex. {\em International Conference on Machine Learning} {\bf 28,} 235--243.
	
	\bibitem[Ledford and Tawn(1996)]{ledford_tawn_1996}
	{\sc Ledford, A. W. and Tawn, J. A.} (1996). Statistics for near independence in multivariate extreme values. {\em Biometrika} {\bf 83,} 169--187.
		
	\bibitem[Lindskog et al.(2014)]{lindskog_resnick_roy_2014}
	{\sc Lindskog, F. and Resnick, S. I. and Roy, J.} (2014). Regularly varying measures on metric spaces: Hidden regular variation and hidden jumps. {\em Probability Survey} {\bf 11,} 270--314.
	
	\bibitem[Massart(2007)]{massart_2007}
	{\sc Massart, P.} (2007). {\em Concentration inequalities and model selection}, Springer, Berlin.
	
	\bibitem[Meyer and Wintenberger(2021)]{meyer_wintenberger}
	{\sc Meyer, N. and Wintenberger, O.} (2021). Sparse regular variation. {\em Advances in Applied Probability}, {\bf 53,} 1115 - 1148.
	
	\bibitem[Raval et al.(2020)]{financial_times}
	{\sc Raval, A. and Owen, W. and Hume, N. and Stephen, M.} (2020). Biggest Banks Sustain Coal Financing despite Defunding Drive. {\em Financial Times}, 3 Sept. 2020.
	
	\bibitem[Resnick(1987)]{resnick_87}
	{\sc Resnick, S. I.} (1987). {\em Extreme Values, Regular Variation and Point Processes}. Springer, New-York.
	
	\bibitem[Resnick(2007)]{resnick_07}
	{\sc Resnick, S. I.} (2007). {\em Heavy-Tail Phenomena: Probabilistic and Statistical Modeling.} Springer, New-York.
	
	\bibitem[Segers(2012)]{segers_12}
	{\sc Segers, J.} (2012). Max-stable models for multivariate extremes. {\em Revstat Statistical Journal} {\bf 10,} 61--82.
	
	\bibitem[Simpson et al.(2020)]{simpson_et_al}
	{\sc Simpson, E., Wadsworth, J. L and Tawn, J. A.} (2020). Determining the dependence structure of multivariate extremes. {\em Biometrika} {\bf 107,} 513--532.
	
	\bibitem[St{\u{a}}ric{\u{a}}(1999)]{starica_1999}
	{\sc St{\u{a}}ric{\u{a}}, C.} (1999). Multivariate extremes for models with constant conditional correlations. {\em Journal of Empirical Finance} {\bf 6,} 515--553.
	
	\bibitem[Van der Vaart and Wellner(1996)]{vandervaart_wellner_96}
	{\sc van der Vaart, A. and Wellner, J.} (1996). {\em Weak convergence and empirical processes: with applications to statistics.} Springer Science \& Business Media, New-York.
	
	\bibitem[Wan and Davis(2019)]{wan_davis_2019}
	{\sc Wan, P. and Davis, R.A.} (2019). Threshold selection for multivariate heavy-tailed data. {\em Extremes} {\bf 22,} 131--166.
	
\end{thebibliography}

\end{document}


\def\spacingset#1{\renewcommand{\baselinestretch}%
{#1}\small\normalsize} \spacingset{1}
\spacingset{1.5} 

\if1\blind
{
	\title{Supplementary material for\\\emph{Multivariate sparse clustering for extremes}}
  \author{Nicolas Meyer\thanks{nicolas.meyer@umontpellier.fr (corresponding author)}\hspace{.2cm}\\
	IMAG, Univ. Montpellier, CNRS, Montpellier, France\\
	LEMON, Inria, Montpellier, France\\
	and \\
	Olivier Wintenberger\thanks{olivier.wintenberger@sorbonne-universite.fr}\\
	Sorbonne Universit\'e, LPSM, F-75005, Paris, France\\
	Wolfgang Pauli Institut, c/o Fakultät für Mathematik,\\
	Universität Wien, 1090 Vienna, Austria }
\maketitle
} \fi

\if0\blind
{
	\bigskip
	\bigskip
	\bigskip
	\begin{center}
		{\LARGE\bf Supplementary material for\\\emph{Multivariate sparse clustering for extremes}}
	\end{center}
	\medskip
} \fi

The supplementary material for this paper contains proofs and some details on our procedure. We also apply the MUSCLE algorithm to a numerical example related to asymptotic independence and wind speed data.

\section{Proofs}

This section contains the proofs of the results introduced in the main text. We start by a result which will be of constant use in what follows.

\subsection{A consequence of Assumption \ref{main-ass:exponent_measure}}\label{proof-subsec:assumption1}

Assumption \ref{main-ass:exponent_measure} allows us to deal with the limit of $\P(\boldX/u_n \in A(x,\beta))$ as $n \to \infty$. Indeed, this assumption implies the following uniform convergence over $x \in [\frac{1}{1+\tau},1+\tau]$, $\tau > 0$:
\begin{equation}\label{eq:uniform_reg_var}
	\dfrac{n}{k}\dfrac{\P(\boldX/u_n \in A(x,\beta))}{p_n(\beta)}
	\to x^{\alpha(\beta)}\,, \quad n \to \infty\,.
\end{equation}
To prove this convergence we use the equivalence established by \cite{meyer_wintenberger}:
\begin{equation}\label{eq:equiv_proj_C_beta}
	\pi(\boldx)\in C_\beta \quad \text{if and only if} \quad \bigg\{
	\begin{array}{ll}
		&\max_{i \in \beta} \boldx_{\beta, \, i} < 1\,,\\
		&\min_{i \in \beta^c} \boldx_{\beta, \, i} \geq 1\,,
	\end{array}
\end{equation}
where $\boldx_{\beta, \, i} = \sum_{j \in \beta} (x_j - x_i)$. The regular variation assumption and the homogeneity property of $\mu_\beta$ then yield
\begin{align*}
	\dfrac{\P(\boldX/u_n \in A(x,\beta))}{\P(\min_{i \in \beta^c} \boldX_{\beta, \, i} \geq u_n)}
	&= \dfrac{\P(\max_{i \in \beta} x \boldX_{\beta, \, i} < u_n,\, \min_{i \in \beta^c} x \boldX_{\beta, \, i} \geq u_n)}{\P(\min_{i \in \beta^c} \boldX_{\beta, \, i} \geq u_n)}\\
	&\xrightarrow[n \to \infty]{} \mu_\beta(\{\boldx \in \R^d_+ : \max_{i \in \beta} x \boldx_{\beta, \, i} < 1,\, \min_{i \in \beta^c} x \boldx_{\beta, \, i} \geq 1\})\\
	&= x^{\alpha(\beta)} \mu_\beta(\{\boldx \in \R^d_+ : \max_{i \in \beta} \boldx_{\beta, \, i} < 1,\, \min_{i \in \beta^c} \boldx_{\beta, \, i} \geq 1\})\,,
\end{align*}
and this limit is positive by Assumption \ref{main-ass:exponent_measure}. This implies in particular that the function $t \mapsto \P(\boldX/t \in A(1, \beta))= \P(\boldX \in A(t^{-1}, \beta))$ is regularly varying with index $-\alpha(\beta)$. We conclude by writing down
\[
\dfrac{n}{k}\dfrac{\P(\boldX/u_n \in A(x,\beta))}{p_n(\beta)}
\sim \dfrac{\P(x \boldX/u_n \in A(1,\beta))}{\P(\boldX/u_n \in A(1,\beta))}
\to x^{\alpha(\beta)}\,, \quad n \to \infty\,.
\]
%

\begin{rem*}
	For $\beta \in \mathcal{S}^*(\boldZ)$, the convergence is automatically satisfied since
	\[
	\dfrac{n}{k}\dfrac{\P(\boldX/u_n \in A(x,\beta))}{p_n(\beta)}
	= \dfrac{\P(\pi(x\boldX/u_n) \in C_\beta \mid x|\boldX|> u_n)}{p_n(\beta)} \dfrac{\P(x|\boldX| > u_n)}{\P(|\boldX| > u_n)}
	\xrightarrow[n \to \infty]{} \dfrac{p^*(\beta)}{p^*(\beta)} x^{\alpha} = x^\alpha\,,
	\]
	where $\alpha$ is the tail index of $\boldX$, see Equation \eqref{main-eq:conv_reg_var_tail_measure} in the main text.
\end{rem*}

\subsection{Proof of Lemma \ref{main-lem:proba_T_n_beta_zero}}
In order to prove the convergence
\[
\dfrac{\log \P(T_{n,k}(\beta)=0)}{-k p_n(\beta)} \to 1\,, \quad n \to \infty\,,
\]
we prove that for $\tau > 0$,
\begin{equation}\label{eq:proof_unif_conv_Taylor}
	\sup_{x \in [\frac{1}{1+\tau}, 1+\tau]}
		\bigg| \dfrac{\log \P(T_{n}(x, \beta)=0)}{-k p_n(\beta)} -x^{\alpha(\beta)} \bigg| \to 0\,, \quad n \to \infty\,,
\end{equation}
Then, letting $x = x_n = u_n / |\boldX|_{(k+1)}$ in \eqref{eq:proof_unif_conv_Taylor} gives the desired result since $x_n \to 1$ in probability.

First we notice that
\begin{align*}
\P(T_n(x, \beta) = 0)
&= \P (\pi(x \boldX_j / u_n) \notin C_\beta \text{ or } x|\boldX_j| \leq u_n,\, j = 1, \ldots, n)\\
&= [1- \P (\boldX/u_n \in A(x,\beta)) ]^n\\
&= \exp \big( n \log [1- \P (\boldX/u_n \in A(x,\beta)) ] \big)\,.
\end{align*}
Therefore, the convergence \eqref{eq:proof_unif_conv_Taylor} is equivalent to
\begin{equation}\label{eq:proof_unif_conv_Taylor2}
\sup_{x \in [\frac{1}{1+\tau}, 1+\tau]}
	\bigg| \dfrac{n \log [1- \P (\boldX/u_n \in A(x,\beta)) ]}{-k p_n(\beta)} -x^{\alpha(\beta)} \bigg| \to 0\,, \quad n \to \infty\,.
\end{equation}
We decompose this quantity as follows:
\begin{align*}
	\bigg| &\dfrac{n \log [1- \P (\boldX/u_n \in A(x,\beta)) ]}{-k p_n(\beta)} -x^{\alpha(\beta)} \bigg|\\
	&\leq
	\bigg| \dfrac{n \P (\boldX/u_n \in A(x,\beta))}{k p_n(\beta)} -x^{\alpha(\beta)} \bigg|
	\bigg| \dfrac{ - \log [1- \P (\boldX/u_n \in A(x,\beta)) ]}{ \P (\boldX/u_n \in A(x,\beta))} \bigg|\\
	&+
	\bigg| \dfrac{ -\log [1- \P (\boldX/u_n \in A(x,\beta)) ]}{\P (\boldX/u_n \in A(x,\beta))} - 1 \bigg| |x^{\alpha(\beta)}|\end{align*}
Regarding the second term, we first notice that on $[\frac{1}{1+\tau}, 1+\tau]$ we have $ |x^{\alpha(\beta)}| \leq (1+\tau)^{\alpha(\beta)}$. Besides, Taylor's inequality ensures that $|\log(1+x) - x| \leq |x|^2/6$ for any $x > -1$. This implies that
\[
\bigg| \dfrac{ -\log [1- \P (\boldX/u_n \in A(x,\beta)) ]}{\P (\boldX/u_n \in A(x,\beta))} - 1 \bigg|
\leq \dfrac{1}{6}
\P (\boldX/u_n \in A(x,\beta))\,.
\]
and this probability converges uniformly to zero by \eqref{eq:uniform_reg_var}.
Finally, we know by \eqref{eq:uniform_reg_var} that 
\[
\bigg| \dfrac{n \P (\boldX/u_n \in A(x,\beta))}{k p_n(\beta)} -x^{\alpha(\beta)} \bigg|\,,
\]
also converges uniformly to zero. This proves \eqref{eq:proof_unif_conv_Taylor2} and concludes the proof.

\subsection{Proof of Theorem \ref{main-theo:CLT}}
We define the class of sets
\[
\mathcal{U}(\mathcal{S}_\infty)
= \big\{ C_{\beta_1} \cup \cdots \cup C_{\beta_l},\, \text{ with distinct } \beta_1, \ldots, \beta_l \in \mathcal{S}_\infty,\, l = 1, \ldots, s_\infty \big\}\,,
\]
whose  cardinality is $2^{s_\infty}-1$. Consequently, for $\mathsf{C} = C_{\beta_1} \cup \cdots \cup C_{\beta_l} \in \mathcal{U}(\mathcal{S}_\infty)$ the set $A(x,\mathsf{C}) := \{\boldy \in \R^d_+ : x |\boldy| > 1, \pi(x \boldy) \in \mathsf{C} \}$ satisfies
\[
T_n(x,\mathsf{C})
:= \sum_{j=1}^n \indic_{\{\boldX/u_n \in A(x,\mathsf{C})\}}
= \sum_{j=1}^n \indic_{\{x |\boldX| > u_n\}} \Big( \sum_{j=1}^l \indic_{\{\pi(x\boldX/u_n) \in C_{\beta_j}\}} \Big)\,,
\]
since the unions which appear in $\mathsf{C}$ are disjoint ones. We also define
\[
p_n(\mathsf{C}) = \P( \pi(\boldX/u_n) \in \mathsf{C} \mid |\boldX| > u_n)\,.
\]
With such definitions if $\mathsf{C} = C_{\beta}$ then we get $T_n(x,\mathsf{C}) = T_n(x,\beta)$ and $p_n(\mathsf{C}) = p_n(\beta)$.

In order to prove Theorem \ref{main-theo:CLT} we establish the convergence in $\ell^\infty([\frac{1}{1+\tau}, 1 +\tau])$ of the process
\begin{equation}\label{eq:process_union}
\bigg( \dfrac{ T_{n}(x, \mathsf{C}) }{k \sqrt{p_n(\mathsf{C})}} - \dfrac{n}{k \sqrt{p_n(\mathsf{C})}} \P\Big(\frac{\boldX}{u_n} \in A(x,\mathsf{C})\Big) \bigg)_{\mathsf{C} \in \mathcal{U}(\mathcal{S}_\infty)}\,,
\end{equation}
under a suitable normalization. Theorem \ref{main-theo:CLT} is then just a consequence of this convergence with a specific choice of subsets $\mathsf{C}$.

To prove the convergence of the process \eqref{eq:process_union} we proceed as usual by first establishing the convergence of all finite dimensional marginal distributions and then the asymptotic equicontinuity of the process.

\noindent 1. \underline{Convergence of the marginals}

We prove the convergence of \eqref{eq:process_union} under a suitable normalization for a fixed $x \in [\frac{1}{1+\tau}, 1+\tau]$. We consider the triangular array $\boldV_n^{1}(x), \ldots, \boldV_n^{n}(x)$, with generic distribution $\boldV_n(x) \in \R^{2^{s_\infty}-1}$, where $\boldV_n^{i}(x)$ has components
\[
\dfrac{\sqrt{n} \indic_{\{\boldX_i / u_n \in A(x,\mathsf{C})\}}}{\sqrt{k p_n(\mathsf{C})}}\,, \quad \mathsf{C} \in \mathcal{U}(\mathcal{S}_\infty)\,.
\]
Let $\boldSigma(\boldV_n(x)) \in \R^{2^{s_\infty}-1} \times \R^{2^{s_\infty}-1}$ be the covariance matrix of $\boldV_n(x)$. For $\mathsf{C},\mathsf{C}' \in \mathcal{U}(\mathcal{S}_\infty)$, we have
\begin{align*}
\boldSigma(&\boldV_n(x))_{\mathsf{C}, \mathsf{C}'}\\
&=
\cov\Big(\dfrac{\sqrt{n} \indic_{\{\boldX / u_n \in A(x,\mathsf{C})\}}}{\sqrt{k p_n(\mathsf{C})}}, \dfrac{\sqrt{n} \indic_{\{\boldX / u_n \in A(x,\mathsf{C}')\}}}{\sqrt{k p_n(\mathsf{C}')}} \Big)\\
&= \dfrac{n}{k} \dfrac{\E \big[  \indic_{\{\boldX / u_n \in A(x,\mathsf{C})\}}  \indic_{\{\boldX / u_n \in A(x,\mathsf{C}')\}} \big]}{\sqrt{p_n(\mathsf{C})} \sqrt{p_n(\mathsf{C}')}}
- \dfrac{n}{k} \dfrac{\P(\boldX / u_n \in A(x,\mathsf{C})) \P(\boldX / u_n \in A(x,\mathsf{C}'))}{\sqrt{p_n(\mathsf{C})} \sqrt{p_n(\mathsf{C}')}}\,.
\end{align*}
If we write $\mathsf{C} = C_{\beta_1} \cup \cdots \cup C_{\beta_l}$ and $\mathsf{C}' = C_{\beta'_1} \cup \cdots \cup C_{\beta'_{l'}}$, with disjoint unions, then the convergence \eqref{eq:uniform_reg_var} ensures that
\[
\P(\boldX / u_n \in A(x,\mathsf{C}))
= \sum_{j=1}^l \P(\boldX / u_n \in A(x,\beta_j))
= \frac{k}{n} \sum_{j=1}^l p_n(\beta_j) x^{\alpha(\beta_j)} (1+o(1))\,,
\]
and similarly for $\mathsf{C}'$. This implies that the second term in the decomposition of $\boldSigma(\boldV_n(x))_{\mathsf{C}, \mathsf{C}'}$ satisfies
\[
\dfrac{n}{k} \dfrac{\P(\boldX / u_n \in A(x,\mathsf{C})) \P(\boldX / u_n \in A(x,\mathsf{C}'))}{\sqrt{p_n(\mathsf{C})} \sqrt{p_n(\mathsf{C}')}}
= \dfrac{k}{n} \dfrac{\sum_{j=1}^l p_n(\beta_j) x^{\alpha(\beta_j)} \sum_{i=1}^{l'} p_n(\beta_i) x^{\alpha(\beta'_i)}}{\sqrt{p_n(\mathsf{C})} \sqrt{p_n(\mathsf{C}')}} (1 + o(1))\,.
\]
Regarding the first term, it is non-null if and only if there are common subsets in $\mathsf{C}$ and $\mathsf{C}'$. We use again \eqref{eq:uniform_reg_var}  and obtain
\begin{align*}
\dfrac{n}{k} \dfrac{\E \big[ \indic_{\{\boldX / u_n \in A(x,\mathsf{C})\}}  \indic_{\{\boldX / u_n \in A(x,\mathsf{C}')\}} \big]}{\sqrt{p_n(\mathsf{C})} \sqrt{p_n(\mathsf{C}')}}
&= \dfrac{n}{k} \dfrac{\sum_{\beta : C_\beta \subset \mathsf{C} \cap \mathsf{C}'} \P(\boldX / u_n \in A(x,\beta))}{\sqrt{p_n(\mathsf{C})} \sqrt{p_n(\mathsf{C}')}}\\
&= \dfrac{ \sum_{\beta : C_\beta \subset \mathsf{C} \cap \mathsf{C}'} p_n(\beta) x^{\alpha(\beta)}}{\sqrt{p_n(\mathsf{C})} \sqrt{p_n(\mathsf{C}')}} (1+o(1))\,,
\end{align*}
where the sum is equal to zero if the intersection $\mathsf{C} \cap \mathsf{C}'$ is empty. This proves that the covariance matrix $\boldSigma(\boldV_n(x))$ has the same asymptotic behavior as the matrix $\widetilde{\boldS}_{n}(x)$ defined by
\begin{equation}\label{eq:covariance_matrix_main_theorem}
\widetilde{\boldS}_{n}(x)_{\mathsf{C}, \mathsf{C}'}
=
\left\{
\begin{array}{cc}
\dfrac{\sum_{\beta : C_\beta \subset \mathsf{C} \cap \mathsf{C}'} p_n(\beta) x^{\alpha(\beta)}}{\sqrt{ p_n(\mathsf{C}) } \sqrt{ p_n(\mathsf{C}') }} & \text{ for } \mathsf{C} \cap \mathsf{C}' \neq \emptyset\,, \\
- \dfrac{k}{n}\dfrac{\sum_{\beta : C_\beta \subset \mathsf{C}} p_n(\beta) x^{\alpha(\beta)} \sum_{\beta' : C_{\beta'} \subset \mathsf{C}'} p_n(\beta') x^{\alpha(\beta')} }{ \sqrt{ p_n(\mathsf{C}) } \sqrt{ p_n(\mathsf{C}') } } & \text{ for } \mathsf{C} \cap \mathsf{C}' = \emptyset \,.
\end{array}
\right.
\end{equation}


Since we assume that the $\beta$'s are in $\mathcal{S}_\infty$, the Lindeberg-Feller's condition is satisfied:
\begin{align}
	&\dfrac{1}{n} \sum_{i=1}^n
	\E\bigg[ |\boldV_n^{i}(x) - \E[\boldV_n^{i}(x)]|_\infty^2 \indic_{|\boldV_n^{i}(x) - \E[\boldV_n^{i}(x)]|_\infty > \epsilon \sqrt{n}} \bigg] \label{eq:linderberg}\\
	&= \E\bigg[ |\boldV_n(x) - \E[\boldV_n(x)]|_\infty^2 \indic{\Big\{ \sup_{\mathsf{C} \in \mathcal{U}(\mathcal{S}_\infty)} |\indic_{\{\boldX/u_n \in A(x, \mathsf{C})\}} - \P(\boldX/u_n \in A(x, \mathsf{C}))| > \epsilon \sqrt{k p_n(\mathsf{C})} \Big\}} \bigg]\nonumber\\
	&\to 0\,, \nonumber
\end{align}
where $|\cdot|_\infty$ denotes the $\ell^\infty$-norm. Indeed, the inequality in the indicator function is never satisfied for $n$ large enough since the left-hand side is bounded by $1$ while the right-hand side diverges to infinity for any $\epsilon > 0$ as all clusters belong to $\mathcal{S}_\infty$.

By Lindeberg-Feller's theorem we obtain the weak convergence
\[
(n \widetilde{\boldS}_{n}(x))^{-1/2} \sum_{i=1}^n \Big\{ \boldV_n^{i}(x) - \E[\boldV^i_n(x)] \Big\} \stackrel{d}{\to} (N(\mathsf{C}))_{\mathsf{C} \in \mathcal{U}(\mathcal{S}_\infty)}\,,
\]
where $(N(\mathsf{C}))_{\mathsf{C} \in \mathcal{U}(\mathcal{S}_\infty)}$ is a standard centered Gaussian random vector. This convergence can be rewritten as
\begin{equation} \label{eq:cv_TCL_proof}
	\sqrt k (\widetilde{\boldS}_{n}(x))^{-1/2} \bigg( \dfrac{ T_{n}(x, \mathsf{C}) }{k \sqrt{p_n(\mathsf{C})}} - \dfrac{n  \P\big(\boldX / u_n \in A(x,\mathsf{C})\big) }{k \sqrt{p_n(\mathsf{C})}} \bigg)_{\mathsf{C} \in \mathcal{U}(\mathcal{S}_\infty)}
	\overset{d}{\to}
	(N(\mathsf{C}))_{\mathsf{C} \in \mathcal{U}(\mathcal{S}_\infty)}\,.
\end{equation}

\noindent 2. \underline{Equicontinuity}

In order to prove that the convergence in \eqref{eq:cv_TCL_proof} holds uniformly over $x \in [\frac{1}{1+\tau}, 1 + \tau]$ it suffices to establish the equicontinuity of the process $\{T_n(x, \mathsf{C})\}_{x, \mathsf{C}}$ in $\ell^\infty([\frac{1}{1+\tau}, 1+\tau] \times \mathcal{U}_\infty)$, see Theorem 2.11.9 by \cite{vandervaart_wellner_96}. We consider the totally bounded semi-metric space $[\frac{1}{1+\tau}, 1+\tau] \times \mathcal{U}_\infty$ equipped with the semi-metric $\rho$ defined by
\[
\rho((x, \mathsf{C}), (x', \mathsf{C}')) = \max_{\gamma \in \mathcal{S}_\infty} \Big(x^{\alpha(\gamma)} \vee {x'}^{\alpha(\gamma)} - x^{\alpha(\gamma)} \wedge {x'}^{\alpha(\gamma)}\Big) + \indic_{\mathsf{C} \neq \mathsf{C}'}\,.
\]
We check that the three assumptions given in Theorem 2.11.9 in \cite{vandervaart_wellner_96} are satisfied.

\noindent 2.a. \underline{The Lindeberg condition:} It is similar to the one established in the first part of the proof.

\noindent 2.b. \underline{The bracketing condition:} Recall that the vector $\boldV_n(x) \in \R^{2^{s_\infty}-1}$ has components
\[
V_n(x, \mathsf{C}) = \dfrac{\sqrt{n} \indic_{\{\boldX/u_n \in A(x, \mathsf{C})\}} }{\sqrt{k p_n(\mathsf{C})}}\,,
\quad \mathsf{C} \in \mathcal{U}(\mathcal{S}_\infty)\,.
\]
We consider a sequence $(\delta_n)_{n \geq 1}$ decreasing to zero. For $n$ large enough the condition $\rho((x,\mathsf{C}), (x',\mathsf{C}')) < \delta_n$ is then satisfied only for $\mathsf{C} = \mathsf{C}'$. This leads to the following relations:
\begin{align*}
	\E\Big[ \big(V_n(x, &\mathsf{C}) - V_n(x', \mathsf{C})\big)^2 \Big]\\
	&= \dfrac{n}{k p_n(\mathsf{C})} \E[ (\indic_{\{\boldX/u_n \in A(x,\mathsf{C})\}} -\indic_{\{\boldX/u_n \in A(x',\mathsf{C})\}})^2]\\
	&= \dfrac{n}{k p_n(\mathsf{C})} \Big[ \P\Big( \boldX/u_n \in A(x,\mathsf{C}) \cap {A(x',\mathsf{C})}^c \Big) + \P\Big( \boldX/u_n \in {A(x,\mathsf{C})}^c \cap A(x',\mathsf{C})\ \Big) \Big]\,,
\end{align*}
where $A(x,\mathsf{C})^c$ denotes the complementary of the set $A(x,\mathsf{C})$. Since the set $A(x,\mathsf{C})$ consists of disjoint unions of $A(x,\beta)$ it can be decomposed as
\[
\P\Big( \boldX/u_n \in A(x,\mathsf{C}) \cap A(x',\mathsf{C})^c \Big)
= \sum_{\beta : C_\beta \subset \mathsf{C}} \P( \boldX/u_n \in A(x,\beta) \cap A(x',\mathsf{C})^c)\,.
\]
If $\beta$ is such that $C_\beta \subset \mathsf{C}'$ then we have the inclusion $A(x',\mathsf{C})^c \subset A(x',\beta)^c$. This gives the inequality
\[	
\P\Big( \boldX/u_n \in A(x,\mathsf{C}) \cap A(x',\mathsf{C})^c \Big)
\leq
\sum_{\beta : C_\beta \subset \mathsf{C}} \P \big( \boldX/u_n \in A(x,\beta) \cap A(x',\beta)^c \big)
=: \sum_{\beta : C_\beta \subset \mathsf{C}} v_n(\beta, x, x')\,.
\]
Similarly we set $w_n(\beta, x, x') = \P \big( \boldX/u_n \in A(x,\beta)^c \cap A(x',\beta) \big)$ and the same arguments applied to $A(x,\mathsf{C})^c \cap A(x',\mathsf{C})$ entail that
\begin{align}
\E\Big[ \big(V_n(x, \mathsf{C}) - V_n(x', \mathsf{C})\big)^2 \Big]
&\leq
\dfrac{n}{k p_n(\mathsf{C})}
\Big(
\sum_{\beta : C_\beta \subset \mathsf{C}} v_n(\beta, x, x') + \sum_{\beta \subset \mathsf{C}} w_n(\beta, x, x')
\Big) \nonumber\\
&\leq
\sum_{\beta : C_\beta \subset \mathsf{C}} \Big( \dfrac{n v_n(\beta, x, x')}{k p_n(\beta)} + \dfrac{n w_n(\beta, x, x')}{k p_n(\beta)}
\Big)\,, \label{eq:proof_decomp_vn_wn}
\end{align}
where we used that $p_n(\beta) \leq p_n(\mathsf{C})$ for any $\beta$ such that $C_\beta \subset \mathsf{C}$.

We deal with the term $v_n(\beta, x,x')$ (the calculations with $w_n(\beta, x,x')$ are similar). Since we assumed the marginals of $\boldX$ to be continuous the event $\{ x'|\boldX| = u_n\}$ has zero probability. Moreover we recall that we project onto the $\ell^1$-ball (see Remark \ref{main-rem:projection_ball} in the main text), therefore if $\pi(x' \boldX / u_n) \in C_\beta$ then $x' |\boldX| > u_n$. Both considerations entail that
\[
v_n(\beta, x, x')
= \P( x |\boldX| > u_n, \pi(x\boldX/u_n) \in C_\beta, \pi(x'\boldX/u_n) \notin C_\beta)\,.
\]
Then we use the equivalence \eqref{eq:equiv_proj_C_beta} which implies that
\begin{align}
v_n(\beta, x, x')
\leq \P \Big( &\max_{i \in \beta} x\boldX_{\beta, \, i} / u_n < 1\,,  \min_{i \in \beta^c} x \boldX_{\beta, \, i} / u_n \geq 1, \max_{i \in \beta} x'\boldX_{\beta, \, i} / u_n \geq 1 \Big) \nonumber\\
&+  \P \Big( \max_{i \in \beta} x\boldX_{\beta, \, i} / u_n < 1\,,  \min_{i \in \beta^c} x \boldX_{\beta, \, i} / u_n \geq 1, \min_{i \in \beta} x'\boldX_{\beta, \, i} / u_n < 1 \Big) \label{eq:proof_v_n_2nd_part}\,.
\end{align}
Note that the first (resp. second) probability is non-null only if  $x<x'$ (resp. $x' < x$).
The first term in \eqref{eq:proof_v_n_2nd_part} multiplied by $\frac{n}{k p_n(\beta)}$ is equivalent to
\begin{align*}
	&\dfrac{\P ( \max_{i \in \beta} x\boldX_{\beta, \, i} / u_n < 1\,,  \min_{i \in \beta^c} x \boldX_{\beta, \, i} / u_n \geq 1, \max_{i \in \beta} x'\boldX_{\beta, \, i} / u_n \geq 1 )}{\P( \pi(\boldX/u_n) \in C_\beta, |\boldX| > u_n)}\\
	&= 	\dfrac{\P( \max_{i \in \beta} x\boldX_{\beta, \, i} / u_n < 1\,,  \min_{i \in \beta^c} x \boldX_{\beta, \, i} / u_n \geq 1, \max_{i \in \beta} x'\boldX_{\beta, \, i} / u_n \geq 1)}{\P( \max_{i \in \beta} \boldX_{\beta, \, i} / u_n < 1\, \min_{i \in \beta^c} \boldX_{\beta, \, i} / u_n \geq 1)}\\
	&\xrightarrow[n \to \infty]{} \dfrac{\mu_\beta(\{\boldy \in \R_+^d : \max_{i \in \beta} x\boldy_{\beta, \, i}(\beta) < 1\,, \min_{i \in \beta^c} x \boldy_{\beta, \, i}(\beta) \geq 1, \max_{i \in \beta} x'\boldy_{\beta, \, i}(\beta) \geq 1\})}{\mu_\beta(\{\boldy \in \R_+^d : \max_{i \in \beta} \boldy_{\beta, \, i}(\beta) < 1,\, \min_{i \in \beta^c} \boldy_{\beta, \, i}(\beta) \geq 1\})}\,,
\end{align*}
where we used the regular variation property of $\boldX$ on $\R^d_+ \setminus \mathbb C_\beta$. By regular variation of the function
\[
x \mapsto \dfrac{\P( \max_{i \in \beta} x\boldX_{\beta, \, i} / u_n < 1\,,  \min_{i \in \beta^c} x \boldX_{\beta, \, i} / u_n \geq 1, \max_{i \in \beta} x'\boldX_{\beta, \, i} / u_n \geq 1)}{\P( \max_{i \in \beta} \boldX_{\beta, \, i} / u_n < 1\, \min_{i \in \beta^c} \boldX_{\beta, \, i} / u_n \geq 1)}
\]
(see the arguments in Section \ref{proof-subsec:assumption1} of the Supplementary Material), the former convergence actually holds uniformly on $[\frac{1}{1+\tau},1+\tau]$, see \cite{bingham_et_al}, Section 1.2. Finally, the homogeneity property of the limit implies that the previous limit is bounded by
\begin{align*}
	&\dfrac{\mu(\{\boldy \in \R^d_+ : 1/x' \leq \max_{i \in \beta} \boldy_{\beta, \, i}(\beta) < 1/x\})}{\mu_\beta(\{\boldy \in \R_+^d : \max_{i \in \beta} \boldy_{\beta, \, i}(\beta) < 1,\, \min_{i \in \beta^c} \boldy_{\beta, \, i}(\beta) \geq 1\})}\\
	&\hspace{1cm} = \dfrac{\mu(\{\boldy \in \R^d_+ : \max_{i \in \beta} \boldy_{\beta, \, i}(\beta) \geq 1/x'\}) - \mu(\{\boldy : \max_{i \in \beta} \boldy_{\beta, \, i}(\beta) \geq 1/x\})}{\mu_\beta(\{\boldy \in \R_+^d : \max_{i \in \beta} \boldy_{\beta, \, i}(\beta) < 1,\, \min_{i \in \beta^c} \boldy_{\beta, \, i}(\beta) \geq 1\})}\\
	&\hspace{1cm} = ({x'}^{\alpha(\beta)} - x^{\alpha(\beta)}) \dfrac{\mu_\beta(\{\boldy \in \R^d_+ : \max_{i \in \beta} \boldy_{\beta,\, i}(\beta) \geq 1\})}{\mu_\beta(\{\boldy \in \R_+^d : \max_{i \in \beta} \boldy_{\beta, \, i}(\beta) < 1,\, \min_{i \in \beta^c} \boldy_{\beta, \, i}(\beta) \geq 1\})}\\
	&\hspace{1cm} =: ({x'}^{\alpha(\beta)} - x^{\alpha(\beta)}) C_1(\beta)\,.
\end{align*}

Similar arguments allow us to deal with the second term in \eqref{eq:proof_v_n_2nd_part}. This quantity multiplied by $\frac{n}{k p_n(\beta)}$ is equivalent to
\begin{align*}
	&\dfrac{\P ( \max_{i \in \beta} x\boldX_{\beta, \, i} / u_n < 1\,,  \min_{i \in \beta^c} x \boldX_{\beta, \, i} / u_n \geq 1, \min_{i \in \beta} x'\boldX_{\beta, \, i} / u_n < 1 )}{\P( \pi(\boldX/u_n) \in C_\beta, |\boldX| > u_n)}\\
	&= \dfrac{\P( \max_{i \in \beta} x\boldX_{\beta, \, i} / u_n < 1\,,  \min_{i \in \beta^c} x \boldX_{\beta, \, i} / u_n \geq 1, \min_{i \in \beta^c} x'\boldX_{\beta, \, i} / u_n < 1)}{\P( \max_{i \in \beta} \boldX_{\beta, \, i} / u_n < 1\, \min_{i \in \beta^c} \boldX_{\beta, \, i} / u_n \geq 1,)}\\
	& = \dfrac{\P( \max_{i \in \beta} x\boldX_{\beta, \, i} / u_n < 1\,,  \min_{i \in \beta^c} x \boldX_{\beta, \, i} / u_n \geq 1, \min_{i \in \beta} x'\boldX_{\beta, \, i} / u_n < 1)}{\P( \max_{i \in \beta} \boldX_{\beta, \, i} / u_n < 1\, \min_{i \in \beta^c} \boldX_{\beta, \, i} / u_n \geq 1)}\,.
\end{align*}
Assumption \ref{main-ass:exponent_measure} combined with \cite{bingham_et_al}, Section 1.2, then ensures the uniform convergence of the former term to
\[
\dfrac{\mu_\beta(\{\boldy \in \R_+^d : \max_{i \in \beta} x\boldy_{\beta, \, i}(\beta) < 1\,, \min_{i \in \beta^c} x \boldy_{\beta, \, i}(\beta) \geq 1, \min_{i \in \beta^c} x'\boldy_{\beta, \, i}(\beta) < 1\})}{\mu_\beta(\{\boldy \in \R_+^d : \max_{i \in \beta} \boldy_{\beta, \, i}(\beta) < 1,\, \min_{i \in \beta^c} \boldy_{\beta, \, i}(\beta) \geq 1\})}\,.
\]
We conclude by using the homogeneity of the measure $\mu_\beta$ which entails the following bound for the limit:
\begin{align*}
	&\dfrac{\mu_\beta(\{\boldy \in \R_+^d : 1/x \leq \min_{i \in \beta^c} \boldy_{\beta, \, i}(\beta) < 1/x'\})}{\mu_\beta(\{\boldy \in \R_+^d : \max_{i \in \beta} \boldy_{\beta, \, i}(\beta) < 1,\, \min_{i \in \beta^c} \boldy_{\beta, \, i}(\beta) \geq 1\})}\\
	&\hspace{1cm} = \dfrac{\mu(\{\boldy \in \R^d_+ : \min_{i \in \beta^c} \boldy_{\beta, \, i}(\beta) \geq 1/x\}) - \mu(\{\boldy : \max_{i \in \beta} \boldy_{\beta, \, i}(\beta) \geq 1/x'\})}{\mu_\beta(\{\boldy \in \R_+^d : \max_{i \in \beta} \boldy_{\beta, \, i}(\beta) < 1,\, \min_{i \in \beta^c} \boldy_{\beta, \, i}(\beta) \geq 1\})}\\
	&\hspace{1cm} = (x^{\alpha(\beta)} - {x'}^{\alpha(\beta)}) \dfrac{\mu_\beta(\{\boldy \in \R^d_+ : \min_{i \in \beta^c} \boldy_{\beta,\, i}(\beta) \geq 1\})}{\mu_\beta(\{\boldy \in \R_+^d : \max_{i \in \beta} \boldy_{\beta, \, i}(\beta) < 1,\, \min_{i \in \beta^c} \boldy_{\beta, \, i}(\beta) \geq 1\})}\\
	&\hspace{1cm}= (x^{\alpha(\beta)} - {x'}^{\alpha(\beta)}) C_2(\beta)\,.
\end{align*}

Similar calculations give the same bound for the term $w_n(x,x')$ in \eqref{eq:proof_decomp_vn_wn}. Going back to Equation \eqref{eq:proof_decomp_vn_wn}, we obtain
\begin{align*}
	\lim_{n \to \infty} \sup_{\rho((x,B),(x',B')) < \delta_n}
	&\E\Big[ V_n(x, \beta) - V_n(x', \beta') \Big]^2\\
	&= \lim_{n \to \infty} \sup_{\rho((x,B),(x',B)) < \delta_n} \sum_{\beta \subset B} \Big( \dfrac{n v_n(\beta, x, x')}{k p_n(\beta)} + \dfrac{n w_n(\beta, x, x')}{k p_n(\beta)} \Big)\\
	&\leq \lim_{n \to \infty} \sup_{\rho((x,B),(x',B)) < \delta_n} \sum_{\beta \subset B}  2 (C_1(\beta) + C_2(\beta)) (x^{\alpha(\beta)} \vee {x'}^{\alpha(\beta)} - x^{\alpha(\beta)} \wedge {x'}^{\alpha(\beta)})\\
	&\leq C \lim_{n \to \infty} \sup_{\rho((x,B),(x',B)) < \delta_n} (x^{\alpha(\beta)} \vee {x'}^{\alpha(\beta)} - x^{\alpha(\beta)} \wedge {x'}^{\alpha(\beta)})\\
	&\leq C \lim_{n \to \infty} \delta_n\\
	&= 0\,,
\end{align*}
where we used in the penultimate line the definition of the semi-metric $\rho$. This concludes the proof of the bracketing condition.


\noindent 2.c. \underline{The entropy condition:} By a standard inequality (see for instance Theorem 2.6.4 in \cite{vandervaart_wellner_96}), it suffices to prove that the Vapnik-Chervonenkis (VC) dimension of the family of sets
\[
\mathcal{A} = \Big\{ A(x, \mathsf{C}) : x \in \Big[\frac{1}{1+\tau}, 1 + \tau\Big], \mathsf{C} \in \mathcal{U}(\mathcal{S}_\infty)\Big\}
\]
is finite. Since $\mathcal{U}(\mathcal{S}_\infty)$ is a finite set it suffices to focus on the family $\big\{ A(x, \mathsf{C}) : x \in \big[\frac{1}{1+\tau}, 1 + \tau\big] \big\}$ for a fixed $\mathsf{C}$. Moreover, by Lemma 2.6.17 in \cite{vandervaart_wellner_96} it suffices to prove that the family $\big\{ A(x, \beta) : x \in \big[\frac{1}{1+\tau}, 1 + \tau\big] \big\}$ for a fixed $\beta \in \mathcal{S}_\infty$ has a finite VC dimension.

In virtue of the equivalence \eqref{eq:equiv_proj_C_beta} the set $A(x,\beta)$ can be rewritten as the intersection of the $|\beta|$ affine hyperplans $\{\boldy \in \R^d_+ : x\sum_{j \in \beta} y_j < x |\beta| y_i + 1\}$ for $i \in \beta$ and the $|\beta^c|$ ones $\{\boldy \in \R^d_+ : x \sum_{j \in \beta} y_j \geq x |\beta| y_i + 1\}$ for $i \in \beta^c$. It is well known that the family of affine hyperplans in $\R^d$ has VC dimension $d+1$. Using again Lemma 2.6.17 in \cite{vandervaart_wellner_96}, we obtain that the family $\big\{ A(x, \beta) : x \in \big[\frac{1}{1+\tau}, 1 + \tau\big] \big\}$ forms a class of sets with finite VC dimension. This concludes the third point of the proof.

\noindent 3. \underline{Conclusion}

We have proved that the convergence
\begin{align}\label{eq:proof_cv_gaussian_process}
	\bigg\{
	\sqrt k (\widetilde{\boldS}_{n,x})^{-1/2} \bigg( \dfrac{ T_{n}(x, \mathsf{C}) }{k \sqrt{p_n(\mathsf{C})}} - \dfrac{n \P(\boldX / u_n \in A(x,\mathsf{C})) }{k \sqrt{p_n(\mathsf{C})}} \bigg)_{\mathsf{C} \in \mathcal{U}(\mathcal{S}_\infty)} \, ; \,
	x \in [&\frac{1}{1+\tau}, 1+\tau]
	\bigg\}\\
	&\overset{d}{\to}
	(N(\mathsf{C}))_{\mathsf{C} \in \mathcal{U}(\mathcal{S}_\infty)}\,, \nonumber
\end{align}
holds in $\ell^\infty([\frac{1}{1+\tau}, 1 +\tau])$, where the limit process is constant in $x$. We identify it to a standard Gaussian random vector in $\R^{2^{s_\infty}-1}$.

Now let $1 \leq s \leq r \leq s_\infty$ and $\beta_1, \ldots, \beta_r$ distinct clusters in $\mathcal{S}_\infty$. Recall that the vector $\boldT_n^{s,r}(x) \in \R^{s+1}$ is defined by
\[
\boldT_n^{s,r}(x)
= \Big( T_n(x, \beta_1), \ldots, T_n(x, \beta_s), \sum_{j=s+1}^r T_n(x,\beta_j) \Big)^\top\,.
\]
Since the $C_\beta$'s are disjoint sets of the simplex, the last component of $\boldT_n^{s,r}(x)$ can be rewritten as
\[
\sum_{j=s+1}^r T_n(x,\beta_j) = T_n(x, \cup_{j=s+1}^r C_{\beta_j})\,.
\]
We use the convergence \eqref{eq:proof_cv_gaussian_process} with $\mathsf{C}_1 = C_{\beta_1}, \ldots, \mathsf{C}_s = C_{\beta_s}, \mathsf{C}_{s+1} = C_{\beta_{s+1}} \cup \cdots \cup C_{\beta_r}$. In this case, the matrix  $\widetilde{\boldS}_{n}(x)$ has the same asymptotic behavior as the diagonal matrix whose diagonal vector is given by
\[
\bigg(x^{\alpha(\beta_1)}, \ldots, x^{\alpha(\beta_s)},
\frac{\sum_{j=s+1}^r p_n(\beta_j) x^{\alpha(\beta_j)} }{ \sum_{j=s+1}^r p_n(\beta_j) }\bigg)^\top\,.
\]
Then the reformulation of the convergence in \eqref{eq:proof_cv_gaussian_process} becomes
\[
	\bigg\{
	\sqrt k (\boldP^{s,r}_{n}(x))^{-1/2} \bigg( \dfrac{ \boldT_{n}^{s,r}(x) }{k} - \E\Big[ \dfrac{ \boldT_{n}^{s,r}(x) }{k}  \Big] \bigg) ; \,
	(1+\tau)^{-1} \leq x \leq 1+\tau
	\bigg\}_{s<r}
	\overset{d}{\to}
	(\boldN^{s,r})_{s<r}\,,
\]
in $\ell^\infty([\frac{1}{1+\tau}, 1 +\tau])$, where the constant limit process is identified to
\[
\boldN^{s,r} = (N(\mathsf{C}_1), \ldots, N(\mathsf{C}_{s+1}))^\top\,,
\]
and where
\[
\boldP^{s,r}_n(x)
=
\Big(
x^{\alpha(\beta_1)} p_n(\beta_1), \ldots, x^{\alpha(\beta_s)} p_n(\beta_s), \sum_{j=s+1}^r x^{\alpha(\beta_j)} p_n(\beta_j)
\Big)^\top \in \R^{s+1}\,.
\]
This concludes the proof of the convergence \eqref{main-eq:CLT_Tnx} of Theorem \ref{main-theo:CLT}.

Finally, in order to prove the convergence in $\ell^\infty([\frac{1}{1+\tau}, 1+\tau])$ of the process
\begin{equation}\label{eq:proof-CLT_bias}
	\bigg\{
	\sqrt k \Diag(\boldP^{s,r}_n(x))^{-1/2} \bigg( \dfrac{ \boldT_{n}^{s,r}(x) }{k} -\boldP^{s,r}_n(x) \bigg)\, ; \,
	(1+\tau)^{-1} \leq x \leq 1+\tau
	\bigg\}_{s<r}
	\overset{d}{\to}
	(\boldN^{s,r})_{s<r}\,,
\end{equation}
it suffices to prove that
\[
	\sup_{x \in [\frac{1}{1+\tau}, 1+\tau]}
	\sqrt k (\boldP^{s,r}_{n}(x))^{-1/2}
	\Big|\E\Big[ \dfrac{ \boldT_{n}^{s,r}(x) }{k} \Big]
	- \boldP^{s,r}_n(x) \Big|
	\to 0\,.
\]
The last convergence holds true if for any $\beta \in \mathcal{S}_\infty$ we have
\[
\sup_{x \in [\frac{1}{1+\tau}, 1+\tau]}
\sqrt{\dfrac{k}{p_n(\beta) x^{\alpha(\beta)}}}
\bigg| \dfrac{n \P(\boldX / u_n \in A(x,\beta))}{k}
- x^{\alpha(\beta)} p_n(\beta) \bigg|
\to 0\,.
\]
This is clearly implied by the bias assumption \eqref{main-eq:bias}.
\subsection{Proof of Proposition \ref{main-prop:CLT_X(k)}}
Recall that we work with a level $(k_n) \in K$ so that the identity $\mathcal {S}_\infty = \widehat{\mathcal{S}}_n$ holds true almost surely for $n$ large enough. Under this condition we get
\[
\sum_{\beta \in \mathcal{S}_\infty} T_{n}(x, \beta) = \sum_{j=1}^n \indic\{ x|\boldX_j| > u_n\}\,.
\]
Then the convergence \eqref{eq:proof_cv_gaussian_process} with $\mathsf{C} = \cup_{\beta \in \mathcal{S}_\infty} C_\beta$ gives
\begin{align*}
	\bigg\{
	\sqrt{k} \dfrac{ \frac{\sum_{j=1}^n \indic\{ x|\boldX_j| > u_n\}}{k} - \sum_{\beta \in \mathcal{S}_\infty} x^{\alpha(\beta)} p_n(\beta)}{\sqrt{\sum_{\beta \in \mathcal{S}_\infty} x^{\alpha(\beta)} p_n(\beta)}}
	\, ; \,
	(1+\tau)^{-1} \leq x \leq 1+\tau
	\bigg\}
	\overset{d}{\to}
	\boldN^{0, s_\infty}\,,
\end{align*}
in $\ell^\infty([\frac{1}{1+\tau}, 1+\tau])$, with $\boldN^{0, s_\infty}$ a standard Gaussian random variable. Since $p^*(\beta)$ is null for $\beta \notin \mathcal{S}^*(\boldZ)$ and $\alpha(\beta) = \alpha$ for $\beta \in \mathcal{S}^*(\boldZ)$ we obtain that
\[
\sum_{\beta \in \mathcal{S}_\infty} x^{\alpha(\beta)} p^*(\beta) =  x^\alpha \sum_{\beta \in \mathcal{S}^*(\boldZ)} p^*(\beta) = x^\alpha\,.
\]
After multiplying both sides of the previous convergence by $x^{\alpha/2}$ and using Slutsky's lemma we get
\[
\sqrt{k} \bigg(\frac{\sum_{j=1}^n \indic\{ x|\boldX_j| > u_n\}}{k} - \sum_{\beta \in \mathcal{S}_\infty} x^{\alpha(\beta)} p_n(\beta) \bigg)
\overset{d}{\to}
x^{\alpha/2} \boldN^{0, s_\infty}\,.
\]
The bias assumption \eqref{main-eq:bias_assumption} allows one to replace the centering term by $\sum_{\beta \in \mathcal{S}_\infty} x^{\alpha(\beta)} p^*(\beta)= x^\alpha$. Applying Vervaat's Lemma (see for instance Lemma A.0.2 in \cite{dehaan_ferreira}) to the nondecreasing function $f_n(x) = k^{-1} \sum_{j=1}^n \indic\{ x|\boldX_j| > u_n\}$ we obtain
\[
\sqrt{k} \bigg( \dfrac{u_n}{|\boldX_{(\lceil x k \rceil + 1)}|}  - x^{1/\alpha} \bigg)
\stackrel{d}{\to} 
-\dfrac{1}{\alpha} x^{1/\alpha - 1} x^{1/2} \boldN^{0,s_\infty}\,,
\quad n \to \infty\,,
\]
where we used that $f_n^\leftarrow(x) = |\boldX|_{(\lceil x k \rceil + 1)} / u_n$. We fix $x=1$ in the previous convergence and use the Delta-method with $g(t) = t^{\alpha(\beta)}$ for a fixed $\beta \in \mathcal{S}_\infty$ which gives the convergence
\[
\sqrt{k} \bigg\{ \bigg(\dfrac{u_n}{|\boldX_{(k + 1)}|}\bigg)^{\alpha(\beta)}  - 1 \bigg\}
\stackrel{d}{\to}
- \dfrac{\alpha(\beta)}{\alpha}\boldN^{0,s_\infty}\,,
\quad n \to \infty\,.
\]
The relation $\alpha(\beta) = \alpha$ for any $\beta \in  \mathcal{S}^*(\boldZ)$ then yields
\begin{equation}\label{eq:cv_order_stat_distinction}
	\sqrt k \bigg\{ \bigg( \dfrac{u_n}{|\boldX_{(k+1)}|} \bigg)^{\alpha(\beta)} - 1 \bigg\} \sqrt{p_n(\beta)}
	\stackrel{d}{\to} 
	\left\{
	\begin{array}{ll}
		- \sqrt{p^*(\beta)} \boldN^{0,s_\infty}\,,& \mbox{for } \beta \in \mathcal{S}(\boldZ)\,,\\
		0\,, & \mbox{for } \beta \in \mathcal{S}_\infty \setminus \mathcal{S}(\boldZ)\,.
	\end{array}
	\right.
\end{equation}

Now we start from the decomposition
\begin{align}
	&\sqrt k (\boldP^{s,r}_{n}(1))^{-1/2} \bigg( \dfrac{ \boldT_{n}^{s,r}(u_n/|\boldX_{(k+1)}|) }{k} - \boldP_n^{s,r}(1) \bigg) \label{eq:decomposition_limit_X_k}\\
	&=
	(\boldP^{s,r}_{n}(1))^{-1/2} (\boldP^{s,r}_{n}(u_n/|\boldX_{(k+1)}|))^{1/2}
	\bigg\{
	\sqrt k 
	(\boldP^{s,r}_{n}(x))^{-1/2} \bigg( \dfrac{ \boldT_{n}^{s,r}(x) }{k} - \boldP_n^{s,r}(x) \bigg)
	\bigg\} \bigg|_{x = u_n / |\boldX_{(k+1)}|} \nonumber\\
	&\hspace{1cm}+ \sqrt k (\boldP_n^{s,r}(1))^{-1/2} \Big( \boldP_n^{s,r}(u_n/|\boldX_{(k+1)}|) - \boldP^{s,r}_n(1) \Big)\,. \nonumber
\end{align}
Since $u_n/|\boldX_{(k+1)}| \to 1$ in probability, the uniform convergence in \eqref{eq:proof-CLT_bias} combined with Slutsky's lemma ensures that the first term converges in distribution to $\boldN^{s,r}$. Regarding the second term Equation \eqref{eq:cv_order_stat_distinction} ensures that it converges to $- \sqrt{\boldP^{s,r}} \boldN^{0, s_\infty}$ where we recall that
\[
\boldP^{s,r}
= \Big(p^*(\beta_1), \ldots, p^*(\beta_s), \sum_{j=s+1}^r p^*(\beta_j)\Big)^\top\,.
\]
The covariance matrix of $\boldN^{s,r} - \sqrt{\boldP^{s,r}} \boldN^{0, s_\infty}$ can be decomposed as follows:
\begin{align}
\Sigma(\boldN^{s,r} &- \sqrt{\boldP^{s,r}} \boldN^{0, s_\infty}) \label{eq:decomposition_cov_mat}\\
&= \Sigma(\boldN^{s,r})
- \Sigma(\boldN^{s,r}, \sqrt{\boldP^{s,r}} \boldN^{0, s_\infty})
- \Sigma(\sqrt{\boldP^{s,r}} \boldN^{0, s_\infty}, \boldN^{s,r})
+ \Sigma(\sqrt{\boldP^{s,r}} \boldN^{0, s_\infty}) \nonumber\\
&= Id_{s+1}
- \Sigma(\boldN^{s,r}, \sqrt{\boldP^{s,r}} \boldN^{0, s_\infty})
- \Sigma(\sqrt{\boldP^{s,r}} \boldN^{0, s_\infty}, \boldN^{s,r})
+ \sqrt{\boldP^{s,r}} \sqrt{\boldP^{s,r}}^\top\,, \nonumber
\end{align}
where we used the fact that $\boldN^{s,r}$ (resp. $\boldN^{0, s_\infty}$) is a centered Gaussian random vector (resp. variable). Here we have denoted by $\Sigma(\boldX, \boldY)$ the matrix $\E[\boldX \cdot \boldY^\top]$ for any centered random vectors $\boldX$ and $\boldY$. Regarding the middle terms, we start from the covariance matrix defined in \eqref{eq:covariance_matrix_main_theorem} with $x=1$ and restricted to the sets
\[
\mathsf{C}_1 = C_{\beta_1}, \ldots, \mathsf{C}_s = C_{\beta_s}, \mathsf{C}_{s+1} = C_{\beta_{s+1}} \cup \cdots \cup C_{\beta_r}, \mathsf{C}_{s+2} = \cup_{\beta \in \mathcal{S}_\infty} C_\beta\,.
\]
This gives the following covariance matrix in $\R^{s+2} \times \R^{s+2}$:
\[
\begin{pmatrix}
	1 & b_{1,2}^n & \cdots & b_{1, s+1}^n & \sqrt \frac{p_n(\beta_1)}{\sum_{\beta \in \mathcal{S_\infty}} p_n(\beta)}\\
	\vdots & \ddots &  & & \vdots\\
	b_{s+1,1}^n & \cdots & b_{s, s+1}^n & 1 & \sqrt \frac{\sum_{j=s+1}^r p_n(\beta_j)}{\sum_{\beta \in \mathcal{S_\infty}} p_n(\beta)}\\
	 \sqrt \frac{p_n(\beta_1)}{\sum_{\beta \in \mathcal{S_\infty}} p_n(\beta)} & \cdots & \cdots & \sqrt \frac{\sum_{j=s+1}^r p_n(\beta_j)}{\sum_{\beta \in \mathcal{S_\infty}} p_n(\beta)} & 1\\
\end{pmatrix}\,,
\]
with
\[
b_{i,j}^n = -\dfrac{k}{n} \sqrt{p_n(\mathsf{C}_i) p_n(\mathsf{C}_j)}\,.
\]
It converges to the covariance matrix of the vector $({\boldN^{s,r}}^\top, \boldN^{0, s_\infty})^\top$:
\[
\Sigma(({\boldN^{s,r}}^\top, \boldN^{0, s_\infty})^\top)
=
\begin{pmatrix}
	1 & 0 & \cdots & 0 & \sqrt{p^*(\beta_1)}\\
	\vdots & \ddots & $ $ & \vdots & \vdots\\
	0 & \cdots & 0 & 1 & \sqrt{\sum_{j=s+1}^r p^*(\beta_j)}\\
	\sqrt{p^*(\beta_1)} & \cdots & \cdots & \sqrt {\sum_{j=s+1}^r p^*(\beta_j)} & 1\\
\end{pmatrix}\,.
\]
We deduce that
\[
\Sigma(\boldN^{s,r}, \sqrt {\boldP^{s,r}} \boldN^{0, s_\infty})
= \Sigma(\boldN^{s,r},\boldN^{0, s_\infty}) \sqrt {\boldP^{s,r}}^\top
= \sqrt {\boldP^{s,r}} \sqrt {\boldP^{s,r}}^\top\,.
\]
Moving back to \eqref{eq:decomposition_cov_mat}, we conclude that the covariance matrix of $\boldN^{s,r} - \sqrt{\boldP^{s,r}} \boldN^{0, s_\infty}$ is equal to $Id_{s+1} - \sqrt {\boldP^{s,r}} \sqrt {\boldP^{s,r}}^\top$.

All in all the decomposition in \eqref{eq:decomposition_limit_X_k} leads to the convergence
\begin{align*}
	\sqrt k &(\boldSigma_n(1))^{-1/2} \bigg( \dfrac{ \boldT_{n}^{s,r}(u_n/|\boldX_{(k+1)}|) }{k} - \boldP_n^{s,r}(1) \bigg)
	\to
	(Id_{s+1} - \sqrt {\boldP^{s,r}} \sqrt {\boldP^{s,r}}^\top) \boldN\,, \quad n \to \infty\,,
\end{align*}
with $\boldN$ a standard centered Gaussian vector.

\subsection{Proof of Remark \ref{main-rem:bias_assumption}}
Section \ref{proof-subsec:assumption1} ensures that the function $t \mapsto \P(\boldX/t \in A(1,\beta))$ is regularly varying with index $\alpha(\beta)$. It follows that $\P(\boldX/u_n \in A(1,\beta)) = u_n^{-\alpha(\beta)} l(u_n)$, where $l$ is a slowly varying function, ie $l(tx)/l(t) \to 1$ as $t \to \infty$, for all $x > 0$. Combining this relation with the asymptotic expansion $k \sim n \P(|\boldX| > u_n) \sim n u_n^{-\alpha}$, we get 
\[
p_n(\beta)
= \dfrac{n \P(\boldX/u_n \in A(1,\beta))}{n \P(|\boldX| > u_n)}
\sim \dfrac{n u_n^{-\alpha(\beta)} l(u_n)}{k}
\sim \dfrac{n}{k} \Big(\dfrac{k}{n}\Big)^{\alpha(\beta)/\alpha} l(u_n)
\]
This leads to the following equivalences, as $n \to \infty$,
\[
\sqrt k p_n(\beta) \to 0
\iff \sqrt k \dfrac{k^{\frac{\alpha(\beta)-\alpha}{\alpha}}}{n^{\frac{\alpha(\beta)-\alpha}{\alpha}}} l(u_n) \to 0
\iff \dfrac{k^{\frac{2\alpha(\beta)-\alpha}{2\alpha}}}{n^{\frac{\alpha(\beta)-\alpha}{\alpha}}} l(u_n) \to 0\,.
\]
After raising the latter expression to the power $2 \alpha / (2\alpha(\beta) - \alpha)$, we get
\[
\sqrt k p_n(\beta) \to 0
\iff \dfrac{k}{n^{\frac{2(\alpha(\beta)-\alpha)}{2\alpha(\beta) - \alpha}}} l(u_n)^{\frac{2 \alpha}{2 \alpha(\beta) - \alpha}} \to 0\,.
\]
Since $l$ is slowly varuing we have $(u_n)^{-\epsilon}l(u_n) \to 0$ for any $\epsilon > 0$. Therefore, if $\kappa>\frac{2(\alpha(\beta) - \alpha)}{2 \alpha(\beta) - \alpha}$ then the assumption $k = o(n^{\kappa})$ implies that $\sqrt k p_n(\beta) \to 0$.

\subsection{Proof of Theorem \ref{main-theo:cv_KL_chi_square}} \label{app:subsec:bias_selection}

In order to prove the convergence
\[
\E[KL(\boldP_k \| \boldM_k )|_{\boldp = \hat \boldp_k}]
- \E \big[ \log L_{\boldP_k} (\boldT_{n,k}) \big]
+ \E[\log L_{\boldM_k} (\hat \boldp ; \boldT_{n,k})]
\to s\,,
\]
we prove that
\[
\E[KL(\boldP_k(x) \| \boldM_k )|_{\boldp = \hat \boldp_k(x)}]
- \E \big[ \log L_{\boldP_k(x)} (\boldT_{n}(x)) \big]
+ \E[\log L_{\boldM_k} (\hat \boldp(x) ; \boldT_{n}(x))]
= \E[\psi_n(x)]\,,
\]
where $\boldT_n(x)$ denotes the vector of $\R^{2^d-1}$ with ordered components
\[
T_{n,i}(x) = \sum_{j=1}^n \indic_{\{\boldX/u_n \in A(x,\beta_i)\}}\,,
\]
$\boldP_k(x)$ denotes its distribution, $\hat \boldp(x)$ the maximum likelihood estimator of the model $\boldM_k$ for the vector $\boldT_n(x)$, and $\psi_n(x)$ satisfies
\[
\psi_n(x) \stackrel{d}{\to} \chi^2(s)\,, \quad n \to \infty\,.
\]
uniformly over $x \in [\frac{1}{1+\tau}, 1+\tau]$.

We recall the expression of the log-likelihood $\log L_{\boldM_k}$ as a function of $\boldp := (p_1,\ldots,p_s)^\top \in \B^s_+(0,1)$:
\begin{align*}
\log L_{\boldM_k} (\boldp ; \boldT_n(x))
= \log(k!) - \sum_{i=1}^{2^d-1} \log(T_{n, i}(x)!)
&+ \sum_{i=1}^{s} T_{n, i}(x) \log( p_i)\\
&+ \Big(\sum_{i=s + 1}^{r} T_{n, i}(x)\Big) \log \Big(\frac{1-\sum_{j=1}^s p_j}{r-s} \Big)\,.
\end{align*}
The expectation of $T_{n,j}(x)$ is given by $\E[T_{n,j}(x)] = p_{n,j}(x) a_n(x)$, where
\[
p_{n,j}(x) = \P(\pi(x \boldX/u_n) \in C_{\beta_j} \mid x |\boldX| > u_n)\,,
\]
and
\[
a_n(x) = n \P(x|\boldX| > u_n) =n  \P(|\boldX| > u_n) \dfrac{\P(x|\boldX| > u_n)}{\P(|\boldX| > u_n)} \sim k x^\alpha\,, \quad n \to \infty\,.
\]
This gives the following expression for the expectation of the log-likelihood:
\begin{equation}\label{eq:log_likelihood_k_esp}
	\begin{aligned}
		\E\big[\log L_{\boldM_k} (\boldp ; \boldT_n(x))\big]
		=
		\log(k!) - \E\bigg[&\sum_{i=1}^{2^d-1} \log(T_{n, i}(x)!)\bigg]
		+  a_n(x) \sum_{i=1}^{s} p_{n, i}(x) \log( p_i)\\
		&+ a_n(x) \Big(\sum_{i=s + 1}^{r} p_{n, i}(x)\Big) \log\Big(\frac{1-\sum_{j=1}^s p_j}{r-s}\Big)\,.
	\end{aligned}
\end{equation}

A similar computation as for the maximum likelihood estimator entails that this expectation is maximized for
\[
\boldp
= \tilde \boldp(x)
:= \Big( \frac{p_{n, 1}(x)}{\sum_{j=1}^{r} p_{n, j}(x)}, \ldots, \frac{p_{n,s}(x)}{\sum_{j=1}^{r} p_{n, j}(x)}\Big)^\top \in \B^s_+(0,1)\,.
\]
Note that since $n$ is large enough so that the event ${\cal S}_\infty = \widehat{\cal S}_n$ holds true almost surely, we have $r = \hat s_n \geq s^*$ and thus $\sum_{j=1}^{\hat s_n} p_{n, j} \geq \sum_{j=1}^{s^*} p_{n, j} \to 1$. This means that asymptotically we have $\tilde p_{j}(x) \sim p_{n,j}(x)$.

The first step of the bias selection consists in establishing a Taylor expansion for the estimator
\[
KL (\boldP_k(x) \| \boldM_k)|_{\boldp = \hat \boldp(x)}
= \E \big[ \log L_{\boldP_k(x)}(\boldT_n(x)) \big]
- \E \big[ \log L_{\boldM_k} (\boldp ; \boldT_n(x)) \big]|_{\boldp = \hat \boldp(x)}\,.
\]

\begin{lem}\label{lem:asymptotic_expansion_hat_p}
	There exists $c_{n,1}(x) \in (0,1)$ such that
	\begin{align}\label{eq:kullback_leibler_approximation_expectation_hat_p}
	KL (\boldP_k&(x) \| \boldM_k )|_{\boldp =\hat \boldp(x)}
	= KL (\boldP_k(x) \| \boldM_k )|_{\boldp = \tilde \boldp(x)}\\
	&+\frac{1}{2} (\hat \boldp(x) - \tilde \boldp(x))^\top
	\frac{\partial^2}{\partial \boldp^2} \E \big[ - \log L_{\boldM_k} (\boldp, \boldT_n(x)) \big]\big|_{c_{n,1}(x) \hat \boldp(x) + (1-c_{n,1}(x)) \tilde \boldp(x)} (\hat \boldp(x) - \tilde \boldp(x))\,. \nonumber
	\end{align}
\end{lem}

Since the quantity $\tilde \boldp(x)$ is deterministic, the first term of the right-hand side in \eqref{eq:kullback_leibler_approximation_expectation_hat_p} can be written as
\begin{equation}\label{eq:Taylor_KL_tilde_p}
KL (\boldP_k(x) \| \boldM_k )|_{\boldp = \tilde \boldp(x)}
= \E \big[ \log L_{\boldP_k(x)}(\boldT_n(x)) \big]
- \E \big[ \log L_{\boldM_k} (\tilde \boldp(x) ; \boldT_n(x)) \big] \, . 
\end{equation}
The idea is then to provide a Taylor expansion of $\log L_{\boldM_k} (\tilde \boldp(x) ; \boldT_n(x))$ around the vector $\hat \boldp(x)$. This is the purpose of the following lemma.

\begin{lem}\label{lem:asymptotic_expansion_p*}
	There exists $c_{n,2}(x) \in (0,1)$ such that
	\begin{align}\label{eq:kullback_leibler_approximation_hat_p}
	\log &L_{\boldM_k} (\tilde \boldp(x) ; \boldT_n(x))
	= \log L_{\boldM_k} (\hat \boldp(x) ; \boldT_n(x))\\
	&+ \frac{1}{2} (\tilde \boldp(x) - \hat \boldp(x))^\top
	\frac{\partial^2}{\partial \boldp^2} \log L_{\boldM_k} (c_{n,2}(x) \tilde \boldp(x) + (1-c_{n,2}(x)) \hat \boldp(x) ; \boldT_n(x)) (\tilde \boldp(x) - \hat \boldp(x))\,. \nonumber
	\end{align}
\end{lem}

Lemmas \ref{lem:asymptotic_expansion_hat_p} and \ref{lem:asymptotic_expansion_p*} are a consequence of the following result known as ``Cauchy's Mean-Value Theorem" (see \cite{hille_1964} for a proof).

\begin{lem}\label{lem:univariate_mean_value_theorem}
	Let $f$ and $g$ be two continuous functions on the closed interval $[a,b]$, $a < b$, and differentiable on the open interval $(a,b)$. Then there exists some $c \in (a,b)$ such that
	\[
	( f(b) - f(a) ) g'(c) = ( g(b) - g(a) ) f'(c) \, .
	\]
\end{lem}

\begin{proof}[Proof of Lemma \ref{lem:asymptotic_expansion_hat_p}]
	Let $f$ be the function defined as $f(t) = h(t \hat \boldp(x) + (1-t) \tilde \boldp(x))$ for $t \in [0,1]$, where $h$ is defined as
	\[
	h(\boldp) =
	KL (\boldP_k(x) \| \boldM_k)
	+\frac{\partial}{\partial \boldp} KL (\boldP_k(x) \| \boldM_k ) (\hat \boldp(x) - \boldp(x)) \, .
	\]
	Some short calculations give the following relations:
	\begin{align*}
	f(1) &= h(\hat \boldp(x))
	= KL (\boldP_k(x) \| \boldM_k )|_{\boldp = \hat \boldp(x)}\,, \\
	f(0) &= h(\tilde \boldp(x))
	= KL (\boldP_k(x) \| \boldM_k)|_{\boldp = \tilde \boldp(x)}
	+\frac{\partial}{\partial \boldp} KL (\boldP_k(x) \| \boldM_k )|_{\boldp = \tilde \boldp(x)} (\hat \boldp(x) - \tilde \boldp(x))\\
	&= KL (\boldP_k(x) \| \boldM_k )|_{\boldp = \tilde \boldp(x)} 
	- \underbrace{\frac{\partial}{\partial \boldp} \E \big[ \log L_{\boldM_k} (\boldp ; \boldT_n(x)) \big]\big|_{\boldp = \tilde \boldp(x)}}_{ = 0 \text{ by definition of } \tilde \boldp(x)} (\hat \boldp(x) - \tilde \boldp(x))\\
	&= KL (\boldP_k \| \boldM_k )|_{\boldp = \tilde \boldp(x)}\, ,\\
	f'(t) &= \frac{\partial h}{\partial \boldp}(t \hat \boldp(x) + (1-t) \tilde \boldp(x)) (\hat \boldp(x) - \tilde \boldp(x))\\
	& = (\hat \boldp(x) - [t \hat \boldp(x) + (1-t) \tilde \boldp(x)])^\top \frac{\partial^2}{\partial \boldp^2} KL (\boldP_k(x) \| \boldM_k )|_{t \hat \boldp(x) + (1-t) \tilde \boldp(x)} (\hat \boldp(x) - \tilde \boldp(x))\\
	& = (1-t) (\hat \boldp(x) - \tilde \boldp(x))^\top \frac{\partial^2}{\partial \boldp^2}
	\E \big[ - \log L_{\boldM_k} (\boldp ; \boldT_n(x)) \big]\big|_{t \hat \boldp(x) + (1-t) \tilde \boldp(x)} (\hat \boldp(x) - \tilde \boldp(x)) \, .
	\end{align*}
	We apply Lemma \ref{lem:univariate_mean_value_theorem} to the functions $f$ and $g : t \mapsto (t-1)^2$. There exists $c_{n,1}(x) \in (0,1)$ such that $( f(1) - f(0) ) g'(c_{n,1}(x)) = ( g(1) - g(0) ) f'(c_{n,1}(x))$, i.e.
	\begin{align*}
	KL (\boldP_k(x) &\| \boldM_k)|_{\boldp = \hat \boldp(x)}
	- KL (\boldP_k \| \boldM_k )|_{\boldp = \tilde \boldp(x)}\\
	&= -\frac{1}{2} (\hat \boldp(x) - \tilde \boldp(x))^\top \frac{\partial^2}{\partial \boldp^2}
	\E \big[ \log L_{\boldM_k} (\boldT_n(x)) \big]\big|_{c_{n,1}(x) \hat \boldp(x) + (1-c_{n,1}(x)) \tilde \boldp(x)} (\hat \boldp(x) - \tilde \boldp(x)) \, .
	\end{align*}
	This concludes the proof.
\end{proof}

\begin{proof}[Proof of Lemma \ref{lem:asymptotic_expansion_p*}]
	We consider $f(t) = h(t \tilde \boldp(x) + (1-t) \hat \boldp(x))$, for $t \in [0,1]$ where $h$ is defined as
	\[
	h(\boldp) =
	\log L_{\boldM_k} (\boldp ; \boldT_n(x))
	+\frac{\partial}{\partial \boldp} \log L_{\boldM_k} (\boldp ; \boldT_n(x)) (\tilde \boldp(x) - \boldp(x)) \, .
	\]
	After some calculations we obtain
	\begin{align*}
	&f(1) = h(\tilde \boldp(x)) = \log L_{\boldM_k} ( \tilde \boldp(x) ; \boldT_n(x))\,, \\
	&f(0)= h(\hat \boldp(x)) = \log L_{\boldM_k} (\hat \boldp(x) ; \boldT_n(x))
	+ \underbrace{\frac{\partial}{\partial \boldp} \log L_{\boldM_k} (\hat \boldp(x) ; \boldT_n(x))}_{= 0 \text{ by definition of } \hat \boldp(x) } (\tilde \boldp(x) - \hat \boldp(x))\,,\\
	&f'(t) = \frac{\partial h}{\partial \boldp}(t \tilde \boldp(x) + (1-t) \hat \boldp(x)) (\tilde \boldp(x) - \hat \boldp(x))\\
	& = (\tilde \boldp(x) - [t \tilde \boldp(x) + (1-t) \hat \boldp(x)])^\top \frac{\partial^2}{\partial \boldp^2} \log L_{\boldM_k} (t \tilde \boldp(x) + (1-t) \hat \boldp(x) ; \boldT_n(x)) (\tilde \boldp(x) - \hat \boldp(x))\\
	& = (1-t) (\tilde \boldp(x) - \hat \boldp(x))^\top \frac{\partial^2}{\partial \boldp^2} \log L_{\boldM_k} (t \tilde \boldp(x) + (1-t) \hat \boldp(x) ; \boldT_n(x)) (\tilde \boldp(x) - \hat \boldp(x))\,.
	\end{align*}
	We apply Lemma \ref{lem:univariate_mean_value_theorem} to the functions $f$ and $g : t \mapsto (t-1)^2$. There exists $c_n^2(x) \in (0,1)$ such that $(f(1) - f(0)) g'(1-c_{n,2}(x)) = (g(1) - g(0)) f'(1-c_{n,2}(x))$, i.e.
	\begin{align*}
	\log &L_{\boldM_k} (\tilde \boldp(x) ; \boldT_n(x))
	- \log L_{\boldM_k} (\hat \boldp(x) ; \boldT_n(x))\\
	&= \frac{1}{2} (\tilde \boldp(x) - \hat \boldp(x))^\top \frac{\partial^2}{\partial \boldp^2} \log L_{\boldM_k} (c_{n,2}(x) \hat \boldp(x) + (1-c_{n,2}(x)) \tilde \boldp(x) ; \boldT_n(x)) (\tilde \boldp(x) - \hat \boldp(x)) \,.
	\end{align*}
	This concludes the proof.
\end{proof}

Combining Equations \eqref{eq:kullback_leibler_approximation_expectation_hat_p}, \eqref{eq:Taylor_KL_tilde_p}, and \eqref{eq:kullback_leibler_approximation_hat_p}, we obtain the following expression for the KL between $\boldP_k(x)$ and $\boldM_k$ evaluated at $\boldp = \hat \boldp(x)$:
\[
\E[KL (\boldP_k(x) \| \boldM_k )|_{\boldp = \hat \boldp(x)}]
= \E \big[ \log L_{\boldP_k(x)} (\boldT_n(x)) \big]
+ \E[-\log L_{\boldM_k} (\hat \boldp ; \boldT_n(x)]
+ \E[\psi_n(x)]
\]
with $\psi_n(x)$ given by
\begin{align*}
\psi_n(x)
=
&- \frac{1}{2} (\hat \boldp(x) - \tilde \boldp(x))^\top 
\frac{\partial^2}{\partial \boldp^2} \E \Big[\log L_{\boldM_k}(\boldp ; \boldT_n(x)) \Big] \Big|_{c_{n,1}(x) \hat \boldp(x) + (1-c_{n,1}(x)) \tilde \boldp(x)} (\hat \boldp(x) - \tilde \boldp(x))^\top\\
&- \frac{1}{2} (\hat \boldp(x) - \tilde \boldp(x))^\top 
\frac{\partial^2}{\partial \boldp^2} \log L_{\boldM_k}(\boldp ; \boldT_n(x)) \Big|_{c_{n,2}(x) \hat \boldp(x) + (1-c_{n,2}(x)) \tilde \boldp(x)} (\hat \boldp(x) - \tilde \boldp(x))^\top\,.
\end{align*}

In order to prove Theorem \ref{main-theo:cv_KL_chi_square} we establish the uniform convergence of $\psi_n(x)$ to a constant process following a chi-square distribution with $s$ degrees of freedom.

\begin{lem}\label{lem:conv_bilinear_form}
	Under the assumptions of Proposition \ref{main-prop:CLT_X(k)} the following convergence holds in $\ell^\infty([\frac{1}{1+\tau}, 1+\tau])$:
	\[
	\psi_n(x)
	\overset{d}{\to} \psi\,, \quad n \to \infty\,.
	\]
	where $\psi$ is a constant process over $x$ following a chi-square distribution with $s$ degrees of freedom.
\end{lem}

\begin{proof}[Proof of Lemma \ref{lem:conv_bilinear_form}]
	We start from the second order derivative of the log-likelihood $\log L_{\boldM_k}$,
	\[
	\frac{\partial^2}{\partial \boldp^2} \log L_{\boldM_k}(\boldp ; \boldT_n(x)) =  - \Diag \Big(\frac{T_{n, 1}(x)}{p_1^2}, \ldots, \frac{T_{n, s}(x)}{p_s^2} \Big) - \frac{\sum_{i=s+1}^{r} T_{n,i}(x) }{(1-\sum_{i=1}^s p_i)^2} \boldone \cdot \boldone^\top\,,
	\]
	where $\boldone = (1, \ldots, 1)^\top \in \R^s$. This gives the expression
	\begin{equation}\label{eq:psi_n}
	\psi_n(x)
	= k \sum_{j=1}^{s} \frac{(T_{n,j}(x)/k - \tilde p_{j}(x))^2}{p_{n, j}(x)} A_{n, j}(x)
	+ \Big( \frac{\sum_{j=s+1}^{r} (T_{n,j}(x) / k - \tilde p_{j}(x))}{\sum_{j=s+1}^{r} p_{n, j}(x)} \Big)^2 B_n(x)\,,
	\end{equation}
	where $A_{n, j} (x)$ corresponds to
	\[
	\dfrac{p_{n, j}(x) T_{n,j}(x) / k}{2(c_{n,2}(x) T_{n,j}(x)/k + (1-c_{n,2}(x)) \tilde p_{j}(x))^2} + \dfrac{p_{n,j}(x)^2}{2(c_{n,1}(x) T_{n,j}(x)/k + (1-c_{n,1}(x)) \tilde p_{j}(x))^2}\,,
	\]
	and 
	\begin{align*}
	B_n(x)
	&= \dfrac{\sum_{j=s+1}^{r} p_{n,j}(x)  \sum_{j=s+1}^{r} T_{n,j}(x)/k }{ 2(c_{n,2}(x) \sum_{j=s+1}^{r} T_{n,j}(x)/k + (1-c_{n,2}(x)) \sum_{j=s+1}^r \tilde p_{j}(x))^2 }\\
	& \hspace{3cm} + \dfrac{(\sum_{j=s+1}^{r} p_{n,j}(x))^2 }{ 2(c_{n,1}(x) \sum_{j=s+1}^{r}T_{n,j}(x)/k + (1-c_{n,1}(x)) \sum_{j=s+1}^r \tilde p_{j}(x))^2 }\,.
	\end{align*}
	Recall that $\tilde p_{j}(x) \sim p_{n, j}(x)$. Then, by Slutsky's lemma and convergence \eqref{main-eq:conv_chi_squared} in the main text it suffices to prove that $A_{n, j}$ and $B_n$ converges to $1$ in probability uniformly in $x$.
	
	The uniform convergence established in Theorem \ref{main-theo:CLT} entails that for any $j$ such that $\beta_j \in {\cal S}_\infty$ we have the convergence in probability
	\begin{equation}\label{eq:conv_functions_psi_j}
	\Big| \frac{T_{n, j}(x) }{k p_{n, j}(x)} - 1 \Big|
	= \frac{1}{\sqrt {k p_{n,j}(x)}}  \Big| \sqrt{k} \frac{T_{n, j}(x) / k - p_{n, j}(x)}{\sqrt{p_{n,j}(x)}} \Big|
	\to 0\,,
	\end{equation}
	uniformly over $x \in [\frac{1}{1+\tau}, 1 +\tau]$.
	Therefore, for any $c_n(x) \in (0,1)$ we have
	\begin{align}
	&\Big| \frac{c_n(x) T_{n, j}(x) / k + (1-c_n(x)) \tilde p_{j}(x)}{p_{n, j}(x)} - 1 \Big| \nonumber\\
	&\hspace{2cm} = \Big| c_n(x) \Big(\frac{T_{n, j}(x) }{k p_{n, j}(x)} - 1\Big) + (1-c_n(x)) \Big(\frac{\tilde p_{j}(x)}{p_{n, j}(x)} - 1\Big)\Big| \nonumber\\
	&\hspace{2cm} \leq \Big| \frac{T_{n, j}(x) }{k p_{n, j}(x)} - 1 \Big| + \Big| \frac{\tilde p_{j}(x)}{p_{n, j}(x)} - 1 \Big|
	 \to 0\,,\label{eq:conv_functions_psi_j_bis}
	\end{align}
	uniformly over $x \in [\frac{1}{1+\tau}, 1 +\tau]$. The uniform convergence \eqref{eq:conv_functions_psi_j} and \eqref{eq:conv_functions_psi_j_bis} imply the uniform convergence of $A_{n,j}(x)$ to $1$. Similar arguments provide the uniform convergence of $B_n(x)$ to $1$.
	
	This proves the weak convergence in $\ell^\infty([\frac{1}{1+\tau},1+\tau])$ of $\psi_n(x)$ to a constant process following a chi-square distribution with $s$ degrees of freedom.
\end{proof}

Taking $x = u_n / |\boldX_{(k+1)}|$ we obtain
\[
\E[KL(\boldP_k \| \boldM_k )|_{\boldp = \hat \boldp}]
- \E \big[ \log L_{\boldP_k} (\boldT_{n,k}) \big]
+ \E[\log L_{\boldM_k} (\hat \boldp ; \boldT_{n,k})]
=\E[\psi_n(u_n / |\boldX_{(k+1)}|)]\,,
\]
with $\psi_n(u_n / |\boldX_{(k+1)}|) \stackrel{d}{\to} \psi$ as $n \to \infty$, where $\psi$ follows a chi-square distribution with $s$ degrees of freedom.

We now prove that the former convergence holds in expectation. For any $x \in [\frac{1}{1+\tau},1+~\tau]$  the vector $\boldT^{s,r}_n(x)$ satisfies Lindeberg condition after a normalization and therefore is uniformly integrable. In order to prove that $\psi_n(x)$ converges  in expectation it suffices to prove that the variables $A_{n,j}(x)$ and $B_n(x)$ have bounded second order moments.

We deal with the terms in $A_{n,j}(x)$. If $c_{n,2}(x) \geq 1/2$ we write
\[
c_{n,2}(x) \frac{T_{n,j}(x)}{k} + (1-c_{n,2}(x)) \tilde p_{j}(x)
\geq c_{n,2}(x) \frac{T_{n,j}(x)}{k}
\geq \dfrac{1}{2} \dfrac{T_{n,j}(x)}{k}\,,
\]
while if $c_{n,2}(x) < 1/2$ we write
\[
c_{n,2}(x) \frac{T_{n,j}(x)}{k} + (1-c_{n,2}(x)) \tilde p_{j}(x)
\geq (1-c_{n,2}(x)) \tilde p_{j}(x)
\geq \dfrac{1}{2} \tilde p_{j}(x)\,.
\]
This leads to the following bound for $A_{n,j}(x)$:
\begin{align*}
A_{n,j}(x)
&\leq
\dfrac{p_{n,j}(x) T_{n,j}(x)/k}{2[(1/2) T_{n,j}(x) / k]^2} \indic_{\{c_{n,2}(x) \geq 1/2\}}
+ \dfrac{p_{n,j}(x) T_{n,j}(x)/k}{2[(1/2) \tilde p_{j}(x)]^2} \indic_{\{c_{n,2}(x) > 1/2\}}\\
&\hspace{1cm} +\dfrac{p_{n,j}(x)^2}{2[(1/2) T_{n,j}(x) / k]^2} \indic_{\{c_{n,2}(x) \geq 1/2\}}
+\dfrac{p_{n,j}(x)^2}{2[(1/2) \tilde p_{j}(x)]^2} \indic_{\{c_{n,2}(x) > 1/2\}}\\
&\leq \dfrac{2 p_{n,j}(x) }{T_{n,j}(x)/k}
+ \dfrac{2 T_{n,j}(x)/k}{p_{n,j}(x)}
+\dfrac{2 p_{n,j}(x)^2}{(T_{n,j}(x)/k)^2}
+2
=: K_1 + K_2 + K_3 + 2\,.
\end{align*}
where we used that
\[
\dfrac{p_{n,j}}{\tilde p_{j}}
= \dfrac{p_{n,j}}{\frac{p_{n,j}}{\sum_{j=1}^r p_{n,j}}}
= \sum_{j=1}^r p_{n,j}
\leq 1\,.
\]
We use the bound established by \cite{cribari_et_al_2000} which gives for any integer $l$,
\[
\E\Big[ \Big(\dfrac{2 p_{n,j}(x) }{T_{n,j}(x)/k}\Big)^l \Big]
= 2 k^l p_{n,j}(x)^l  \E\Big[ \dfrac{1}{T_{n,j}(x)^l} \Big]
\leq
C_1 \dfrac{k^l p_{n,j}(x)^l}{n^l \P( \boldX/u_n \in A(x, \beta_j))^l}
= C_1 \dfrac{k^l}{a^l_n(x)}\,,
\]
where $C_1$ is a constant and where we recall that $a_n(x) = n \P(x | \boldX| > u_n)$. Karamata's theory on regular variation ensures that the ratio $k/a_n(x)$ converges uniformly to $x^\alpha$, and thus that it is bounded. Using this bound with $l=1$ and $l=2$ entails that $\E[K_1^2]$ and $\E[K_3^2]$ are bounded.

Regarding the term $K_2$, its second moment satisfies
\begin{align*}
\E[K_2^2]
&= 
\dfrac{4 \E[T_{n,j}(x)^2]}{k^2 p_{n,j}(x)^2}\\
&=
4 \dfrac{n \P(\boldX/u_n \in A(x, \beta_j))(1-\P(\boldX/u_n \in A(x, \beta_j))) + n^2 \P(\boldX/u_n \in A(x, \beta_j))^2}{k^2 p_{n,j}(x)^2}\\
& \leq
4 \dfrac{n \P(\boldX/u_n \in A(x, \beta_j)) + n^2 \P(\boldX/u_n \in A(x, \beta_j))^2}{k^2 p_{n,j}(x)^2}\\
& = 4 \dfrac{a_n(x) p_{n,j}(x) + a_n(x)^2 p_{n,j}(x)^2}{k^2 p_{n,j}(x)^2}\\
& = 4 \Big( \dfrac{a_n(x)}{k^2 p_{n,j}(x)} + \dfrac{a_n(x)^2}{k^2} \Big)\,.
\end{align*}
We know that $k p_{n,j}(x) \to \infty$ since we assume $\beta_j \in \mathcal{S}_\infty$ and that $a_n(x)/k $ converges uniformly to $x^\alpha$. Using again Karamata's theory on regular variation we obtain that $\E[K_2^2]$ is bounded by a constant $C_2$.

Therefore, we have proved that $\E[|A_{n,j}(x)|^2]$ is bounded. Similar arguments ensure that $\E[|B_{n}(x)|^2]$ is also bounded. This holds for any $x \in [\frac{1}{1+\tau}, 1 +\tau]$ and therefore we obtain the uniform integrability of $\psi_n(u_n / |\boldX_{(k+1)}|)$ which gives the convergence in expectation
\[
\E[KL(\boldP_k \| \boldM_k )|_{\boldp = \hat \boldp}]
- \E \big[ \log L_{\boldP_k} (\boldT_{n,k}) \big]
+ \E[\log L_{\boldM_k} (\hat \boldp ; \boldT_{n,k})]
\to s\,,
\quad n \to \infty\,.
\]
This concludes the proof of Theorem \ref{main-theo:cv_KL_chi_square}.

\subsection{Level selection}\label{app:threshold_selection}

We recall the assumptions \ref{main-ass:bias_T}, \ref{main-ass:bias_k}, and \ref{main-ass:bias_nb_clusters_vs_nb_extremes} of the main text.
\begin{enumerate}[label=\textbf{(B\arabic*)}]
	\item\label{SM-ass:bias_T} For $k \in K$ and $\beta_j \in \mathcal{S}_\infty$ we have
	\[
	\dfrac{\E[T_{n,n-T'_{n, 2^d}, j} \mid T'_{n, 2^d}]}{n - T'_{n, 2^d}}
	= \dfrac{\E[T_{n,k, j}]}{k} + O(1)\,, \quad n \to \infty\,.
	\]
	\item\label{SM-ass:bias_k} For $n$ sufficiently large, $k \in K$, there exist $c$, $C>0$ such that
	\[
	c n q_n \leq k \leq C n q_n\,.
	\]
\end{enumerate}

The likelihood $L_{\boldM'_n} (\boldp' ; \boldT'_n)$ of the model $\boldM'_n$ is given by
\[
\dfrac{n!}{\prod_{j=1}^{2^d} T'_{n, j}!} \Big[ \prod_{j=1}^{s'} (p'_j q')^{T'_{n, j}} \Big] \Big[ \prod_{j=s' + 1}^{r'} (p' q')^{T'_{n, j}} \Big] (1 - q')^{T'_{n, 2^d}} \,.
\]

The likelihood is maximal when $ r'=\hat s_n$ and when the $(T'_{n, j})_{1\le j\le 2^d-1}$ are ordered. We will work under these conditions hereafter without changing the notation for simplicity. Using the relation $\sum_{j=1}^{r'} T'_{n,j} = n-T'_{n, 2^d}$ we obtain the following expression for the log-likelihood $\log L_{\boldM'_n}$:
\begin{align}
	\log L_{\boldM'_n} (\boldp' ; \boldT'_n)
	&= \log(n!) - \sum_{j=1}^{2^d} \log(T'_{n, j}!)
	+ \sum_{j=1}^{s'} T'_{n, j} \log(p'_j)
	+ \Big(\sum_{j=s' + 1}^{r'} T'_{n, j} \Big) \log(p') \nonumber\\
	 &+ T'_{n, 2^d} \log(1 - q') + (n-T'_{n, 2^d}) \log(q')\nonumber\\
	&=: \log L_{\boldM_{n- T'_{n, 2^d}}}(\boldp, \boldT_{n, n-T'_{n, 2^d}})
	+ \log L_{\boldB_n}(q', T'_{n, 2^d}) \label{eq:log_likelihood_M'_n_conditional}\,,
\end{align}
introducing the likelihood of a binomial model
\[
L_{\boldB_n}(q' ; T'_{n, 2^d}) = \dbinom{n}{T'_{n, 2^d}} {q'}^{n-T'_{n, 2^d}} (1-q')^{T'_{n, 2^d}}\,.
\]

%
Similarly to Akaike's procedure our aim is to study the expectation of the Kullback-Leibler divergence 
\[
\E[KL(\boldM'_n \| \boldT'_n )|_{\widehat{\boldp'}}]
= \E[\log L_{\boldP'_{n}}(T'_n)]
+ \E \bigg[ 
\E[-\log L_{\boldM'_{n}}(\boldp', T'_n)] \mid_{\boldp' = \widehat{\boldp'}}
\bigg]\,,
\]
where we recall that $\widehat{\boldp'}$ denotes the maximum likelihood estimator of $\boldp'$.
This encourages us to look at
\begin{align*}
	\E \bigg[ 
	\E[-&\log L_{\boldM'_{n}}(\boldp', T'_n)] \mid_{\boldp' = \widehat{\boldp'}}
	\bigg]\\
	= \E& \bigg[ 
	\E[- \log L_{\boldM_{n- T'_{n, 2^d}}}(\boldp, \boldT_{n, n-T'_{n, 2^d}}) \mid T'_{n, 2^d}]\mid_{\boldp' = \widehat{\boldp'}} \bigg] + \E \bigg[ \E[ -\log L_{\boldB_n}(q', T'_{n, 2^d}) ]\mid_{q' = \hat{q}} \bigg]\\
	= E&_1 + E_2\,,
\end{align*}
where we condition with respect to $T'_{n, 2^d}$ and use the decomposition \eqref{eq:log_likelihood_M'_n_conditional}.

For the term $E_1$, we write
\begin{align}
	\E[\log L_{\boldM_{n- T'_{n, 2^d}}}&(\boldp, \boldT_{n, n-T'_{n, 2^d}}) \mid T'_{n, 2^d}] \label{eq:exp_log_lik}\\
	&= \log((n- T'_{n, 2^d})!) - \sum_{j=1}^{2^d-1} \E[\log(T_{n,n-T'_{n, 2^d}, j}!) \mid T_{n, 2^d}'] \nonumber\\
	&+ \sum_{j=1}^{s'} \E[T_{n,n-T'_{n, 2^d}, j} \mid T_{n, 2^d}'] \log(p_j)
	+ \sum_{j=s' + 1}^{r'} \E[T_{n, n-T'_{n, 2^d},j} \mid T_{n, 2^d}'] \log(p) \nonumber\,.
\end{align}
and the following lemma provides an asymptotic expansion of this term.

\begin{lem}
	Under Assumption \ref{SM-ass:bias_T} we have
	\begin{align*}
		\E[\log L_{\boldM_{n- T'_{n, 2^d}}}(&\boldp, \boldT_{n, n-T'_{n, 2^d}}) \mid T'_{n, 2^d}]\\
		&= \log((n- T'_{n, 2^d})!)
		+ \dfrac{n - T'_{n, 2^d}}{k} \E[\log L_{\boldM_k}(\boldp, \boldT_{n, k})]\\
		&+ (n - T'_{n, 2^d})(\log(n - T'_{n, 2^d})+1)- 2 (n - T'_{n, 2^d}) \log(k)+O\big(n-T'_{n, 2^d}\big)\,.
	\end{align*}
\end{lem}

\begin{proof}
	The assumption implies that the last two terms of the right-hand side in \eqref{eq:exp_log_lik} satisfy
	\[
	\sum_{j=1}^{s'} \E[T_{n,n-T'_{n, 2^d}, j} \mid T_{n, 2^d}'] \log(p_j)
	= \dfrac{n - T'_{n, 2^d}}{k} \sum_{j=1}^{s'}  \E[T_{n,k, j}] \log(p_j)+O\big(n - T'_{n, 2^d}\big)
	\]
	and
	\[
	\sum_{j=s' + 1}^{r'} \E[T_{n, n-T'_{n, 2^d},j} \mid T_{n, 2^d}'] \log(p)
	= \dfrac{n - T'_{n, 2^d}}{k} \sum_{j=s' + 1}^{r'} \E[T_{n, k, j}] \log(p)+O\big(n - T'_{n, 2^d}\big)\,.
	\]
	In order to deal with the second term in in \eqref{eq:exp_log_lik} we recall the following Taylor expansions:
	\begin{align*}
	&\log(m!) = m \log(m) - m + O(\log m)\,, \\
	&\E[X \log(X)] = \E[X] \log(\E[X]) + O(1)\,,
	\end{align*}
	where the second one holds for any random variable $X$ following a binomial distribution with parameters $n$ and $p_n$ satisfying $n p_n \to \infty$. We remind that $T_{n,n-T'_{n, 2^d},j} > 0$ for all $1 \leq j \leq r' = r =\hat s_n$ which ensures that the previous expansions can be applied with $m=T_{n,n-T'_{n, 2^d},j} $ or $X=T_{n,n-T'_{n, 2^d},j} $. Using that $n p_{n,j} \to \infty$ and therefore that $\E[T_{n,n-T'_{n, 2^d}, j}^{-1}] = o(1)$ by dominated convergence, we obtain
	\begin{align*}
		\E[\log(&T_{n,n-T'_{n, 2^d}, j}!) \mid T_{n, 2^d}']\\
		&= \E[ T_{n,n-T'_{n, 2^d}, j} \mid T_{n, 2^d}'] \log \E[T_{n,n-T'_{n, 2^d}, j} \mid T_{n, 2^d}'] - \E[T_{n, n-T'_{n, 2^d},j} \mid T_{n, 2^d}']+ O(1)\\
		&= \dfrac{n - T'_{n, 2^d}}{k} \E[T_{n,k, j}] \log\Big( \frac{n - T'_{n, 2^d}}{k} \E[T_{n,k, j}]\Big)
		- \dfrac{n - T'_{n, 2^d}}{k} \E[T_{n,k, j}] +O(1)\,.
	\end{align*}
for any fixed $1 \leq j \leq r$, where we used \ref{SM-ass:bias_T}. Using the same approximation for $\E[\log(T_{n,k, j}!)]$ we get
	\begin{align*}
		\E[\log(&T_{n,n-T'_{n, 2^d}, j}!) \mid T_{n, 2^d}']\\
		&= \dfrac{n - T'_{n, 2^d}}{k} \E[\log(T_{n,k, j}!)]
		+ \dfrac{n - T'_{n, 2^d}}{k} \E[T_{n,k, j}] \log\Big(\frac{n - T'_{n, 2^d}}{k}\Big) + O\big(n-T'_{n, 2^d}\big)\,.
	\end{align*}
Summing these relations for $j=1, \ldots, 2^d-1$ such that $T_{n,k,j} > 0$ we obtain the following expression for \eqref{eq:exp_log_lik}:
\begin{align*}
	\E[\log L_{\boldM_{n- T'_{n, 2^d}}}(&\boldp, \boldT_{n, n-T'_{n, 2^d}}) \mid T'_{n, 2^d}]\\
	&= \log((n- T'_{n, 2^d})!)
	+ \dfrac{n - T'_{n, 2^d}}{k} \E[\log L_{\boldM_k}(\boldp, \boldT_{n, k})]\\
	&- \dfrac{n - T'_{n, 2^d}}{k} \log(k!) + \dfrac{n - T'_{n, 2^d}}{k} \sum_{j=1}^{2^d-1} \E[T_{n,k, j}] \log(\frac{n - T'_{n, 2^d}}{k}) + O\big(n-T'_{n, 2^d}\big)\,.
\end{align*}
The relation $\sum_{j=1}^{2^d-1} \E[T_{n,k, j}] = k$ yields
\begin{align*}
		\E[\log L_{\boldM_{n- T'_{n, 2^d}}}(&\boldp, \boldT_{n, n-T'_{n, 2^d}}) \mid T'_{n, 2^d}]\\
		&= \log((n- T'_{n, 2^d})!)
		+ \dfrac{n - T'_{n, 2^d}}{k} \E[\log L_{\boldM_k}(\boldp, \boldT_{n, k})]\\
		&- \dfrac{n - T'_{n, 2^d}}{k} \log(k!) + (n - T'_{n, 2^d}) \log(\frac{n - T'_{n, 2^d}}{k}) + O\big(n-T'_{n, 2^d}\big)\,.
	\end{align*}
Using Stirling's formula as $k$ and $n-T'_{n, 2^d}$ diverge (a.s. for the second) to infinity with $n$, we obtain 
\begin{align*}
&\log((n- T'_{n, 2^d})!)- \dfrac{n - T'_{n, 2^d}}{k} \log(k!) + (n - T'_{n, 2^d}) \log(\frac{n - T'_{n, 2^d}}{k})\\
&= \log((n- T'_{n, 2^d})!)+ (n - T'_{n, 2^d})(\log(n - T'_{n, 2^d})+1)- 2 (n - T'_{n, 2^d}) \log(k)+O\big(n-T'_{n, 2^d}\big)\,.
\end{align*}
\end{proof}

This lemma proves that the quantity $E_1$ is equal to 
\begin{align*}
E_1
= \dfrac{nq_n}{k} \E \Big[ \E[-\log L_{\boldM_k}(\boldp, \boldT_{n,k})]|_{\boldp = \hat \boldp} \Big]
&- \E[\log(n - T'_{n, 2^d})!]\\
& - nq_n(\log(nq_n)+1)+ 2 nq_n\log(k)+O(nq_n)\,.
\end{align*}

Moving on to term $E_2$, we apply similar arguments than in the proof of Theorem 2.

\begin{lem} For $k \in K$ such that \ref{SM-ass:bias_k} holds we have 
\[
\E[ \log L_{\boldB_n}(nq_n, T'_{n, 2^d})]=- \E[\log L_{\boldB_n}(k/n, T'_{n, 2^d})]
+ \Big(\dfrac{k}{n} - q_n\Big)^2 \dfrac{n}{k/n} + O(k)\,.
\]
\end{lem}
\begin{proof}
In order to apply Lemmas \ref{lem:asymptotic_expansion_hat_p} and \ref{lem:asymptotic_expansion_p*}  we introduce artificially the Kullback-Leibler divergence
\[
KL(D \| \boldB_n)
= \E[\log L_D] - \E[\log L_{\boldB_n}(q' ; T'_{n, 2^d})]\,,
\]
where $L_D$ corresponds to the likelihood of the distribution of the data, that is the distribution of $n-T'_{n, 2^d}$. The derivatives of $\log L_{\boldB_n(q',T'_{n, 2^d})}$ are
\[
\dfrac{\partial}{\partial q'} \log L_{\boldB_n}(q' ; T'_{n, 2^d})
= \dfrac{n-T'_{n, 2^d}}{q'} - \dfrac{T'_{n, 2^d}}{1-q'}\,,
\]
and
\[
\dfrac{\partial^2}{\partial {q'}^2} \log L_{\boldB_n}(q' ; T'_{n, 2^d})
= -\dfrac{n-T'_{n, 2^d}}{{q'}^2} - \dfrac{T'_{n, 2^d}}{(1-q')^2}\,.
\]

We adapt Lemmas \ref{lem:asymptotic_expansion_hat_p} and \ref{lem:asymptotic_expansion_p*} in this context. Lemma \ref{lem:asymptotic_expansion_hat_p} applied with $KL(D \| \boldB_n(q' ; T'_{n, 2^d}))$ between $q' = \widehat{q'}$ and $q' = q_n = \E[T'_{n, 2^d}]/n$ gives
\begin{align*}
KL(D \| \boldB_n(q' ; T'_{n, 2^d})) \mid_{q' = \widehat{q'}}
= KL&(D \| \boldB_n(q' ; T'_{n, 2^d}))\mid_{q' = q_n}\\
&+ \dfrac{1}{2} \Big(\dfrac{k}{n} - q_n\Big)^2 \dfrac{\partial^2}{\partial {q'}^2} KL(D \| \boldB_n(q' ; T'_{n, 2^d}))  \mid_{q = c_n \widehat{q'} + (1-c_n) q_n}\,,
\end{align*}
which gives
\begin{align*}
	KL(D \| \boldB_n(q' ; T'_{n, 2^d})) \mid_{q' = \widehat{q'}}
	= KL(D \| \boldB_n(q' ; T'_{n, 2^d})) \mid_{q' = q_n}
	+ \dfrac{1}{2} (\widehat{q'} - q_n)^2 \kappa_n(c_n)\,,
\end{align*}
with
\[
\kappa_n(c)
= \Big(\dfrac{n q_n}{(c \widehat{q'} + (1-c) q_n)^2} + \dfrac{n (1-q_n)}{[1-(c \widehat{q'} + (1-c) q_n)]^2} \Big)\,, \quad c \in (0,1)\,.
\]
We now mimic Lemma \ref{lem:asymptotic_expansion_p*} with $\log L_{\boldB_n}(q' ; T'_{n, 2^d})$ between $q' = k/n$ and $q'= q_n$ gives
\begin{align*}
	\log L_{\boldB_n}(q_n; T'_{n, 2^d})
	= \log L_{\boldB_n}(k/n ; T'_{n, 2^d})
	+ \Big(\dfrac{k}{n} - q_n\Big) \dfrac{\partial}{\partial q'} \log &L_{\boldB_n}(k/n ; T'_{n, 2^d})\\
	&+ O\Big(\Big(\dfrac{k}{n} - q_n\Big)^2\Big)\,,
\end{align*}
and the first order term appears since the optimum of the log-likelihood is not $k/n$. We obtain
\begin{align*}
	\E[\log L_{\boldB_n}(q_n ; T'_{n, 2^d})]
	&= \E[\log L_{\boldB_n}(k/n ; T'_{n, 2^d})]
	+ \Big(\dfrac{k}{n} - q_n\Big) \Big( \dfrac{n q_n}{k/n} - \dfrac{n (1-q_n)}{1-k/n} \Big)\,.
\end{align*}

Gathering together both results leads to
\begin{align*}
	&KL(D \| \boldB_n(q' ; T'_{n, 2^d})) \mid_{q' = \widehat{q'}}\\
	&= KL(D \| \boldB_n(q' ; T'_{n, 2^d})) \mid_{q' = q_n} + \dfrac{1}{2} (\widehat{q'} - q_n)^2 \kappa_n(c_n)\\
	&= \E[\log L_D] - \E[\log L_{\boldB_n}(k/n ; T'_{n, 2^d})]
	- \Big(\dfrac{k}{n} - q_n\Big) \Big( \dfrac{n q_n}{k/n} - \dfrac{n (1-q_n)}{1-k/n} \Big) + O\Big(\Big(\dfrac{k}{n} - q_n\Big)^2\Big)\,.
\end{align*}
The second order term is of smaller order than the first order one since $k/n$ converge to zero and so does $q_n$ by assumption \ref{SM-ass:bias_k}. So all in all the variations with respect to $k$ of the KL are the same as the ones of
\begin{align*}
- \E[\log L_{\boldB_n}(k/n ; T'_{n, 2^d})]
&- \Big(\dfrac{k}{n} - q_n\Big) \Big( \dfrac{n q_n}{k/n} - \dfrac{n (1-q_n)}{1-k/n} \Big)\\
&=- \E[\log L_{\boldB_n(q',T'_{n, 2^d})}(k/n)]
+ \Big(\dfrac{k}{n} - q_n\Big)^2 \dfrac{n}{k/n} + O(k)\,,
\end{align*}
after using the asymptotic expansion
\[
\dfrac{n q_n}{k/n} - \dfrac{n (1-q_n)}{1-k/n} 
= \dfrac{n q_n}{k/n} - n +O(k)
= \Big(q_n - \dfrac{k}{n}\Big) \dfrac{n}{k/n} + O(k)\,,
\]
as $q_n$ and $k/n$ converge to $0$. This gives the desired result.
\end{proof}

Going back to the whole model, we get
\begin{align*}
	\E& \bigg[ 
	\E[-\log L_{\boldM'_{n}}(\boldp', T'_n)] \mid_{\boldp' = \widehat{\boldp'}}
	\bigg]
	= E_1 + E_2\\
	&= \dfrac{nq_n}{k} \E \Big[ \E[-\log L_{\boldM_k}(\boldp, \boldT_{n,k})]|_{\boldp = \hat \boldp} \Big]
	- \E[\log(n - T'_{n, 2^d})!]
	- nq_n(\log(nq_n)+1)+ 2 nq_n\log(k)\\
	&- \E[ \log L_{\boldB_n}(k/n ; T'_{n, 2^d})]
	+ \Big(\dfrac{k}{n} - q_n\Big)^2 \dfrac{n}{k/n}+ O(k)\,,
\end{align*}
where we recall that
\begin{align*}
- \E[ \log L_{\boldB_n}(k/n ; T'_{n, 2^d})]
&= -\log(n!) + \E[\log(n-T'_{n, 2^d})!]\\
&+ \E[\log(T'_{n, 2^d}!)] - nq_n \log(k/n) - n(1-q_n) \log(1-k/n)\,.
\end{align*}
The terms $\E[\log(n - T'_{n, 2^d})!]$  vanish and the terms depending on $k$ which remain are
\[
n \Big[ q_n\log(k/n) - (1-q_n) \log(1-k/n) + \Big(\dfrac{k}{n} - q_n\Big)^2 \dfrac{n}{k} \Big]\,.
\]
Under Assumption \ref{SM-ass:bias_k} the previous terms can then be simplified as
\begin{align*}
q_n\log(k/n) - (1-q_n) \log(1-k/n) + \Big(\dfrac{k}{n} - q_n\Big)^2 \dfrac{n}{k}
&= q_n\log(k/n) + \dfrac{k}{n} + \dfrac{k}{n} - 2 q_n + q_n^2 \dfrac{n}{k}\\
&= q_n\log(k/n) + O(q_n)\,.
\end{align*}
All in all we have proved that
\begin{align*}
\E \bigg[ 
\E[-\log L_{\boldM'_{n}}(\boldp', T'_n)] \big|_{\boldp' = \widehat{\boldp'}}
\bigg]
= \dfrac{nq_n}{k} \E \Big[ \E[-\log L_{\boldM_k}(\boldp, \boldT_{n,k})]|_{\boldp = \hat \boldp} \Big]
+ n&q_n \log(k/n)\\
&+ O(nq_n)\,.
\end{align*}
\section{Numerical examples with asymptotic independence} \label{subsec:asymptotic_indep}

The example we develop here is related to asymptotic independence. The code can be found at \url{https://drive.google.com/drive/folders/11TvhbVMPXcSkxmdnnAySvZt64lpMKZqL?usp=sharing}.

Let $\boldX = (X_1, \ldots, X_d)^\top$ be a random vector with a Gaussian copula with a common correlation parameter $\rho <1$ and marginal distributions satisfying $\P(X_j > x) = x^{-1}$. Then $\boldX$ is regularly varying with tail index $-1$ and its marginals are asymptotically independent, see \cite{resnick_87}, Corollary 5.28. The spectral measure only places mass on the subsets $C_\beta$ such that $|\beta|=1$, and so does the distribution of $\boldZ$, see Example 1 in the main text. The aim of our procedure is then to recover these $d$ directions among the $2^d-1$ clusters. These directions are all the more identifiable as the parameter $\rho$ is close to $0$ (no dependence). 

We first consider $d=40$, a sample size $n=30\,000$, and a correlation parameter $\rho =0.5$. We plot the evolution of the penalized log-likelihood for a given sample $\boldX_1, \ldots, \boldX_n$. Figure \ref{fig:evolution_s_k_asympt_indep} shows that this quantity first decreases for small values of $k$ before it slightly increases with $k$. The minimum is reached for $\hat k = 1050$ which corresponds to a proportion of extreme values of $\hat k / n = 3.5\%$. Regarding the evolution of $\hat s(k)$, we observe that the value of $\hat s(k)$ remains constant for $k$ close to $\hat k$. It stabilizes around an optimal value of $\hat s(\hat k) = 41$. Recall that in this example the true clusters corresponds to the $d=40$ one-dimensional ones. In turns out that the algorithm identifies the $40$ one-dimensional clusters. The extra cluster which appears is $\{1, \ldots, d\}$.

\begin{figure}[!th]
	\includegraphics[scale=0.27]{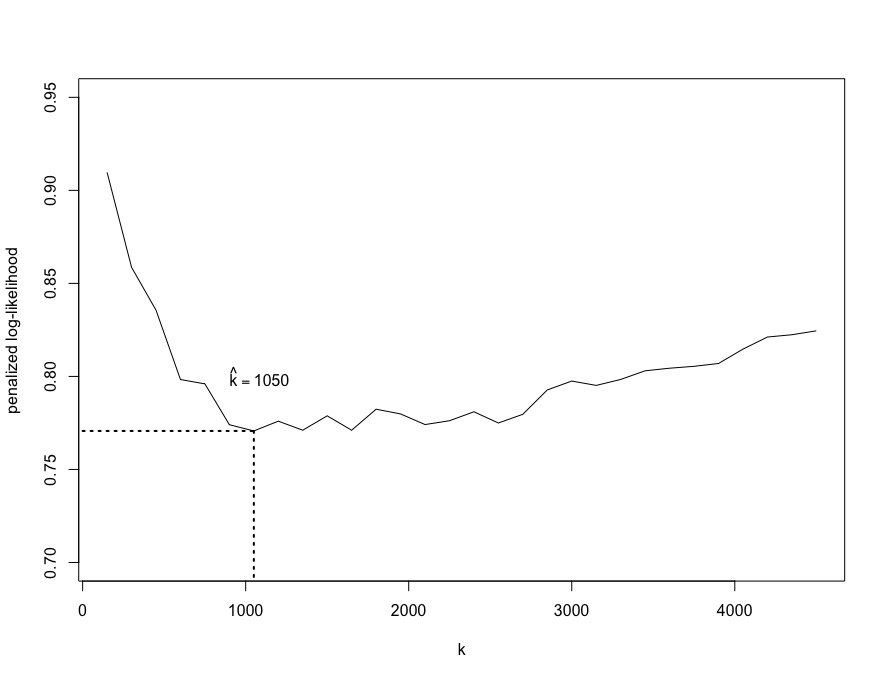}\hfill
	\includegraphics[scale=0.27]{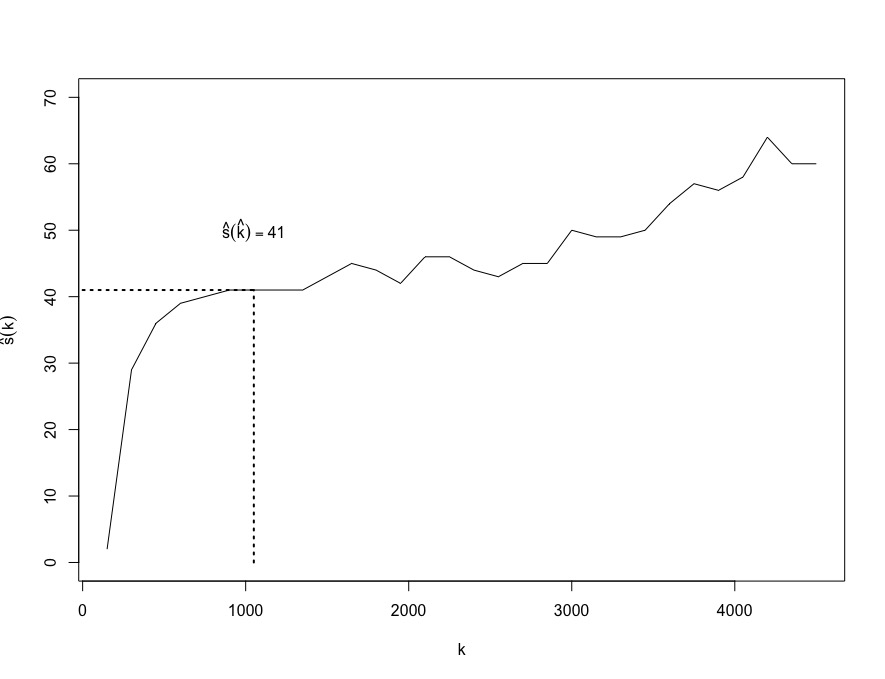}
	\captionof{figure}{Evolution of the penalized log-likelihood (left) and of $\hat s(k)$ (right) with respect to $k$. Here $\rho = 0.5$ and $n=30\,000$. \label{fig:evolution_s_k_asympt_indep}}
\end{figure}

We then compare the estimated probabilities $\hat \boldzeta$ given by our algorithm with the true ones $p^*(\beta) = 1/d$ for $|\beta| = 1$ and zero elsewhere. To do so, we still consider $d=40$ and $n=30\,000$. We apply our procedure for different correlation parameters $\rho \in \{0, 0.25, 0.5, 0.75\}$ and repeat our procedure over $N=100$ simulations. Then we compute the Hellinger distance and we compare ourselves with the approach of \cite{goix_sabourin_clemencon_17}. For the two methods proposed by \cite{simpson_et_al} it is necessary to compute the empirical mass on all $2^d-1$ subsets $C_\beta$ which can not be achieved for such a high dimension. This is why for this example we restrict our comparison with DAMEX.

Figure \ref{fig:boxplot_asympt_indep} shows the mean Hellinger distance achieved by our method and the one of \cite{goix_sabourin_clemencon_17} over $100$ simulations. We observe that the performance of both methods deteriorates as the value of the parameter $\rho$ increases. For any $\rho \in \{0, 0.25, 0.5, 0.75\}$ our approach leads to better results than the one of \cite{goix_sabourin_clemencon_17}.

\begin{figure}[!th]
	\centering
	\includegraphics[scale=0.15]{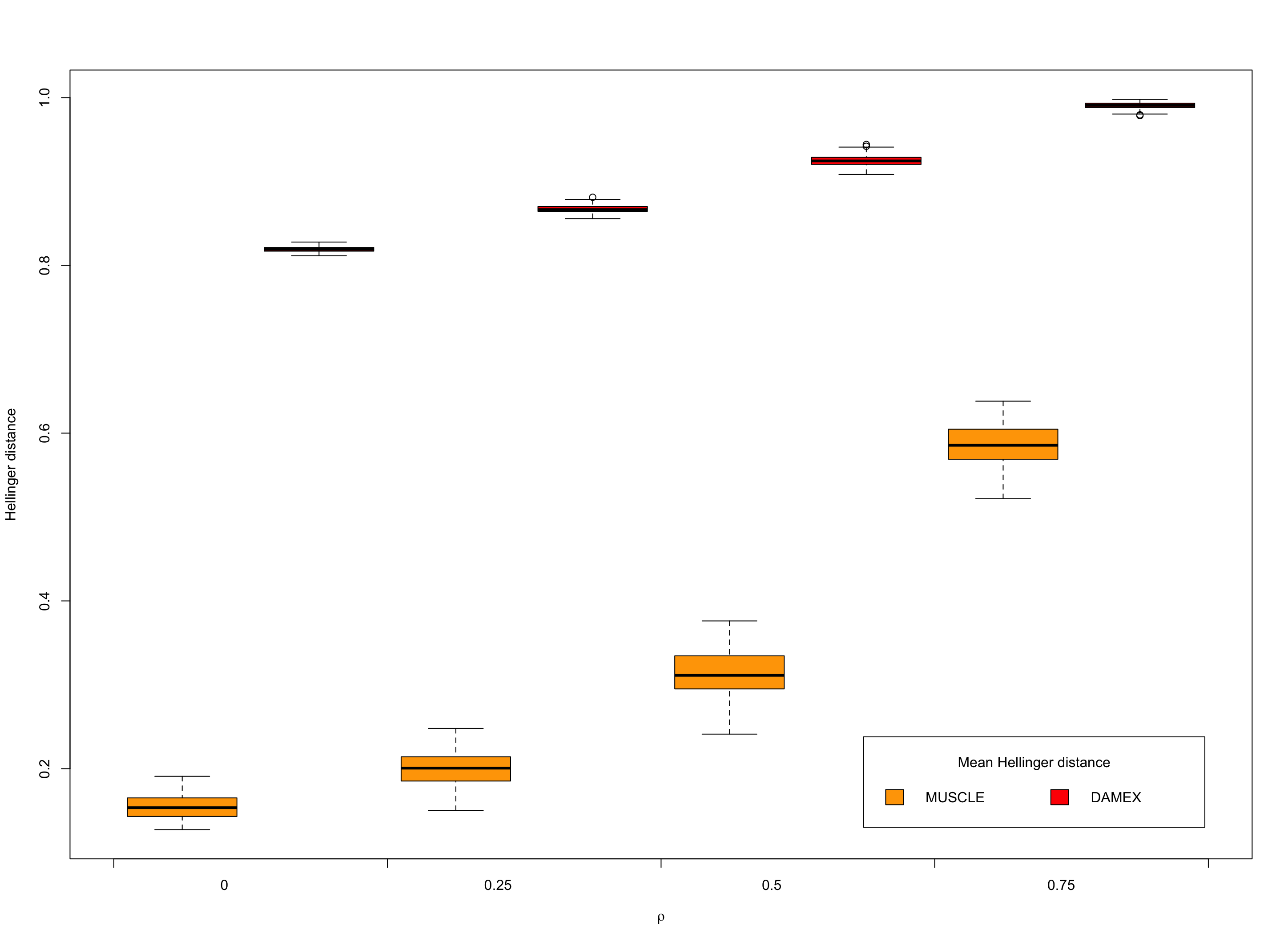}
	\caption{Mean Hellinger distance for $\rho \in \{0, 0.25, 0.5, 0.75\}$ over 100 simulations. In orange: MUSCLE. In red: \cite{goix_sabourin_clemencon_17}. \label{fig:boxplot_asympt_indep}}
\end{figure}

\section{Application to wind speed data}

The data correspond to the daily-average wind speed for 1961-1978 at 12 synoptic meteorological stations in the Republic of Ireland ($n=6574$, $d=12$). They are available at \url{http://lib.stat.cmu.edu/datasets/} and have been analyzed in detail by \cite{haslett_raftery_1989}. The stations are the following ones: Malin Head (Mal), Belmullet (Bel), Clones (Clo), Claremorris (Cla), Mullingar (Mul), Dublin (Dub), Shannon (Sha), Birr (Bir), Kilkenny (Kil), Valentia (Val), Roche's Pt. (Rpt), Rosslare (Ros). Seven of these stations are along the sea: Belmullet (west), Dublin (east), Malin Head (north), Roche's Pt. (south), Rosslare (east), Shannon (west), and Valentia (southwest). The five other stations are more than 50 kilometers away from the coast. We refer to \cite{haslett_raftery_1989} for a map of the stations.

The preprocessing of the data provides a Hill estimator of $\hat \alpha = 10.7$. Before applying MUSCLE we plot the evolution of the penalized log-likelihood as a function of $k$. The optimal value $\hat k = 460$ is clearly identified and corresponds to a proportion $\hat k / n = 7 \%$. This choice of $\hat k$ leads to a number of clusters $\hat s(\hat k)= 11$. Note that contrary to the numerical examples the value of $\hat s(k)$ does not stabilize when $k$ is close to $\hat k$.

\begin{figure}[!th]
	\centering
	\includegraphics[scale=0.17]{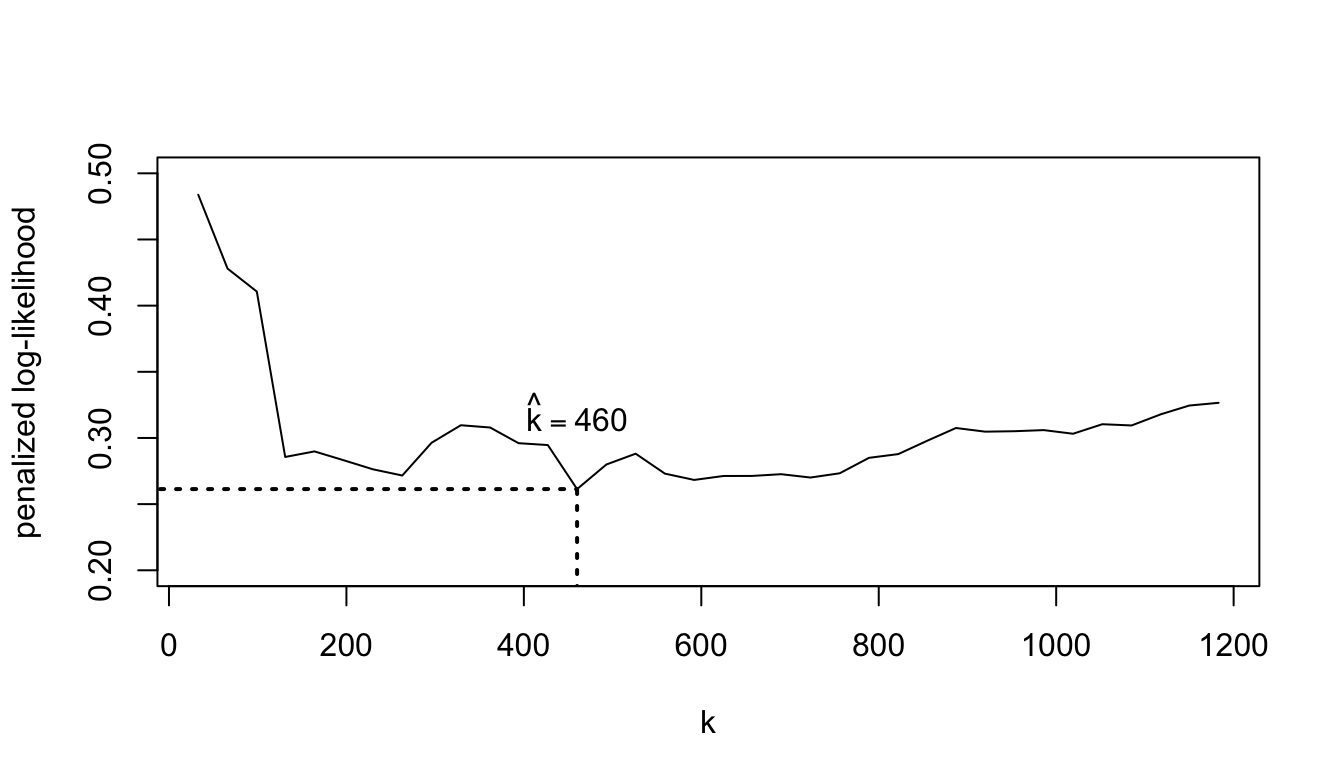}\hfill
	\includegraphics[scale=0.17]{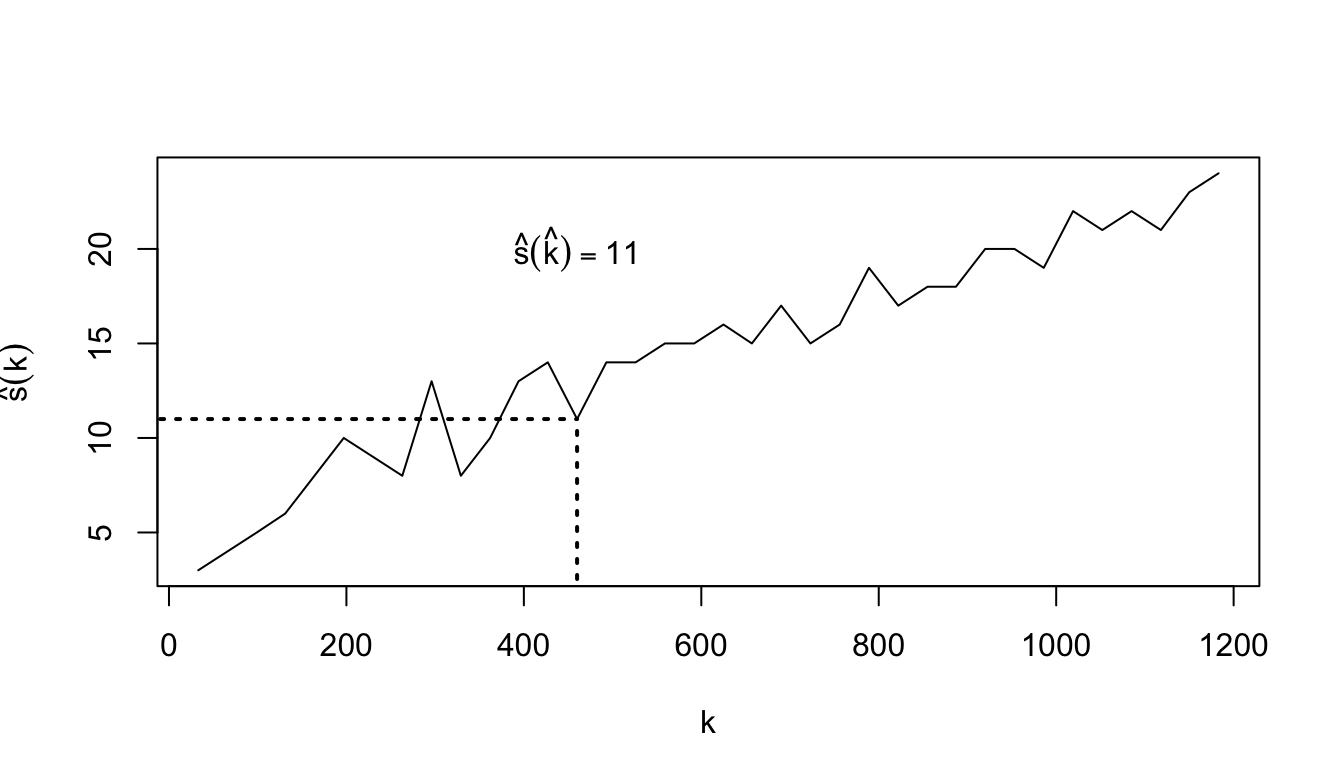}
	\caption{Evolution of the penalized log-likelihood (left) and of $\hat s_n(k)$ (right) with respect to $k$ for the wind speed data. \label{fig:evolution_winspeed}}
\end{figure}

MUSCLE provides $11$ extremal clusters which correspond to $6$ stations: Belmullet, Malin Head, Roche's Pt., Rosslare, Dublin, and Shannon. All of these stations are located close to the sea, where wind speed is likely to be higher than in inland cities. The only coastal city which does not appear in the extremal clusters is Valentia which is located more than 200 kilometers away from the other stations.

The $11$ clusters and their inclusions are illustrated on Figure \ref{fig:features_wind_data}. The algorithm exhibits low-dimensional clusters as the largest ones are of dimension 3. The northern station Malin Head appears in all multivariate clusters. Most of the clusters are related to a specific localization: $\{\text{Sha, Bel, Mal}\}$ and $\{\text{Bel, Mal}\}$ correspond to stations in the north/west, $\{\text{Mal, Dub}\}$ and $\{\text{Mal, Ros}\}$ to stations in the north/east. Three extremal clusters gather the northern station Malin Head and the southern one Roche's Pt.

We conclude that the aforementioned $11$ clusters correspond to subsets $C_\beta$ which gather the mass of the angular vector $\boldZ$. In particular, the subsets related to the clusters $\{\text{Sha, Bel, Mal}\}$,  $\{\text{Rpt, Bel, Mal}\}$, $\{\text{Rpt, Ros, Mal}\}$, and $\{\text{Dub, Mal}\}$ gather some mass of $\boldZ$ and are not included in larger subsets on which $\boldZ$ places mass. Following \cite{meyer_wintenberger}, Theorem 2, these \emph{maximal} subsets also concentrate the mass of the spectral measure. The remaining clusters, which correspond to \emph{non-extremal} subsets, contain almost all the station Malin Head. We interpret this as follows: among the maximal subsets the wind speed in Malin Head is likely to be larger than in the other stations. We also refer to \cite{meyer_wintenberger}, Section 3.2, for a discussion on maximal and non-maximal subsets. A separate study can then be conducted on each group of stations for which standard methods for low-dimensional extremes can be applied, see \cite{coles_tawn_91}, \cite{einmahl_dehaan_huang_93}, \cite{einmahl_97}, \cite{einmahl_segers_09}.

\begin{figure}[!th]
	\centering
	\begin{tikzpicture}
	\tikzstyle{group}=[ellipse, draw]
	
	\node[group] (ShaBelMal) at (2,5) {Sha, Bel, Mal};
	\node[group] (RptBelMal) at (8,5) {Rpt, Bel, Mal};
	\node[group] (RptRosMal) at (14,5) {Rpt, Ros, Mal};
	
	\node[group] (BelMal) at (2,3) {Bel, Mal};
	\node[group] (DubMal) at (6,3) {Dub, Mal};
	\node[group] (RptMal) at (10,3) {Rpt, Mal};
	\node[group] (RosMal) at (14,3) {Ros, Mal};	
	
	\node[group] (Bel) at (1,1) {Bel};
	\node[group] (Mal) at (6,1) {Mal};
	\node[group] (Rpt) at (10,1) {Rpt};
	\node[group] (Ros) at (14,1) {Ros};
	
	\draw (Mal) to (DubMal);
	\draw (Mal) to (BelMal);
	\draw (Mal) to (RptMal);
	\draw (Mal) to (RosMal);
	\draw (Bel) to (BelMal);
	\draw (Rpt) to (RptMal);
	\draw (Ros) to (RosMal);
	
	\draw (BelMal) to (ShaBelMal);
	\draw (BelMal) to (RptBelMal);
	\draw (RptMal) to (RptBelMal);
	\draw (RptMal) to (RptRosMal);
	\draw (RosMal) to (RptRosMal);
	
	\end{tikzpicture}
	\caption{Representation of the $11$ clusters and their inclusions. \label{fig:features_wind_data}}
\end{figure}

In order to study the remaining stations, we remove the $6$ extremal stations and reapply our procedure. MUSCLE then provides $16$ clusters:
\begin{itemize}
	\item Four one-dimensional clusters: Val, Clo, Cla, Mul.
	\item Five two-dimensional clusters: $\{\text{Val, Cla}\}$, $\{\text{Val, Clo}\}$, $\{\text{Val, Mul}\}$, $\{\text{Mul, Clo}\}$, $\{\text{Cla, Clo}\}$.
	\item Two three-dimensional clusters: $\{\text{Val, Cla, Clo}\}$, $\{\text{Val, Mul, Clo}\}$.
	\item and other clusters: $\{\text{Val, Cla, Mul, Clo}\}$, $\{\text{Val, Bir, Cla, Clo}\}$, $\{\text{Val, Bir, Cla, Mul, Clo}\}$,\\ $\{\text{Val, Kil, Bir, Cla, Mul, Clo}\}$.
\end{itemize}
The station Valentia appears in almost all of these clusters. It is the only remaining coastal station, the other ones are inland ones. No particular tail dependence structure appears for these non-extremal stations. In particular, the largest clusters $\{\text{Val, Kil, Bir, Cla, Mul, Clo}\}$ indicates that it is likely that the wind speed in all of these six stations is simultaneously large.

\bibliographystyle{plainnat}